\documentclass[oneside,a4paper,10pt,reqno]{amsart}
\usepackage{amsfonts}
\usepackage{amssymb}
\usepackage{amsmath}
\usepackage[height=9in,a4paper,hmargin={2.5cm,2.5cm}]{geometry}
\usepackage{epsfig}
\usepackage{graphicx}
\usepackage{enumitem}
\usepackage{verbatim}
\usepackage[T1]{fontenc}
\usepackage[utf8]{inputenc}
\usepackage{mathtools}
\usepackage{esint}
\usepackage{graphicx}
\usepackage{mathrsfs}
\usepackage{multirow}
\usepackage{amsthm}
\usepackage[english]{babel}
\usepackage{amsfonts}
\usepackage{listings}
\usepackage{caption}
\usepackage{csquotes}
\usepackage{amsmath}
\usepackage{amssymb}
\usepackage{mathtools}
\usepackage{tabularx}
\usepackage{enumitem}
\usepackage{bm}
\usepackage{tikz}
\tikzset{
	mynode/.style={fill,circle,inner sep=1.5pt,outer sep=0pt}
}
\usepackage{pgfplots} 
\numberwithin{equation}{section}
\pgfplotsset{/pgf/number format/use comma,compat=newest}
\theoremstyle{plain}
\newtheorem{thm}{Theorem}[section]
\newtheorem{cor}[thm]{Corollary}
\newtheorem{lem}[thm]{Lemma}
\newtheorem{prop}[thm]{Proposition}
\theoremstyle{definition}
\newtheorem{defn}[thm]{Definition}
\newtheorem{rem}[thm]{Remark}

\newcommand{\R}{\mathbb{R}}

\newcommand{\N}{\mathbb{N}}

\newcommand{\Sf}{\mathbb{S}}

\newcommand{\dist}{\textnormal{dist}}

\newcommand{\med}{\textnormal{med}}

\newcommand{\diam}{\textnormal{diam}}
\newcommand{\lip}{\textnormal{Lip}}
\newcommand{\supp}{\textnormal{Supp} }
\newcommand{\sbvp}{SBV^p}
\newcommand{\sbv}{SBV}
\newcommand{\gsbv}{GSBV}
\newcommand{\gsbvp}{GSBV^p}

\newcommand{\gammaiso}{\gamma_{\textnormal{iso}}}
\newcommand{\usurf}{\overline{u}_{x_0}^{\textnormal{surf}}}
\newcommand{\ubulk}{\overline{u}_{x_0}^{\textnormal{bulk}}}
\DeclareMathOperator*{\esssup}{ess\,sup}
\DeclareMathOperator*{\essinf}{ess\,inf}
\DeclareMathOperator*{\aplim}{ap-\,lim}
\newcommand{\G}{\mathcal{G}}
\newcommand{\Om}{\Omega}

\usepackage{tikz}
\usetikzlibrary{decorations.pathreplacing}

\definecolor{blue_links}{RGB}{13,0,180} 

\usepackage{hyperref}
\hypersetup{
	colorlinks=true, 
	linktoc=all,     
	linkcolor=blue_links,  
	citecolor=blue_links,
	urlcolor=blue_links,
}

\usepackage{mathtools}

\usepackage{color}

\title[Lower semicontinuity and relaxation for free discontinuity functionals with non-standard growth]{Lower semicontinuity and relaxation for free discontinuity functionals with non-standard growth}

\author[S. Almi]{Stefano Almi}
\address[Stefano Almi]{Dipartimento di Matematica e Applicazioni ``R.~Caccioppoli'',
	Universit\`a di Napoli Federico II, via Cintia, 80126 Napoli, Italy.}
\email{stefano.almi@unina.it}

\author[D. Reggiani]{Dario Reggiani}
\address[Dario Reggiani]{Scuola Superiore Meridionale, Universit\'a di Napoli Federico II
	Largo San Marcellino 10, 80138 Napoli, Italy.}
\email{dario.reggiani@unina.it}

\author[F. Solombrino]{Francesco Solombrino}
\address[Francesco Solombrino]{Dipartimento di Matematica e Applicazioni ``R.~Caccioppoli'',
	Universit\`a di Napoli Federico II, via Cintia, 80126 Napoli, Italy.}
\email{francesco.solombrino@unina.it}

\makeatletter
\@namedef{subjclassname@2020}{%
  \textup{2020} Mathematics Subject Classification}
\makeatother

\begin{document}
	

	\maketitle
	
	\begin{abstract}
		A lower semicontinuity result and a relaxation formula for free discontinuity functionals with non-standard growth in the bulk energy are provided. Our analysis is based on a non-trivial adaptation of the blow-up~\cite{Ambrosiolsc} and of the global method for relaxation~\cite{BFLMrelaxSBVp} to the setting of generalized special function of bounded variation with Orlicz growth. Key tools developed in this paper are an integral representation result and a Poincar\'e inequality under non-standard growth.
	\end{abstract}

	\section{Introduction}
	\label{s:intro}
	 Integral functionals with non-standard growth first appeared in the works of Zhikov~\cite{zikov3, zikov} for modeling composite materials characterized by a strongly anisotropic behavior. The non-standard character of such functionals is typically expressed in terms of a point-dependent integrability of the deformation gradient, which may be captured, in a functional setting, in terms of variable exponents spaces~\cite{czech} or, more in general, in Orlicz type of spaces (see, e.g.,~\cite{Hastobook}). As relevant examples of bulk energies undergoing non-standard growth we report here the so-called {\em variable exponent} and the {\em double-phase} case
	\begin{equation}
	\label{e:intro1}
	|\xi|^{p(x)} \qquad\text{and} \qquad |\xi|^{p} + a(x) |\xi|^{q} \qquad \text{for $(x, \xi) \in \R^{d} \times \R^{d \times m}$,}
	\end{equation}
	for suitable choices of the exponent function $p \colon \R^{d} \to (1, +\infty)$, of the exponents  $1 < p < q <+\infty$, and of the weight function~$a \colon \R^{d} \to [0, +\infty)$. We refer to Section~\ref{subs: examples} for a list of relevant examples in the literature. 
	
	In a Sobolev/Orlicz setting, the study of integral functionals with non-standard growth has attracted an increasing attention in the last decades. Lower semicontinuity, relaxation, and integral representation have been tackled in a number of papers: we mention the works~\cite{MR1972867, MR1632814, MR2021666,  MR1450951}, dealing with the role of convexity and quasiconvexity in the {\em gap problem} for functionals characterized by a $(p, q)$-growth, i.e., different grow rates from above and below. Further results on relaxation in the gap problem have been presented in~\cite{MR2139562}, where some explicit examples of concentration effects for bulk densities of double-phase type~\eqref{e:intro1} are discussed, pointing out the importance of the H\"older regularity of~$a(\cdot)$ in the relaxation procedure. The variable exponent setting has been dealt with in~\cite{CosciaMucci} under the so called log-H\"older continuity assumption on the exponent function~$p(\cdot)$.  Further instances of non-standard growth in partial differential equations may be found in~\cite{MR3102165, MR3073153, MR3865198, BlancoDoublephVarexpo, MR3023424, MR4305170, MR3348928}, while applications ranging from optimal design  to electro-rheological fluids and homogeneization appeared in~\cite{MR4008674, MR4415789, MR1810360, MR1607245, MR1329546}. Let us also mention the somehow related topic of regularity of minimizers of functionals with non-standard growth, which have been thoroughly discussed in~\cite{acerbi2, acerbi3, MR1675954, MR3209686, ELEUTERI2023103815} for the variable exponent, in, e.g.,~\cite{MR3570955, MR3775180,  MR3360738, MR3294408, MR3765553, MR3945612, MR1094446} for the double phase, and in~\cite{MR4397041} in a unified and generalized framework.
			
When dealing with composite materials, it is rather natural to account for failure phenomena, such as fracture, which can not be captured by a mere bulk energy defined on Sobolev or Orlicz spaces. This leads us to the extension of the above non-standard growth functionals to a Free Discontinuity setting, where singularities may appear in the form of jump discontinuities. Besides Materials Science, applications of a variable exponent in the setting of functions of bounded variation already appeared in image reconstruction~\cite{MR2246061} (see also~\cite{MR2391646, MR2971400}), where an intermediate regime between Total Variation and the isotropic diffusion away from the edges was proposed. 

The aim of this paper it to provide a {\em unified framework} for lower semicontinuity and relaxation of functionals of the form
\begin{equation}
\label{e:intro2}
\mathcal{G}(u) := \int_{\Om} f(x, \nabla u) \, d x + \int_{J_{u}} g(x, [u](x), \nu_{u}) \, d \mathcal{H}^{d-1}\,,
\end{equation}
focusing on the role played by the  non-standard growth condition of volume integrand~$f (x, \xi)$. In formula~\eqref{e:intro2}, $\Om \subseteq \R^{d}$ is an open bounded subset of~$\R^{d}$, $u \in GSBV(\Om; \R^{m})$ is a function of {\em generalized special bounded variation} (see~\cite[Section~4.5]{Ambrosio2000FunctionsOB}), $J_{u}$ denotes the jump set of~$u$, $\nu_{u}$ stands for the approximate unit normal to~$u$, and $[u]:= u^{+} - u^{-}$ represents the jump of~$u$, that is, the difference between the traces~$u^{+}$ and~$u^{-}$ of~$u$ on~$J_{u}$, defined according to the orientation of~$\nu_{u}$. For the sole variable exponent case, lower semicontinuity results for~$\mathcal{G}$ have been obtained in~\cite{DCLVlscvariable}, while $\Gamma$-convergence and relaxation issues have been recently considered in~\cite{SSSvariableexpo}. The key assumptions in the mentioned papers are superlinearity of $p(\cdot)$ (meaning $\min_{\Om} p(\cdot) >1$), which allows for a separation of scales in the $\Gamma$-convergence and relaxation processes, and a log-H\"older continuity of the exponent, necessary to avoid the Lavrentiev's phenomenon, as originally observed in~\cite{zikov2}.

Our primary interest is in providing suitable conditions which either imply the lower semicontinuity of the functional~$\G$ in~\eqref{e:intro2} or guarantee an explicit formula for its lower semicontinuous envelope, while assuming $BV$-ellipticity~\cite[Chapter~5]{Ambrosio2000FunctionsOB} of~$g$ and non-standard growth for~$f$. In this regard, we will suppose that there exists a superlinear {\em generalized $\Phi$-function} (see also Definitions~\ref{wcs} and~\ref{def: genorliczclass}) $\psi\colon \Om \times [0, +\infty) \to [0, +\infty)$ such that
\begin{equation}
\label{e:intro3}
a \psi(x, |\xi|) \leq f(x, \xi) \leq b(1 + \psi(x, |\xi|))  \qquad \text{for $(x, \xi) \in \Om \times \R^{d \times m}$}
\end{equation}
for some $0 < a < b<+\infty$. The superlinearity of~$\psi$ (and thus of~$f$) is expressed by the conditions~\hyperlink{id}{$\textnormal{(Inc)}_p$} and \hyperlink{id}{$\textnormal{(Dec)}_{q}$} for some $p , q \in (1, +\infty)$ (see also Definition~\ref{inc dec general def}), meaning that for a.e.~$x \in \Om$ the maps $t \mapsto \frac{\psi(x, t)}{t^{p}}$ and $t \mapsto \frac{\psi(x, t)}{t^{q}}$ are monotone increasing and monotone decreasing, respectively. The third basic hypothesis on~$\psi$, common to all our results, is~\hyperlink{A0}{\textnormal{(A0)}}. Loosely speaking, such condition does not allow for a too degenerate behavior on small balls contained in~$\Om$ of the functions
\begin{equation}
\label{e:intro3}
\psi^{+}_{B} (t) := \sup_{x \in B} \, \psi(x, t) \qquad \psi^{-}_{B} (t):= \inf_{x \in B} \, \psi(x, t) \qquad \text{for $t \in [0, +\infty)$ and $B \subseteq \Om$}.
\end{equation}
We remark that~\hyperlink{id}{$\textnormal{(Inc)}_p$},~\hyperlink{id}{$\textnormal{(Dec)}_{q}$}, and~\hyperlink{A0}{\textnormal{(A0)}} are well-suited for a blow-up argument, which is at the core of our proofs strategy, and are rather standard in the theory of generalized Orlicz spaces~\cite{Hastobook}.

Let us discuss our results in more details. In Theorem~\ref{lsc of the functional} we prove the  lower semicontinuity of the functional~$\G$ in the space~$GSBV^{\psi} (\Om; \R^{m})$, the space of functions~$u \in GSBV(\Om; \R^{m})$ with~$L^{\psi}$-integrable approximate gradient~$\nabla u$. We refer to Definition~\ref{norma orlicz} and Section~\ref{s:gen} for more details on such space. Besides~\hyperlink{id}{$\textnormal{(Inc)}_p$},~\hyperlink{id}{$\textnormal{(Dec)}_{q}$}, and~\hyperlink{A0}{\textnormal{(A0)}}, in Theorem~\ref{lsc of the functional} we ask for the quasiconvexity of~$f$ and for a mild continuity property of~$\psi(x, t)$ with respect to~$x \in \Om$ (see~\eqref{.} for a precise statement). We notice that the last condition allows for improvements in the existing theory of generalized Orlicz spaces. For instance, in the variable exponent case the log-H\"older continuity of~$p(\cdot)$ is not necessary for~\eqref{.}, as pointed out in Section~\ref{subs: examples}. The proof of Theorem~\ref{lsc of the functional} is based on the approximation strategy of~\cite{marc} and on a localization technique which leads us to study the asymptotic behavior of~$\G$ around Lebesgue points of the limit energy, as first done in~\cite{Ambrosiolsc} in the $GSBV$-setting (see also~\cite[Theorem~5.29]{Ambrosio2000FunctionsOB}). Our assumptions are designed in such a way that the maximal operator for $SBV^{p}$-functions, exploited in~\cite{Ambrosiolsc} for the construction of more regular approximating sequences, can be replaced by the maximal operator in Orlicz spaces,  for which continuity estimates has been obtained in~\cite{Hastothemaximal}  (see also Theorem~\ref{max op bdd teo}). 

In Theorem~\ref{relax} we show a relaxation formula of the functional~$\G$ of the form
\begin{equation}
\label{e:intro5}
\overline{\G} (u) := \int_{\Om} \mathcal{Q}f(x, \nabla u)\, dx + \int_{J_{u}} \mathcal{R}g(x, [u], \nu_{u}) d \mathcal{H}^{d-1},
\end{equation}
which therefore maintains the original structure~\eqref{e:intro2}. In particular, in~\eqref{e:intro5} the functions~$\mathcal{Q}f\colon \Om \times \R^{m \times d} \to [0, +\infty)$ and~$\mathcal{R}g\colon \Om \times \R^{m} \times \mathbb{S}^{d-1} \to [0, +\infty)$ denote the quasiconvex and the BV-elliptic envelope of~$f$ and~$g$, respectively. When dealing with relaxation, we have to strengthen condition~\eqref{.}. Thus, in Theorem~\ref{relax} we replace the latter with the stronger assumption~\hyperlink{adA1}{\textnormal{(adA1)}}, which allows to estimate~$\psi^{+}_{B}$ with~$\psi^{-}_{B}$ on small balls~$B \subseteq \Om$ with a fixed control rate. We remark that condition~\hyperlink{adA1}{\textnormal{(adA1)}} is weaker than the more traditional~\hyperlink{A1}{(A1)}, introduced in the framework of Orlicz spaces (see, e.g.,~\cite{Hastobook}). In the variable exponent case,~\hyperlink{adA1}{\textnormal{(adA1)}} still requires log-H\"older continuity of~$p(\cdot)$, while in the double-phase case it calls for a H\"older continuity of the weight~$a(x)$, which may be however weaker than the one considered in regularity theory in Sobolev/Orlicz setting (see, for instance,~\cite{MR3775180, MR3294408}). Further examples are discussed in Section~\ref{subs: examples}.

The proof of Theorem~\ref{relax} follows the well-established strategy of the global method of relaxation~\cite{BFLMrelaxSBVp, BFMrelaxglobal}, which need to be adapted to the non-standard growth framework~\eqref{e:intro3}. The crucial step towards the relaxation formula~\eqref{e:intro5} is the integral representation of the lower semicontinuous envelope~$\overline{\G}$ of~$\G$ in $GSBV^{\psi} (\Om; \R^{m})$ (see Theorem~\ref{int rep teo} and Corollary~\ref{trans invar rep}). This ensures that~$\overline{\G}$ can still be written as the sum of a bulk energy depending on the approximate gradient~$\nabla u \in L^{\psi} (\Om; \R^{m \times d})$ and of a surface term obtained by integrating over the jump set~$J_{u}$ of~$u$ a suitable function depending on~$x$,~$[u]$, and~$\nu_{u}$. The fundamental tool to characterize the behavior of the blow-up energy on jump points and on approximate differentiability points of~$u$ is a new Poincar\'e inequality for special functions of bounded variation with~$L^{\varphi}$-integrable approximate gradient, where $\varphi$ denotes a (generalized) $\Phi$-function in the sense of Definition~\ref{wcs}. Such inequality,  obtained in Theorems~\ref{POINCARE} and~\ref{POINCARE GEN}, is of independent mathematical interest, as it can not be recovered by the classical approach of~\cite{DGCLexistence} or of~\cite{SSSvariableexpo} for the variable exponent, while it builds upon fine properties of rearrangments of $BV$-functions, adapting and combining the techniques of~\cite{AlvinoLionTromb, CianchiNachrichten, CianchiFuscoCrelle, CianchiFuscoBV}. The possible non-homogeneity of~$\varphi$ in space implies that a suitable truncation of~$u$ can not be controlled only in term of~$\nabla u$, as some remainder term appears, that can be estimated along the blow-up procedure thanks to~\hyperlink{adA1}{\textnormal{(adA1)}}. The crucial role played by the Poincar\'e inequality is evident in Lemmas~\ref{vol seq} and~\ref{surf seq} and enables us to recover the densities of the lower semicontinuous envelope~$\overline{\G}$ as blow-up limits of cell minimization formulas. In both lemmas, the Poincar\'e inequality is exploited to replace the optimal blow-up sequences with more regular functions, without excessively increasing the energy.

\subsection*{Outlook and open problems.} Our paper presents a general framework for lower semicontinuity and relaxation of free discontinuity functionals under non-standard growth of the bulk functional. The generalized $\Phi$-functions we consider cover most of the examples appeared in the literature, and in particular among them the variable exponent and the double phase~\eqref{e:intro1}. Because of the superlinear assumption~\hyperlink{id}{$\textnormal{(Inc)}_p$} for $p>1$, however, we can not handle an $\ell \log \ell$-kind of behavior, which has been studied in~\cite{leoverdeLSC} for a lower-semicontinuity problem in~$SBV$ without $x$-dependence on the integrand function in~\eqref{e:intro2}. Hence, an extension of our results in the above direction will be considered in a forthcoming research. Furthermore, regularity issues for minimizers of the functional~$\G$ may be investigated, in the spirit of~\cite[Chapter~6]{Ambrosio2000FunctionsOB}, and could lead to stronger assumptions on the growth function~$\psi$ (see, e.g.,~\cite{acerbi2, MR1895714, acerbi3} for the variable exponent). Also the investigation of lower semicontinuity and relaxation issues for free discontinuity functionals with bulk energies having mixed $(p,q)$-growth condition, relevant for the modeling of determinant constraints, is not fully covered by our theory and deserves further analysis.

\subsection*{Plan of the paper.} In Section~\ref{s:notation} we introduce the basic notation of the paper. In Section~\ref{s:mainresults} we state the main results of our paper: the integral representation in Theorem~\ref{int rep teo} and Corollary~\ref{trans invar rep}, the lower semicontinuity result in Theorem~\ref{lsc of the functional}, and the relaxation formula in Theorem~\ref{relax}. Section~\ref{sec: preliminaries} is devoted to some preliminaries on (generalized) $\Phi$-function and on $(G)SBV$-spaces, as well as to the discussion of the main assumptions of the above theorems. In particular, in Section~\ref{subs: examples} we report a number of generalized $\Phi$-function known in the literature, and discuss how they fall into our theory. In Section~\ref{s:poincare} we state and prove the Poincar\'e inequality for $SBV$-functions with $L^{\varphi}$-integrable approximate gradient (see Theorems~\ref{POINCARE} and~\ref{POINCARE GEN}). Finally, in Sections~\ref{s:integral-repr}--\ref{s:relax} we proceed with the proof of Theorems~\ref{int rep teo},~\ref{lsc of the functional}, and~\ref{relax}, respectively.

	\section{Notations}
	\label{s:notation}
	
	Through the paper we assume that $\Omega \subset \R^d$ is a bounded open set with Lipschitz boundary and that $d \geq 2$. We denote by $\mathcal{A}(\Omega)$ and $\mathcal{B}(\Omega)$ the family of open sets and the family of Borel measurable sets contained in $\Omega$, respectively. For every $x \in \R^d$ and $\varepsilon>0$ we indicate with $B_\varepsilon(x)$ the ball centered in $x$ with radius $\varepsilon$, if $x=0$ we write $B_\varepsilon$. Given $x \in \R^d$ we indicate with $|x|$ its Euclidean norm. The set $\R^{m \times d}$ is the set of $m \times d$ matrices with real coefficients, $\mathbb{S}^{d-1}:=\{ x \in \R^d \colon |x|=1 \}$, and $\R^d_0:=\R^d \setminus \{ 0 \}$. The Lebesgue measure of the $d$-dimensional unit ball is indicated $\omega_d$. We denote by $\mathcal{L}^d$ and $\mathcal{H}^k$ the $d$-dimensional Lebesgue measure and the $k$-dimensional Hausdorff measure, respectively. The space $L^0(\Omega)$ stands for the space of all measurable functions in $\Omega$.
	Given $x_0 \in \R^d$ and $\varepsilon>0$, for any set $A \subset \R^d$ we set
	\begin{equation}\label{000}
		A_{\varepsilon,x_0}:=x_0+\varepsilon (A-x_0).
	\end{equation}
	The closure of a set $A$ is indicated with $\overline{A}$, the diameter with $\diam(A)$. Given two sets $A_1,A_2 \subset \R^d$ we denote their symmetric difference with $A_1 \Delta A_2$. We write $\chi_A$ for the characteristic function of any set $A \subset \R^d$. If $A$ is a set of finite perimeter we indicate with $\partial^M A$ its essential boundary (the points of the boundary which do not have density zero nor one) and with $\partial^* A$ its reduced boundary (the points for which a "normal" can be defined).
	For a Carath\'eodory function $\varphi \colon \Omega \times \mathbb{R}^+\to \mathbb{R}$ and a ball $B \subset \R^d$ we define
	\begin{equation}\label{+-}
		\varphi^-_B(t):=\essinf_{x \in B \cap \Omega} \varphi(x,t) \ \ \ \mbox{and} \ \ \ \varphi^-_B(t):=\esssup_{x \in B \cap \Omega} \varphi(x,t).
	\end{equation}
 	for every $t \geq 0$.
	For a monotone function $\varphi$ on the real line, the customary notations $\varphi(t-)$ and $\varphi(t+)$ are used to denote left and right limits in $t$, respectively.
	For an increasing coercive function~$\varphi$, setting for simplicity $\varphi(+\infty):=+\infty$, we denote by $\varphi^{-1} \colon [0,+\infty] \to [0,+\infty]$  the left inverse of $\varphi$ defined as
	$$
	\varphi^{-1}(s):=\inf \{ t \geq 0 \colon \varphi(t) \geq s \}.
	$$

	\section{Main Results}
	\label{s:mainresults}
	
	The paper will be concerned with integral representation, lower semicontinuity, and relaxation of functionals defined on the space  $\gsbv^{\psi}(\Omega,\R^m)$ of generalized special functions of bounded variation with $\psi$-growth on the gradient. Here $\psi$ is a suitable (generalized) Orlicz function $\psi$. We will consider functionals $$\mathcal{F}: \gsbv^{\psi}(\Omega,\R^m) \times \mathcal{A}(\Omega) \to [0,+\infty)$$ satisfying the following assumptions:
	\begin{enumerate}[label=(H\arabic*), ref=H\arabic*]
		\item $\mathcal{F}(u,\cdot)$ is a Borel measure for any $u \in \gsbv^{\psi}(\Omega,\R^m)$; \label{H1}
		\item $\mathcal{F}(\cdot,A)$ is lower semicontinuous with respect to convergence in measure in $\Omega$ for any $A \in \mathcal{A}(\Omega)$; \label{H2}
		\item $\mathcal{F}(\cdot,A)$ is local for every $A \in \mathcal{A}(\Omega)$, that is, if $u,v \in \gsbv^{\psi}(\Omega,\R^m)$ satisfy $u=v$ $\mathcal{L}^d$-a.e. in $A$ then $\mathcal{F}(u,A)=\mathcal{F}(v,A)$; \label{H3}
		\item there exist $0<a<b$ such that for every $u \in \gsbv^{\psi}(\Omega,\R^m)$ and every $B \in \mathcal{B}(\Omega)$ it holds \label{H4}
		$$
		a\left( \int_B \psi(x,|\nabla u|) dx + \mathcal{H}^{d-1}(J_u \cap B) \right) \leq \mathcal{F}(u,B) \leq b\left( \int_B (1+\psi(x,|\nabla u|)) dx + \mathcal{H}^{d-1}(J_u \cap B) \right).
		$$
	\end{enumerate}

In \eqref{H4}, as we said, we consider a generalized Orlicz function $\psi$. The assumptions we make on $\psi$ will be detailed in Section \ref{sec: preliminaries} (see \hyperlink{idg}{(aInc)}, \hyperlink{idg}{(aDec)} and the so-called weight condition \hyperlink{A0}{(A0)}). These properties are standard in the theory of generalized Orlicz functions. We also assume that $\psi$ satisfies 
	\begin{itemize}
	\item[(adA1)] For every ball $B \subset \Omega$ with $\diam(B) \leq 1$ there exists $\beta \in (0,1)$ such that 
	\begin{equation*}
		\psi_B^+(\beta t) \leq \psi_B^-(t) \ \ \ \mbox{for all $t \in \left[\sigma,(\psi_B^-)^{-1}\left( \frac{1}{\diam(B)} \right)\right]$}\,,
	\end{equation*}
	\end{itemize}
	where $\sigma \ge 1$ is the constant in the weight condition \hyperlink{A0}{\textnormal{(A0)}}. We remark that this condition is a {\it weaker variant} of the standard assumption \hyperlink{A1}{(A1)} (see Definition \ref{adA1}), called local continuity condition and usually considered for generalized Orlicz spaces. 
All the main results of the paper will hold, if one requires that  \hyperlink{A0}{\textnormal{(A0)}}, \hyperlink{adA1}{\textnormal{(adA1)}}, \hyperlink{idg}{$\textnormal{(aInc)}$} and \hyperlink{idg}{\textnormal{(aDec)}} are satisfied on $\Omega$. They encompass a broad range of applications, with some relevant examples that will be discussed in Section \ref{subs: examples}.

	Our first result concerns the integral representation, and requires some notation to be fixed. For every $u \in \gsbv^{\psi}(\Omega,\R^m)$ and every $A \in \mathcal{A}(\Omega)$ we define
	\begin{equation}\label{m}
		\mathbf{m}_\mathcal{F}(u,A):=\inf_{v \in \gsbv^{\psi}(\Omega,\R^m)} \{ \mathcal{F}(v,A) \colon v=u \ \mbox{in a neighborhood of $\partial A$} \}.
	\end{equation}
	Moreover, given $x_0 \in \Omega$, $u_0 \in \R^m$ and $\xi \in \R^{m \times d}$, we define the affine function $\ell_{x_0,u_0,\xi} \colon \R^d \to \R^m$ as 
	$$
	\ell_{x_0,u_0,\xi}:=u_0+\xi(x-x_0).
	$$
	Given $x_0 \in \Omega$, $\mathfrak{a},\mathfrak{b} \in \R^m$ and $\nu \in \mathbb{S}^{d-1}$ we also introduce $u_{x_0,\mathfrak{a},\mathfrak{b},\nu} \colon \R^d \to \R^m$ as
	$$
	u_{x_0,\mathfrak{a},\mathfrak{b},\nu}=
	\begin{cases*}
		\mathfrak{a} & if $(x-x_0) \cdot \nu>0$, \\
		\mathfrak{b} & if $(x-x_0) \cdot \nu <0$.
	\end{cases*}
	$$
Our main result concerning the integral representation is the following. Below, $\Phi_w(\Omega)$ denotes the class of weak generalized  $\Phi$-functions, whose definition is recalled in Definition \ref{def: genorliczclass}.
	\begin{thm}[Integral representation in $\gsbv^{\psi}$]\label{int rep teo}
		Let $\psi \in \Phi_w(\Omega)$ satisfy \hyperlink{A0}{\textnormal{(A0)}}, \hyperlink{adA1}{\textnormal{(adA1)}}, \hyperlink{idg}{$\textnormal{(aInc)}$} and \hyperlink{idg}{\textnormal{(aDec)}} on $\Omega$. Assume that $\mathcal{F}: \gsbv^{\psi}(\Omega,\R^m) \times \mathcal{A}(\Omega) \to [0,+\infty)$ satisfies assumptions \eqref{H1}--\eqref{H4}. Then, for all $u \in \gsbv^{\psi}(\Omega,\R^m)$ and all $A \in \mathcal{A}(\Omega)$
		$$
		\mathcal{F}(u,A)=\int_A f(x,u(x),\nabla u(x))\, dx+ \int_{J_u \cap A} g(x,u^+(x),u^-(x),\nu_u(x)) \, d\mathcal{H}^{d-1},
		$$
		where
		\begin{align}\label{f}
			&f(x_0,u_0,\xi):=\limsup_{\varepsilon \to 0^+} \frac{\mathbf{m}_\mathcal{F}(\ell_{x_0,u_0,\xi},B_\varepsilon(x_0))}{\omega_d \varepsilon^d} \ \ \ \mbox{for $(x_0,u_0,\xi_0) \in \Omega \times \R^m \times \R^{m \times d}$}, \\
			&g(x_0,\mathfrak{a},\mathfrak{b},\nu):=\limsup_{\varepsilon \to 0^+} \frac{\mathbf{m}_\mathcal{F}(u_{x_0,\mathfrak{a},\mathfrak{b},\nu},B_\varepsilon(x_0))}{\omega_{d-1} \varepsilon^{d-1}} \ \ \ \mbox{for all $x_0 \in \Omega$, $\mathfrak{a},\mathfrak{b} \in \R^m$ and $\nu \in \mathbb{S}^{d-1}$}. \label{g}
		\end{align}
	\end{thm}
\noindent If translation invariance with respect to $u$ is assumed  on $\mathcal{F}$, namely
	\begin{enumerate}[label=(H5), ref=H5]
		\item $\mathcal{F}(u+c,A)=\mathcal{F}(u,A)$ for every $u \in \gsbv^\psi(\Omega,\R^m)$, every $A \in \mathcal{B}(\Omega)$ and every $c \in \R^m$\,,\label{H5}
	\end{enumerate}
we have the following result.
	\begin{cor}\label{trans invar rep}
		Let $\psi \in \Phi_w(\Omega)$ be as in Theorem \ref{int rep teo} and suppose that $\mathcal{F}: \gsbv^{\psi}(\Omega,\R^m) \times \mathcal{A}(\Omega) \to [0,+\infty)$ satisfies assumptions \eqref{H1}--\eqref{H5}. Then, for all $u \in \gsbv^{\psi}(\Omega,\R^m)$ and all $A \in \mathcal{A}(\Omega)$
		$$
		\mathcal{F}(u,A)=\int_A f(x,\nabla u(x))\, dx+ \int_{J_u \cap A} g(x,[u](x),\nu_u(x)) \, d\mathcal{H}^{d-1},
		$$
		where $f$ and $g$ are as in \eqref{f} and \eqref{g}, respectively.
	\end{cor}
	
	Our next result concerns the lower semicontinuity in $\gsbv^\psi(\Omega,\R^m)$ of variational functionals $\mathcal{G} : L^0(\Omega,\R^m) \times \mathcal{A}(\Omega) \to [0,+\infty]$ of the form
	\begin{equation}\label{functional}
		\mathcal{G}(u,A)= \int_A f(x,\nabla u(x))+\int_{J_u} g(x,[u](x),\nu_u(x)) \, d\mathcal{H}^{d-1},
	\end{equation}
	In \eqref{functional} the function $f:\Omega \times \R^{m \times d} \to [0,+\infty)$ satisfies the following assumptions:
	\begin{enumerate} [label=(f\arabic*), ref=f\arabic*]
		\item \label{f1} $f$ is a Carathéodory function;
		\item \label{f2} there exist two constants $a,b>0$ such that 
		\begin{equation}\label{growth of f}
			a \psi(x,|\xi|) \leq f(x,\xi) \leq b(1+\psi(x,|\xi|))
		\end{equation}
		for every $x \in \Omega$ and $\xi \in \R^{m \times d}$.
	\end{enumerate}
 	On the other hand, the function $g \colon \Omega \times  \R^m_0 \times \Sf^{d-1} \to [0,+\infty)$ satisfies the following assumptions:
 	\begin{enumerate} [label=(g\arabic*), ref=g\arabic*]
 		\item \label{g1} $g$ is a Borel measurable function lower semicontinuous in $x$ and continuous in the remaining variables;
 		\item \label{g2} there exist $\alpha_1,\alpha_2>0$ such that for every $x \in \Omega$, $\zeta \in \R^m_0$ and $\nu \in \Sf^{d-1}$
 		$$
 		\alpha_1 \leq g(x,\zeta,\nu) \leq \alpha_2.
 		$$
 	\end{enumerate}
	 For the result below, we can further weaken \hyperlink{adA1}{\textnormal{(adA1)}}. We will namely assume that $\psi$ complies with the following property: for $\sigma \ge 1$ being the constant in the weight condition  \hyperlink{A0}{\textnormal{(A0)}}, and for $\mathcal{L}^d$-a.e. $x_0 \in \Omega$ there exists $C=C(x_0)>0$ such that
		\begin{align}\label{.}
			\begin{split}
				\mbox{given $\theta>\sigma$, we can find $\varepsilon_0>0$} & \mbox{ such that for every $\varepsilon \leq \varepsilon_0$ and every $t \in [\sigma,\theta],$} \\
				&\psi^+_{B_\varepsilon(x_0)}(t) \leq C \psi^-_{B_\varepsilon(x_0)}(t).
			\end{split}
		\end{align} 
	As we will discuss in Section \ref{sec: preliminaries}, if $\psi \in \Phi_w(\Omega)$ satisfies \hyperlink{A0}{\textnormal{(A0)}}, \hyperlink{idg}{\textnormal{(aDec)}} and \hyperlink{adA1}{(adA1)} on $\Omega$, then it also satisfies \eqref{.}.  Hence, the assumptions for the lower semicontinuity Theorem below are weaker than those in Theorem \ref{int rep teo}.
	
	\begin{thm}\label{lsc of the functional}
		Let $\psi \in \Phi_w(\Omega)$ satisfying \hyperlink{A0}{\textnormal{(A0)}},  \hyperlink{idg}{$\textnormal{(aInc)}$} and \hyperlink{idg}{\textnormal{(aDec)}} on $\Omega$. Assume also that $\psi$ is satisfies \eqref{.}. 
		Consider a functional $\mathcal{G} \colon L^0(\Omega,\R^m) \times \mathcal{A}(\Omega) \to [0,+\infty]$ as in \eqref{functional}. Let $f \colon \Omega \times \R^{m \times d} \to [0,+\infty)$ be a function satisfying \eqref{f1}-\eqref{f2} and such that $z \mapsto f(x,z)$ is quasiconvex in $\R^{m \times d}$ for every $x \in \Omega$. 
		Let $g \colon \Omega \times \R^m_0 \times \mathbb{S}^{d-1} \to [0, +\infty)$ be a function satisfying \eqref{g1}--\eqref{g2} and such that $(\zeta,\nu) \mapsto g(x,\zeta,\nu)$ is BV-elliptic for every $x \in \Omega$.
		Then, for every $A \in \mathcal{A}(\Omega)$, we have
		$$
		\mathcal{G}(u,A) \leq \liminf_{k \to + \infty} \mathcal{G}(u_k,A)
		$$
		for every sequence $\{ u_k \}_k \subset \gsbv^{\psi}(A,\R^m)$ converging to a function $u \in \gsbv^{\psi}(A,\R^m)$ in measure.
	\end{thm}

	The third main result is a relaxation result, which requires both the use of Theorems \ref{int rep teo} and \ref{lsc of the functional}. Given $\mathcal{G}$ as in \eqref{functional}, for every $u \in \gsbv^{\psi}(\Omega,\R^m)$ and every $A \in \mathcal{A}(\Omega)$ we denote the lower semicontinuous envelope of the functional $\mathcal{G}$ as
	\begin{equation*}
		\overline{\mathcal{G}}(u,A):=\inf \left\{ \liminf_{k \to + \infty} \mathcal{G}(u_k,A) \colon \{ u_k \}_k \subset \gsbv^{\psi}(A,\R^m) \mbox{ and } u_k \to u \mbox{ in measure on $A$} \right\}.
	\end{equation*}
	We assume that $g \colon \Omega \times \R^m_0 \times \mathbb{S}^{d-1} \to [0,+\infty)$ satisfies also the following additional properties:
	\begin{enumerate}[label=(g3), ref=g3]
		\item \label{g3} there exists $c>0$ such that for every $x \in \Omega$ and every $\nu \in \Sf^{d-1}$ it holds
		$$
		g(x,\zeta_1,\nu) \leq g(x,\zeta_2,\nu) \ \ \ \mbox{for every $\zeta_1, \zeta_2 \in \R^m_0$ with $c |\zeta_1| \leq |\zeta_2|$;}
		$$
	\end{enumerate}
	\begin{enumerate}[label=(g4), ref=g4]
		\item \label{g4} for every $x \in \Omega$, $\zeta \in \R^m_0$ and $\nu \in \Sf^{d-1}$
		$$
		g(x,\zeta,\nu)=g(x,-\zeta,-\nu).
		$$
	\end{enumerate}
	
	\begin{thm}\label{relax}
		Let $\psi\in \Phi_w(\Omega)$ satisfying \hyperlink{A0}{\textnormal{(A0)}}, \hyperlink{adA1}{\textnormal{(adA1)}}, \hyperlink{idg}{$\textnormal{(aInc)}$} and \hyperlink{idg}{\textnormal{(aDec)}} on $\Omega$. Let $\mathcal{G}$ be as in \eqref{functional} and $f \colon \Omega \times \R^{m \times d} \to [0,+\infty)$ satisfying \eqref{f1} and \eqref{f2}.  Assume also that $g \colon \Omega \times  \R^m_0 \times \Sf^{d-1} \to [0,+\infty)$ is a continuous function satisfying \eqref{g1}--\eqref{g4}.
		Then, for every $u \in \gsbv^{\psi}(\Omega,\R^m)$ and every $A \in \mathcal{A}(\Omega)$,
		\begin{equation*}
			\overline{\mathcal{G}}(u,A)=\int_A \mathcal{Q}f(x,\nabla u(x)) \, dx +\int_{J_u \cap A} \mathcal{R}g(x,[u](x),\nu_u(x)) \, d\mathcal{H}^{d-1},
		\end{equation*}
		where $\mathcal{Q}f$ is the quasiconvex envelope of $f$ and $\mathcal{R}g$ is the BV-elliptic envelope of $g$.
	\end{thm}

	\section{Preliminaries}
	\label{sec: preliminaries}
	
In this preliminary section we recall some basic facts about (generalized) $\Phi$-functions and Orlicz spaces, and on $GSBV$ functions. Finally, in Subsection \ref{subs: examples} we give  a list of non-standard growth functions $\psi$ which fit into the scope of our results.	
	\subsection{(Generalized) $\Phi$-functions and Orlicz spaces}
	We begin by collecting some basic definitions and useful facts about $\Phi$-functions and generalized Orlicz spaces. For a complete treatment of the topic (and the proofs of the statements below which are left without proof) we refer to \cite{Hastobook}.
	
	\begin{defn}\label{almost inc def}
		A function $g\colon (0,+\infty) \to \R$ is almost increasing (resp. almost decreasing) if there exists a constant $a \geq 1$ such that $g(s) \leq ag(t)$ (resp $ag(s) \geq g(t)$) for every $0<s<t$. 
	\end{defn}
	
	Increasing and decreasing functions are included in the above definition if $a=1$.
	
	\begin{defn}\label{inc dec} \hypertarget{id} 
		Let $f \colon (0,\infty) \to \R$ and $p,q>0$. We say that $f$ satisfies 
		\begin{itemize}
			\item[$\textnormal{(Inc)}_p$] if $\frac{f(t)}{t^p}$ is increasing; 
			\item[$\textnormal{(aInc)}_p$] if $\frac{f(t)}{t^p}$ is almost increasing;
			\item[$\textnormal{(Dec)}_q$] if $\frac{f(t)}{t^q}$ is decreasing;
			\item[$\textnormal{(aDec)}_q$] if $\frac{f(t)}{t^q}$ is almost decreasing;
		\end{itemize}
		We say that $f$ satisfies $\textnormal{(aInc)}$, $\textnormal{(Inc)}$, $\textnormal{(aDec)}$ or $\textnormal{(Dec)}$ if there exist $p>1$ or $q<\infty$ such that $f$ satisfies $\textnormal{(Inc)}_p$, $\textnormal{(aInc)}_p$, $\textnormal{(Dec)}_q$ or $\textnormal{(aDec)}_q$, respectively.
	\end{defn}
	
	\begin{defn}\label{wcs}
		Let $\varphi : [0,+\infty) \to [0,+\infty]$ be increasing with $\varphi(0)=0=\lim_{t \to 0^+} \varphi(t)$ and $\lim_{t \to +\infty} \varphi(t)=+\infty$. Such $\varphi$ is called a $\Phi$-prefunction. We say that a $\Phi$-prefunction $\varphi$ is a
		\begin{itemize}
			\item[-] weak $\Phi$-function if it satisfies \hyperlink{id}{$\textnormal{(aInc)}_1$} on $(0,+\infty)$;
			\item[-] convex $\Phi$-function if it is left-continuous and convex;
			\item[-] strong $\Phi$-function if it is continuous in the topology of $[0,+\infty]$ and convex.
		\end{itemize}
		The set of weak, convex and strong $\Phi$-functions are denoted by $\Phi_w$, $\Phi_c$ and $\Phi_s$, respectively.
	\end{defn}
	
	It follows by definition that $\Phi_s \subset \Phi_c \subset \Phi_w$. If a function $\varphi \in \Phi_c$ satisfies \hyperlink{id}{(aDec)} then $\varphi \in \Phi_s$.
	
	\begin{defn}\label{equiv func}
		Two function $\varphi$ and $\psi$ are called equivalent, $\varphi \simeq \psi$, if there exists $L \geq 1$ such that $\varphi(t/L) \leq \psi(t) \leq \varphi(Lt)$ for all $t \geq 0$. The notation "$\approx$" is instead used with the following meaning: $\varphi \approx \psi$ if and only if there exist two constants $c_1,c_2>0$ such that $c_1 \varphi \leq \psi \leq c_2 \varphi$. 
	\end{defn}
	
	\begin{lem}\label{equiv properties}
		Let $\varphi,\psi \colon [0,+\infty) \to [0,+\infty]$ be increasing with $\varphi \simeq \psi$. Then, the following facts hold:
		\begin{enumerate}
			\item[\textnormal{(a)}] if $\varphi$ is a $\Phi$-prefunction, then $\psi$ is a $\Phi$-prefunction;
			\item[\textnormal{(b)}] if $\varphi$ satisfies \hyperlink{id}{$\textnormal{(aInc)}_p$}, then $\psi$ satisfies \hyperlink{id}{$\textnormal{(aInc)}_p$};
			\item[\textnormal{(c)}] if $\varphi$ satisfies \hyperlink{id}{$\textnormal{(aDec)}_q$}, then $\psi$ satisfies \hyperlink{id}{$\textnormal{(aDec)}_q$}.
		\end{enumerate}
	\end{lem}
	
	\begin{lem}\label{upgrade}
		If $\varphi \in \Phi_w$ satisfies \hyperlink{id}{$\textnormal{(aInc)}_p$} with $p \geq 1$, then there exists $\psi \in \Phi_c$ equivalent to $\varphi$ such that $\psi^{1/p}$ is convex. In particular, $\psi$ satisfies  \hyperlink{id}{$\textnormal{(Inc)}_p$}.
	\end{lem}
	
	\begin{thm}\label{weak and strong equiv}
		Every weak $\Phi$-function is equivalent to a strong $\Phi$-function. Moreover, if $\varphi \in \Phi_w$ is finite valued and satisfies \hyperlink{id}{$\textnormal{(Inc)}_1$}, then we can find $\psi \in \Phi_s$ such that
		$$
		\psi(t) \leq \varphi(t) \leq \psi(2t) \ \ \ t \geq 0.
		$$
	\end{thm}
	
	Now, we recall the concept of doubling functions and its equivalence with  \hyperlink{id}{$\textnormal{(aDec)}$}.
	
	\begin{defn}\label{doubling}
		We say that a function $\varphi \colon [0,+\infty) \to [0,+\infty]$ satisfies $\Delta_2$, or that it is doubling, if there exists a constant $K \geq 2$ such that 
		$$
		\varphi(2t) \leq K\varphi(t) \ \ \ \mbox{for all $t \geq 0$}.
		$$
	\end{defn}
	
	\begin{lem}\label{doubling and dec}
		The following statements hold.
		\begin{enumerate}
			\item[\textnormal{(a)}] If $\varphi \in \Phi_w$, then $\Delta_2$ is equivalent to  \hyperlink{id}{$\textnormal{(aDec)}$}.
			\item[\textnormal{(b)}] If $\varphi \in \Phi_c$, then $\Delta_2$ is equivalent to  \hyperlink{id}{$\textnormal{(Dec)}$}.
		\end{enumerate}
	\end{lem}
	
	\begin{rem}
	By the previous Lemma it follows that if $\varphi \in \Phi_c$ satisfies  \hyperlink{id}{$\textnormal{(aDec)}_q$}, then it satisfies  \hyperlink{id}{$\textnormal{(Dec)}_{q_2}$} for some possibly larger $q_2$. This and Lemma \ref{equiv properties}(c) imply that if $\varphi$ satisfies $\Delta_2$ and $\psi \simeq \varphi$, then $\psi$ satisfies $\Delta_2$.
	\end{rem}
	
	\begin{thm}\label{inverse equiv}
		Let $\psi,\varphi \in \Phi_w$. Then $\varphi \simeq \psi$ if and only if $\varphi^{-1} \approx \psi^{-1}$. In particular, if $\psi(t/L) \leq \varphi(t) \leq \psi(Lt)$ for some $L \geq 1$ then,
		$$
		\psi^{-1}(t)/L \leq \varphi^{-1}(t) \leq L\psi^{-1}(t).
		$$ 
	\end{thm}
	We now come to the generalized setting, where explicit dependence on the space variable $x$ is allowed.
	
	\begin{defn}\label{inc dec general def}\hypertarget{idg}
		Let $(A,\Sigma,\mu)$ be a complete, $\sigma$-finite measure space. Let $\varphi \colon A \times [0,+\infty) \to [0,+\infty]$ and $p,q>0$. We say that $\varphi$ satisfies $\textnormal{(aInc)}_p$ or $\textnormal{(aDec)}_q$ if there exists a constant $a \geq 1$ such that $t \mapsto \varphi(x,t)$ satisfies \hyperlink{id}{$\textnormal{(aInc)}_p$} or \hyperlink{id}{$\textnormal{(aDec)}_q$} respectively, for $\mu$-a.e. $x \in A$. When $a=1$ we use the notation $\textnormal{(Inc)}_p$ and $\textnormal{(Dec)}_q$. For (Inc) and (Dec) the definition is analogous.
	\end{defn}
	
	\begin{defn}\label{def: genorliczclass}
		Let $(A,\Sigma,\mu)$ be a complete, $\sigma$-finite measure space. A function $\varphi \colon A \times [0,+\infty) \to [0,+\infty]$ is said to be a generalized $\Phi$-prefunction on $(A,\Sigma,\mu)$ if $x \mapsto \varphi(x,|f(x)|)$ is measurable for any $f \in L^0(A,\mu)$ and $\varphi(x,\cdot)$ is a $\Phi$-prefunction for $\mu$-a.e. $x \in A$. We say that the generalized $\Phi$-prefunction $\varphi$ is
		\begin{enumerate}
			\item[-] a weak generalized $\Phi$-function if it satisfies \hyperlink{idg}{$\textnormal{(Inc)}_1$} on $(0,+\infty)$;
			\item[-] a convex generalized $\Phi$-function if $\varphi(x,\cdot) \in \Phi_c$ for $\mu$-a.e. $x \in A$;
			\item[-] a strong generalized $\Phi$-function if $\varphi(x,\cdot) \in \Phi_s$ for $\mu$-a.e. $x \in A$.
		\end{enumerate}
	\end{defn}
	If $\varphi$ is a generalized weak $\Phi$-function on $(A,\Sigma,\mu)$, we write $\varphi \in \Phi_w(A,\mu)$. Similarly we define $\varphi \in \Phi_c(A,\mu)$ and $\varphi \in \Phi_s(A,\mu)$. If $A \subset \R^d$ is an open set and $\mu=\mathcal{L}^d$, we omit the measure dependence and simply write $\Phi_w(A)$, $\Phi_c(A)$ or $\Phi_s(A)$ and we say that $\varphi$ is a generalized $\Phi$-function on $A$. 
	
	\begin{prop}
		Let $\varphi \colon A \times [0,+\infty) \to [0,+\infty]$, $x \mapsto \varphi(x,t)$ be measurable for every $t \geq 0$ and $t \mapsto \varphi(x,t)$ be increasing and left-continuous for $\mu$-a.e. $x \in A$. If $f \in L^0(A,\mu)$. Then $x \mapsto \varphi(x,|f(x)|)$ is measurable.
	\end{prop}
	
	Properties of $\Phi$-functions are generalized point-wise uniformly to the generalized $\Phi$-function case.
	
	\begin{defn}\label{equiv func gen}
		Two function $\varphi, \psi \colon A \times [0,+\infty) \to [0,+\infty]$ are called equivalent, $\varphi \simeq \psi$, if there exists $L \geq 1$ such that $$\varphi(x,t/L) \leq \psi(x,t) \leq \varphi(x,Lt) \ \ \ \mbox{for $\mu$-a.e. $x \in A$ and for all $t \geq 0$.}$$
	\end{defn}
	
	\begin{lem}\label{equiv properties gen}
		Let $\varphi,\psi \colon A \times [0,+\infty) \to [0,+\infty]$, be increasing with respect to the second variable, such that $\varphi \simeq \psi$, and $x \mapsto \varphi(x,|f(x)|)$ and $x \mapsto \psi(x,|f(x)|)$ be measurable for any $f \in L^0(A,\mu)$. Then
		\begin{enumerate}
			\item[\textnormal{(a)}] if $\varphi$ is a generalized $\Phi$-prefunction, then $\psi$ is a generalized $\Phi$-prefunction;
			\item[\textnormal{(b)}] if $\varphi$ satisfies \hyperlink{idg}{$\textnormal{(aInc)}_p$}, then $\psi$ satisfies \hyperlink{idg}{$\textnormal{(aInc)}_p$};
			\item[\textnormal{(c)}] if $\varphi$ satisfies \hyperlink{idg}{$\textnormal{(aDec)}_q$}, then $\psi$ satisfies \hyperlink{idg}{$\textnormal{(aDec)}_q$}.
		\end{enumerate}
	\end{lem}
	
	\begin{lem}\label{upgrade gen}
		If $\varphi \in \Phi_w(A,\mu)$ satisfies \hyperlink{id}{$\textnormal{(aInc)}_p$} with $p \geq 1$, then there exists $\psi \in \Phi_c(A,\mu)$ equivalent to $\varphi$ such that $\psi^{1/p}$ is convex. In particular, $\psi$ satisfies  \hyperlink{id}{$\textnormal{(Inc)}_p$}.
	\end{lem}
	
	We can now discuss our main assumptions on the growth function $\psi$, starting from the weight condition (A0) and the local continuity condition (A1).
	\begin{defn}\label{A0} \hypertarget{A0}
		Let $\varphi \in \Phi_w(\Omega)$. We say that $\varphi$ satisfies (A0) if there exists a constant $\sigma \geq 1$ such that
		$$
		\varphi\left(x,\frac{1}{\sigma}\right) \leq 1 \leq \varphi(x,\sigma) \ \ \ \mbox{for $\mathcal{L}^d$-a.e. $x \in \Omega$}.
		$$
	\end{defn}
	
	\begin{prop}\label{equiv A0}
		Given $\varphi,\psi \in \Phi_w(\Omega)$, if $\varphi \simeq \psi$ and $\varphi$ satisfies \hyperlink{A0}{\textnormal{(A0)}}, then $\psi$ satisfies \hyperlink{A0}{\textnormal{(A0)}}. 
	\end{prop}
	
	If $\varphi \in \Phi_w(\Omega)$ satisfies \hyperlink{A0}{(A0)}, then $\varphi_B^-,\varphi_B^+ \in \Phi_w$ for every $B \Subset \Omega$ (see \cite[Lemma 2.5.16]{Hastobook}).
	
	\begin{defn}\label{A1} \hypertarget{A1}
		Let $\varphi \in \Phi_w(\Omega)$ satisfy \hyperlink{A0}{(A0)}. We say that $\varphi$ satisfies (A1) if for every ball $B \subset \Omega$ with $\mathcal{L}^d(B) \leq 1$ there exists $\beta \in (0,1)$ such that
			\begin{equation}\label{iv}
				\varphi_B^+(\beta t) \leq \varphi_B^-(t) \ \ \ \mbox{for all $t \in \left[\sigma,(\varphi_B^-)^{-1}\left( \frac{1}{\mathcal{L}^d(B)} \right)\right]$}.
			\end{equation}
	\end{defn}
	
	Notice that the above definition explicitly depends on the space dimension $d$. To our purposes, an adimensional (and weaker) version of \hyperlink{A1}{(A1)} will be sufficient. We state it below.

	\begin{defn}\label{adA1} \hypertarget{adA1}
		Let $\varphi \in \Phi_w(\Omega)$ satisfy \hyperlink{A0}{(A0)}. We say that $\varphi$ satisfies (adA1)
		if for every ball $B \subset \Omega$ with $\diam(B) \leq 1$ there exists $\beta \in (0,1)$ such that 
		\begin{equation}\label{iiiv}
			\varphi_B^+(\beta t) \leq \varphi_B^-(t) \ \ \ \mbox{for all $t \in \left[\sigma,(\varphi_B^-)^{-1}\left( \frac{1}{\diam(B)} \right)\right]$}.
		\end{equation}
	\end{defn}
	
	\begin{lem}\label{equiv A1}
		Let $\varphi,\psi \in \Phi_w(\Omega)$ and $\varphi \simeq \psi$. If $\varphi$ satisfies \hyperlink{A0}{\textnormal{(A0)}} and \hyperlink{A1}{\textnormal{(A1)}} \textnormal{(}resp. \hyperlink{adA1}{\textnormal{(adA1)}}\textnormal{)}, then $\psi$ does as well. 
	\end{lem}

	\begin{rem}\label{RESCALE}
		Let $\varphi \in \Phi_s(B_\varepsilon)$. Define $\varphi_\varepsilon(x,\cdot):=\varphi(\varepsilon x,\cdot)$ for every $x \in B_1$. If $\varphi$ satisfies \hyperlink{A0}{(A0)} with a constant $\sigma \geq 1$ then $\psi$ satisfies \hyperlink{A0}{(A0)} with the same constant $\sigma$ of $\varphi$. If $\varphi$ satisfies \hyperlink{idg}{$\textnormal{(Inc)}_p$} with $p \in [1,\infty)$ (resp. \hyperlink{idg}{$\textnormal{(Dec)}_q$} with $q \in [1,\infty)$), then $\psi$ satisfies \hyperlink{idg}{$\textnormal{(Inc)}_p$} with the same $p$ (resp. \hyperlink{idg}{$\textnormal{(Dec)}_q$} with the same $q$).
	\end{rem}
	
	The next two results show that the validity of our main assumptions is unchanged by passing to equivalent $\Phi$-functions. will be used to simplify our proofs while remaining in full generality.
	
	\begin{lem}\label{UPGRDE}
		Let $\varphi \in \Phi_w(\Omega)$ satisfy \hyperlink{A0}{\textnormal{(A0)}}, \hyperlink{adA1}{\textnormal{(adA1)}}, \hyperlink{idg}{$\textnormal{(aInc)}$} and \hyperlink{idg}{\textnormal{(aDec)}} on $\Omega$. Then, there exist $\psi \in \Phi_s(\Omega)$ which satisfies \hyperlink{A0}{\textnormal{(A0)}}, \hyperlink{adA1}{\textnormal{(adA1)}}, \hyperlink{idg}{$\textnormal{(Inc)}$} and \hyperlink{idg}{\textnormal{(Dec)}} on $\Omega$ and such that $\psi \simeq \varphi$ and $\psi \approx \varphi$.
	\end{lem}
	
	\begin{proof}
		By assumption we have that $\varphi$ satisfies \hyperlink{idg}{\textnormal{(aDec)}} on $\Omega$ and thus is finite valued. Therefore, using Lemma \ref{upgrade gen}, there exists $\psi \in \Phi_c(\Omega)$ and hence in $\Phi_s(\Omega)$ such that $\psi \simeq \varphi$ and satisfying \hyperlink{idg}{\textnormal{(Inc)}} on $\Omega$. Moreover, by Lemma \ref{equiv properties gen} and Lemma \ref{doubling and dec}, we have that $\psi$ satisfies \hyperlink{idg}{\textnormal{(Dec)}} on $\Omega$. Finally, properties \hyperlink{A0}{\textnormal{(A0)}} and \hyperlink{adA1}{\textnormal{(adA1)}} come from Lemma \ref{equiv A0} and \ref{equiv A1}, respectively, while the fact that $\psi \approx \varphi$ is a consequence of $\psi \simeq \varphi$ and the fact that both functions are doubling.
	\end{proof}

	\begin{lem}\label{equiv of 2.6}
		Let $\varphi \in \Phi_w(\Omega)$ satisfy \hyperlink{A0}{\textnormal{(A0)}}, and \hyperlink{idg}{\textnormal{(aDec)}} on $\Omega$ and \eqref{.} $\mathcal{L}^d$-a.e. in $\Omega$. Given $\psi \in \Phi_w(\Omega)$ such that $\varphi \simeq \psi$ then $\psi$ satisfies \eqref{.} $\mathcal{L}^d$-a.e. in $\Omega$.
	\end{lem}

	\begin{proof}
		We have that $\psi$ satisfies \hyperlink{A0}{\textnormal{(A0)}} with some $\sigma \geq 1$, and it is doubling on $\Omega$ with a constant $K \geq 2$. Let $x_0 \in \Omega$ such that \eqref{.} holds for $\varphi$, we want to prove that \eqref{.} holds for $\psi$ in $x_0$ as well. We have that there exists $C_\varphi=C_\varphi(x_0)>0$ such that given $\theta > \sigma$, we can find $\varepsilon_0 >0$ such that for every $\varepsilon \leq \varepsilon_0$ and every $t \in [\sigma,\theta]$
		$$
		\varphi_{B_\varepsilon(x_0)}^+(t) \leq C_\varphi \varphi_{B_\varepsilon(x_0)}^-(t).
		$$
		There exists $L \geq 1$ with $\varphi(t/L) \leq \psi(t) \leq \varphi(Lt)$ for every $t \geq 0$. Observe that $\psi$ satisfies \hyperlink{A0}{\textnormal{(A0)}} with $L\sigma$. Then, for every $\theta \geq L\sigma$ we deduce that there exists $\varepsilon_0=\varepsilon_0(L\theta)>0$ such that for every $\varepsilon \leq \varepsilon_0$
		$$
		\psi_{B_\varepsilon(x_0)}^+(t) \leq \varphi_{B_\varepsilon(x_0)}^+(Lt) \leq C_\varphi \varphi_{B_\varepsilon(x_0)}^-(Lt) \leq  C_\varphi \psi_{B_\varepsilon(x_0)}^-(L^2 t) \leq C_\varphi K^{2\log_2(L)+1}\psi_{B_\varepsilon(x_0)}^-(t), \ \ \ t \in [L\sigma,\theta],
		$$
		where we have used property \eqref{.} of $\varphi$ in $x_0$ and the fact that $\psi$ is doubling. Therefore $\psi$ satisfies property \eqref{.} in $x_0$ with the constant $C_\varphi K^{2\log_2(L)+1}$.
	\end{proof}

	Finally, in the next elementary Lemma we observe that \hyperlink{adA1}{(adA1)}, together with  \hyperlink{A0}{\textnormal{(A0)}}, \hyperlink{idg}{\textnormal{(aDec)}}, is stronger than condition \eqref{.} on $\Omega$.
	
	\begin{lem}\label{adA1 then 2.6}
		Let $\varphi \in \Phi_w(\Omega)$ satisfy \hyperlink{A0}{\textnormal{(A0)}}, \hyperlink{idg}{\textnormal{(aDec)}} and \hyperlink{adA1}{\textnormal{(adA1)}}. Then $\varphi$ satisfies \eqref{.} on every $x_0 \in \Omega$.
	\end{lem}

	\begin{proof}
		Let $x_0 \in \Omega$, let $\sigma$ be the constant of \hyperlink{A0}{\textnormal{(A0)}} and let $K$ be the doubling constant of $\psi$. Fix $\theta > \sigma$. We take $\varepsilon_0>0$ such that $\theta \leq \varphi_{B_{\varepsilon_0}(x_0)}^-\left(\frac{1}{2\varepsilon_0}\right)$. By \hyperlink{adA1}{\textnormal{(adA1)}} there exists $\beta \in (0,1)$ such that for every $\varepsilon \leq \varepsilon_0$
		$$
		\varphi_{B_\varepsilon(x_0)}^+(\beta t) \leq \varphi_{B_\varepsilon(x_0)}^-( t), \ \ \ \mbox{ for every $t \in [\sigma,\theta] \subseteq \left[\sigma,\varphi_{B_\varepsilon(x_0)}^-\left(\frac{1}{2\varepsilon}\right)\right]$}.
		$$
		This together with the doubling property of $\varphi$ implies that \eqref{.} holds for every $x_0 \in \Omega$ with the uniform constant $K^{1-\log_2(\beta)}$.
	\end{proof}
	
	We recall the concept of generalized Orlicz space and report a brief list of useful definitions and properties of such spaces.
	Given $\varphi \in \Phi_w(\Omega)$ and $f \in L^0(\Omega,\R^m)$ we define the \textit{modular} $\rho_\varphi(f)$ as
	$$
	\rho_\varphi(f):=\int_{\Omega} \varphi(x,|f(x)|) \, dx.
	$$
	The set 
	$$
	L^{\varphi}(\Omega,\R^m):=\{ f \in L^0(\Omega) \colon \rho_\varphi(\lambda f)<+\infty \ \ \ \mbox{for some $\lambda>0$} \},
	$$
	is called a generalized Orlicz space. 
	
	\begin{lem}\label{lemma modulor}
		Let $\varphi \in \Phi_w(\Omega)$. Then
		\begin{enumerate}
			\item[\textnormal{(a)}] $L^\varphi(\Omega,\R^m)=\{ f \in L^0(\Omega,\R^m) \colon \lim_{\lambda \to 0^+} \rho_\varphi(\lambda f)=0 \}$;
			\item[\textnormal{(b)}] if $\varphi$ additionally satisfies \hyperlink{idg}{$\textnormal{(aDec)}$} we have
			$$
			L^\varphi(\Omega,\R^m)=\{ f \in L^0(\Omega,\R^m) \colon \rho_\varphi(f)<+\infty \}.
			$$
		\end{enumerate}
	\end{lem}
	
	Next we define the (quasi)-norm associated to generalized Orlicz spaces.
	
	\begin{defn}\label{norma orlicz}
		Let $\varphi \in \Phi_w(\Omega)$. For $f \in L^0(\Omega,\R^m)$ we define
		$$
		\Vert f \Vert_{L^\varphi(\Omega,\R^m)}:=\inf \left\{ \lambda>0 \colon \rho_\varphi\left(\frac{f}{\lambda}\right) \leq 1 \right\}.
		$$
	\end{defn}
	Most of the time we will abbreviate the notation writing $\Vert f \Vert_\varphi$. Now we present some properties linking the modular with the Orlicz norm.
	
	\begin{lem}\label{orl norm o no}
		If $\varphi \in \Phi_w(\Omega)$ then $\Vert \cdot \Vert_\varphi$ is a quasi-norm. If $\varphi \in \Phi_c(\Omega)$ then $\Vert \cdot \Vert_\varphi$ is a norm.
	\end{lem}
	
	\begin{prop}\label{equiv funcs then equiv norms}
		Let $\varphi,\psi \in \Phi_w(\Omega)$. If $\varphi \simeq \psi$, then $L^\varphi(\Omega,\R^m)=L^\psi(\Omega,\R^m)$ and the two norms are equivalent.
	\end{prop}

\begin{rem}\label{rem: convexgratis}
Lemmas \ref{UPGRDE}, \ref{equiv of 2.6} and \ref{equiv funcs then equiv norms} entail, in particular, that it suffices to prove all our results under the stronger assumption that $\psi \in \Phi_s(\Omega)$ satisfying \hyperlink{A0}{$\textnormal{(A0)}$}, \hyperlink{idg}{$\textnormal{(Inc)}$}, \hyperlink{idg}{$\textnormal{(Dec)}$} and \hyperlink{adA1}{(adA1)} (or \eqref{.}, depending on the need) on $\Omega$. The extension to a $\psi \in\Phi_w(\Omega)$ is indeed immediate, upon noticing that, if we take equivalent Orlicz functions, then the generalized Orlicz space remains the same.
\end{rem}

	\begin{lem}[Unit ball property]\label{unit ball prop}
		Let $\varphi \in \Phi_w(\Omega)$. Given $f \in L^0(\Omega,\R^m)$ we have
		$$
		\Vert f \Vert_\varphi < 1 \ \ \ \Rightarrow \ \ \ \rho_\varphi(f) \leq 1 \ \ \ \Rightarrow \ \ \ \Vert f \Vert_\varphi \leq 1.
		$$
		If in addition $\varphi$ is left continuous then $\rho_\varphi(f) \leq 1 \Leftrightarrow \Vert f \Vert_\varphi \leq 1$. Moreover, the following properties hold:
		\begin{enumerate}
			\item[\textnormal{(a)}] if $\Vert f \Vert_\varphi <1$, then $\rho_\varphi(f) \leq \Vert f \Vert_\varphi$;
			\item[\textnormal{(b)}] if $\Vert f \Vert_\varphi > 1$, then $\Vert f \Vert_\varphi \leq \rho_\varphi(f)$;
			\item[\textnormal{(c)}] in any case, $\Vert f \Vert_\varphi \leq \rho_\varphi(f)+1$.
		\end{enumerate}
	\end{lem}
	
	In generalized Orlicz spaces we also have a generalization of the concept of Lebesgue points.
	
	\begin{prop}\label{leb pts prop}
		Let $\varphi \in \Phi_w(\Omega)$ satisfying \hyperlink{A0}{\textnormal{(A0)}} and \hyperlink{idg}{\textnormal{(aDec)}}. Then, for every $f \in L^0(\Omega,\R^m)$ such that $\rho_\varphi(f)<+\infty$, we have
		$$
		\lim_{\varepsilon \to 0^+} \fint_{B_\varepsilon(x_0)} \varphi(x,|f(x)-f(x_0)|) \, dx=0 \ \ \ \mbox{for $\mathcal{L}^d$-a.e. $x_0 \in \Omega$}.
		$$
	\end{prop}

	\begin{proof}
		The proof can be carried out exactly as the proof of Theorem 3.1 in \cite{HHlebpts} with some minor changes, so we omit it.
	\end{proof}
	
	We present some properties of the maximal operator in Orlicz  spaces (see \cite[Chapter 4]{Hastobook} or~\cite{Hastothemaximal}). Since the proof of Proposition \ref{100} only needs the estimate for $\Phi$-functions independent of $x$, we give the statement only for non generalized $\Phi$-functions.
	\begin{defn}\label{maximal ope}
		Given $f \in L^1_{\textnormal{loc}}(\Omega,\R^m)$ we define the restricted maximal operator to $\Omega$ as 
		$$
		M f(x):=\sup_{r>0} \frac{1}{\omega_d r^d} \int_{B_r(x) \cap \Omega} |f(y)| \, dy.
		$$
		Analogously, for a non-negative, finite Radon measure $\mu$ on $\Omega$ one can define
		$$
		M \mu(x):=\sup_{r >0} \frac{\mu(B_r(x) \cap \Omega)}{\omega_d r^d}.
		$$
	\end{defn}

	\begin{thm}\label{max op bdd teo}
		Let $\varphi \in \Phi_w$ finite valued satisfy \hyperlink{idg}{$\textnormal{(Inc)}_\gamma$} with $\gamma >1$. Then, the restricted maximal operator $M \colon L^\varphi(\Omega,\R^m) \to L^\varphi(\Omega,\R^m)$ is bounded.
		In particular, we have that if $\Vert f \Vert_\varphi \neq 0$, setting $\varepsilon:=\frac{1}{2\Vert f \Vert_\varphi}$, there exists $c=c(\varphi^{-1}(1),d,\gamma)$ such that $\rho_\varphi(c\varepsilon Mf) \leq 1$.
	\end{thm}

	\begin{proof}
		The proof builds upon the key estimate \cite[Theorem 4.3.3]{Hastobook}. In particular, let $c_1=c_1(d,\gamma)$ be the constant such that $\Vert M f \Vert_{L^\gamma(\Omega)} \leq c_1 \Vert f \Vert_{L^\gamma(\Omega)}$. Then, the constant $c$ is given by
		$$
		c:=\frac{(\varphi^{-1}(1))^3}{16 c_1}.
		$$
	\end{proof}
	
	\begin{cor}\label{max op bdd cor}
		Let $\varphi \in \Phi_w$ finite valued satisfy \hyperlink{idg}{$\textnormal{(Inc)}_\gamma$} with $\gamma >1$. Assume in addition that it satisfies \hyperlink{idg}{$\textnormal{(Dec)}_q$} with $1<q<\infty$. Then, there exists $C=C(\varphi^{-1}(1),d,\gamma,q)>0$ such that for every $f \in L^\varphi(\Omega,\R^m)$ 
		\begin{equation}\label{-100}
		\rho_\varphi(Mf) \leq C (\rho_\varphi(f)+1)^q.
		\end{equation}
	\end{cor}
	
	\begin{proof}
		Using Theorem \ref{weak and strong equiv}, we can find $\psi \in \Phi_s$ such that $\psi(t) \leq \varphi(t) \leq \psi(2t)$ for every $t \geq 0$. Moreover, by Lemma \ref{equiv properties}, $\psi$ satisfies \hyperlink{idg}{$\textnormal{(Inc)}_\gamma$} and \hyperlink{idg}{$\textnormal{(Dec)}_q$}. Let us prove \eqref{-100} for $\psi$. If $f \equiv 0$ then there is nothing to prove. Assume that $f$ is not identically zero on $\Omega$. Since $\Vert \cdot \Vert_\psi$ is a norm by Lemma \ref{orl norm o no}, we have that $\Vert f \Vert_\psi \neq 0$. Let $\varepsilon:=\frac{1}{2\Vert f \Vert_\psi}$ and $c=c(\psi^{-1}(1),d,\gamma)=c(\varphi^{-1}(1),d,\gamma)>0$ be the constant from Theorem \ref{max op bdd teo}. If $c\varepsilon \leq 1$ by \hyperlink{idg}{$\textnormal{(Dec)}_q$} and Theorem \ref{max op bdd teo} we have
		$$
		1 \geq \int_\Omega \psi(c\varepsilon Mf(x)) \, dx \geq (c\varepsilon)^{q} \int_\Omega \psi(x,Mf(x)) \, dx.
		$$
		Therefore, using also Lemma \ref{unit ball prop}(c) we get
		$$
		\int_\Omega \psi(Mf(x)) \, dx \leq c^{-q} (2 \Vert f \Vert_\psi)^q \leq 2^q c^{-q} \left(1+ \int_\Omega \psi(|f(x)|) \, dx \right)^q,
		$$
		If $c\varepsilon >1$, then it is enough to observe that since $\psi$ is increasing 
		$$
		\int_\Omega \psi( Mf(x)) \, dx \leq \int_\Omega \psi(c \varepsilon Mf(x)) \, dx \leq 1 \leq 1+ \Vert f \Vert_\psi \leq 2+2\rho_\psi(f) \leq 2(1+\rho_\psi(f))^q.
		$$
		Setting $C:=\max\{ 2^qc^{-q},2 \}$ we conclude \eqref{-100} for $\psi$. Finally, for $\varphi$ we have
		$$
		\rho_\varphi(Mf/2) \leq \rho_\psi(Mf) \leq C (\rho_\psi(f)+1)^q \leq C(\rho_\varphi(f)+1)^q,
		$$
		using again \hyperlink{idg}{$\textnormal{(Dec)}_q$}, we conclude \eqref{-100} for $\varphi$ as well with $C=C(\varphi^{-1}(1),d,\gamma,q)$.
	\end{proof}
	
	\subsection{(Generalized) Special functions of Bounded Variation}
	\label{s:gen}
	 We now discuss the definition and some useful properties of $\gsbv$ and $\sbv$ function. For a complete treatment of the topic we refer to \cite{Ambrosio2000FunctionsOB}.
	
	Let $A \in \mathcal{A}(\R^d)$ and $x \in A$ such that 
	$$
	\limsup_{r \searrow 0} \frac{\mathcal{L}^d(A \cap B(x,r))}{r^d}>0,
	$$
	Given $u \in L^0(A,\R^m)$, we say that $a \in \R^m$ is the approximate limit of $u$ at $x$ if
	$$
	\lim_{r \searrow 0} \frac{\mathcal{L}^d(\Omega \cap B(x,r) \cap \left\{ |a-v(x)|>\varepsilon \right\})}{r^d}=0 \ \ \ \mbox{for every $\varepsilon>0$}. 
	$$
	In that case we write 
	$$
	\aplim_{y \to x} u(y)=a.
	$$
	We say that $x \in A$ is an approximate jump point of $u$, and we write $x \in J_u$, if there exists $a, b \in \R^m$ with $a \neq b$ and $\nu \in \mathbb{S}^{d-1}$ such that
	$$
	\underset{\underset{(y-x) \cdot \nu>0}{y \to x}}{\textnormal{ap-}\lim} u(y)=a \ \ \ \ \mbox{and} \ \ \ \ \underset{\underset{(y-x) \cdot \nu<0}{y \to x}}{\textnormal{ap-}\lim} u(y)=b.
	$$
	In particular, for every $x \in J_u$ the triple $(a, b, \nu)$ is uniquely determined up to a change of sign of $\nu$ and a permutation of $a$ and $b$. We indicate such triple by $(u^+(x), u^-(x), \nu_u(x))$. The jump of $u$ at $x \in J_u$ is defined as $[u](x):=u^+(x)-u^-(x) \in \R^m_0$. The set $J_u$ is countably rectifiable and has $\nu$ as an approximate unit normal vector at $\mathcal{H}^{d-1}$-every point (see \cite{Nin20211233}).
	
	The space $\textnormal{BV}(A,\R^m)$ of functions of bounded variation is the set of $u \in L^1(A;\R^m)$ whose distributional gradient $Du$ is a bounded Radon measure on $A$ with values in $\R^{m \times d}$. Given $u \in \textnormal{BV}(A,\R^m)$ we can write $Du=D^a u+ D^s u$, where $D^a u$ is absolutely continuous and $D^s u$ is singular w.r.t. $\mathcal{L}^d$.  The density $\nabla u \in L^1(A,\R^{m \times d})$ of $D^a u$ w.r.t. $\mathcal{L}^d$ coincides a.e. in $A$ with the approximate gradient of $u$. That is, for a.e. $x \in A$ it holds
	$$
	\underset{y \to x}{\textnormal{ap-}\lim} \frac{u(y)-u(x)-\nabla u(x)\cdot (y-x)}{|x-y|}=0.
	$$
	
	The space $\sbv(A,\R^m)$ of special functions of bounded variation is defined as
	the set of all $u \in \textnormal{BV}(A,\R^m)$ such that $|D^s u|(A \setminus J_u)=0$. Moreover, we denote by $\sbv_{\textnormal{loc}}(\Omega,\R^m)$ the space of functions belonging to $\sbv(U,\R^m)$ for every $U \Subset A$ open. For $p \in [1,+\infty)$, $\sbvp(A,\R^m)$ stands for the set of functions $u \in \sbv(A,\R^m)$, with approximate gradient $\nabla u \in L^p(A,\R^{m \times d})$ and $\mathcal{H}^{d-1}(J_u)<+\infty$, that is
	$$
	\sbv^p(A,\R^m):=\{ u \in \sbv(A,\R^m) \colon \nabla u \in L^p(\Omega,\R^m), \mathcal{H}^{d-1}(J_u)<+\infty \}.
	$$
	We say that $u \in \gsbv(A,\R^m)$ if $X(u) \in \sbv_{\textnormal{loc}}(A,\R^m)$ for every $X \in C_c^1(\R^m,\R^m)$. Also for $u \in \gsbv(A,\R^m)$ the approximate gradient $\nabla u$ exists $\mathcal{L}^d$-a.e. in $A$ and the jump set $J_u$ is countably $\mathcal{H}^{d-1}$-rectifiable with approximate unit normal vector $\nu_u$. Finally, for $p \in [1,+\infty)$, we define as before
	$$
	\gsbv^p(A,\R^m):=\{ u \in \gsbv(A,\R^m) \colon \nabla u \in L^p(\Omega,\R^m), \mathcal{H}^{d-1}(J_u)<+\infty \};
	$$ 
	it is known that $\gsbvp(A,\R^m)$ is a vector space (see e.g. \cite[pg. 172]{DMFTqscrackgrowth}).
	
	Moreover, given a generalized Orlicz function $\varphi \in \Phi_w(\Omega)$ we denote with $\sbv^{\varphi}(\Omega,\R^m)$ the space of functions $u \in \sbv(\Omega,\R^m)$ with $\mathcal{H}^{d-1}(J_u) < +\infty$ and $\nabla u \in L^{\varphi}(\Omega,\R^{m \times d})$. The definition of $\gsbv^\varphi(\Omega,\R^m)$ is analogous. \\
	
	The following three theorems can be found in \cite{DGCLexistence} for the scalar case and in \cite{CLskvalued} for vector valued $\sbv$ functions.
	
	\begin{thm}[{\cite[Lemma 2.6]{DGCLexistence}}]\label{DGCL 1}
		Let $p \geq 1$ and $u \in \sbv(\Omega,\R^m)$ be such that 
		$$
		\int_K |\nabla u|^p \, dx+\mathcal{H}^{d-1}(J_u \cap K) <+\infty
		$$
		for every compact $K \Subset \Omega$. Then
		$$
		\lim_{\varepsilon \to 0^+} \frac{1}{\varepsilon^{d-1}} \left( \int_{B_\varepsilon(x_0)} |\nabla u(x)|^p \, dx +\mathcal{H}^{d-1}(B_\varepsilon(x_0) \cap J_u) \right)=0 \ \ \ \mbox{for $\mathcal{H}^{d-1}$-a.e. $x_0 \in \Omega \setminus J_u$.}
		$$
	\end{thm}
	
	For the other two theorems we need first to fix some notations.
	Given $a=(a_1,\dots,a_m)$ and $ b=(b_1,\dots,b_m)$ vectors in $ \R^m$, we set $$a \wedge b:=(\min(a_1, b_1),\dots,\min(a_m,b_m)) \ \ \ \mbox{and} \ \ \ a \vee b:=(\max(a_1, b_1),\dots,\max(a_m,b_m)).$$ 
	Let $B$ be a ball in $\R^d$. For every measurable function $u \colon B \to \R^m$, given $0 \leq s \leq \mathcal{L}^d(B)$, we define
	$$
	u_*(s;B):=((u_1)_*(s;B),\dots,(u_m)_*(s;B)),
	$$
	where, for every $i=1,\dots,m$,
	$$
	(u_i)_*(s;B):=\inf \left\{ t \in \R \colon \mathcal{L}^d (\left\{u_i<t \right\}) \geq s \right\}.
	$$
	Moreover, we define $\displaystyle \med(u;B):=u_*\left(\frac{\mathcal{L}^d(B)}{2};B\right)$.
	
	Let $\gammaiso$ be the dimensional constant in the relative isoperimetric inequality in balls. For every $u \in \gsbv(\Omega,\R^m)$ such that
	$$
	\left( 2 \gammaiso\mathcal{H}^{d-1}(J_u \cap B) \right)^{\frac{d}{d-1}} \leq \frac{1}{2} \mathcal{L}^d(B),
	$$ 
	we define
	\begin{align}
		\begin{split}\label{tau}
			&\tau'(u;B):=u_* \left( \left( 2\gammaiso\mathcal{H}^{d-1}(J_u \cap B) \right)^{\frac{d}{d-1}};B \right), \\ 
			&\tau''(u;B):=u_* \left( \mathcal{L}^d(B)- \left( 2\gammaiso\mathcal{H}^{d-1}(J_u \cap B) \right)^{\frac{d}{d-1}};B \right)
		\end{split}
	\end{align}
	and the truncation operator in $B$
	\begin{equation}\label{truncat operat}
		T_Bu(x):=(\tau'(u;B) \vee u(x)) \wedge \tau''(u;B).
	\end{equation}
	Notice that it holds
	\begin{equation}\label{truncated difference}
		\mathcal{L}^d(\{T_B u \neq u \} \cap B) \leq 2 \left( 2 \gammaiso \mathcal{H}^{d-1}(J_u \cap B) \right)^{\frac{d}{d-1}}.
	\end{equation}
	
	\begin{thm}[{\cite[Theorem 3.1 and Remark 3.2]{DGCLexistence}}]\label{DGCL 2 alas poinc ineq}
		Let $B \subset \R^d$ be a ball and $u \in \sbv(B,\R^m)$ satisfy
		\begin{equation}\label{poinc 1}
			\left( 2\gammaiso \mathcal{H}^{d-1}(J_u \cap B) \right)^{\frac{d}{d-1}} \leq \frac{1}{2}\mathcal{L}^d(B).
		\end{equation}
		If $1 \leq p <d$ then,
		\begin{equation}\label{poinc 2}
			\left( \int_B |T_B u-\med(u,B)|^{p^*} \, dx \right)^{1/p^*} \leq \frac{2 \gammaiso p(d-1)}{d-p} \left( \int_B |\nabla u|^{p} \, dx \right)^{1/p},
		\end{equation}
		and 
		\begin{equation}\label{poinc 3}
			\mathcal{L}^d(\{ T_B u \neq u \} \cap B) \leq 2\left( 2 \gammaiso \mathcal{H}^{d-1}(J_u \cap B) \right)^{\frac{d}{d-1}}.
		\end{equation}
		If $p \geq d$ instead, for every $q \geq 1$ we have
		\begin{equation}\label{poinc 4}
			\Vert T_B u -\med(u,B) \Vert_{L^q(B,\R^m)} \leq \frac{2 \gammaiso q (d-1)}{d} \mathcal{L}^d(B)^{\frac{1}{d}+\frac{1}{q}-\frac{1}{p}}\Vert \nabla u \Vert_{L^p(B,\R^m)}.
		\end{equation}
	\end{thm}
	
	\begin{thm}[{\cite[Theorem 3.6]{DGCLexistence}}]\label{DGCL 3}
		Let $u \in \sbv_{\textnormal{loc}}(\Omega,\R^m)$ and $p>1$ and let $x_0 \in \Omega$. If 
		$$
		\lim_{\varepsilon \to 0^+} \frac{1}{\varepsilon^{d-1}} \left( \int_{B_\varepsilon(x_0)} |\nabla u(x)|^p \, dx +\mathcal{H}^{d-1}(B_\varepsilon(x_0) \cap J_u) \right)=0,
		$$
		then $x_0 \neq J_u$ and 
		$$
		\lim_{\varepsilon \to 0^+} \med(u,B_\varepsilon(x_0))=\aplim_{x \to x_0} u(x) \in \R^m.
		$$
		Moreover, there exists $\varepsilon_0>0$ such that for $\delta \in (3/4,1)$ and $\varepsilon, \rho>0$ with $\delta \varepsilon < \rho < \varepsilon < \varepsilon_0$,
		\begin{align*}
			& \tau'(u,B_\rho(x_0)) \leq \med(u,B_\varepsilon(x_0)) \leq \tau''(u,B_\rho(x_0)), \\
			& \tau'(u,B_\varepsilon(x_0)) \leq \med(u,B_\rho(x_0)) \leq \tau''(u,B_\varepsilon(x_0)).
		\end{align*}
	\end{thm}

	\subsection{Examples}
	\label{subs: examples}
	In this last subsection of preliminaries, we present some examples of generalized Orlicz functions to which our theory applies. We consider the following Orlicz functions:
	
	\begin{itemize}
		\item[(I)] \textbf{Variable exponent}: $\psi(x,t)=t^{p(x)}$.
		For the variable exponent case, the semicontinuity was already addressed in \cite{DCLVlscvariable}, while integral representation and $\Gamma$-convergence are studied in \cite{SSSvariableexpo}. 
		\item[(II)] \textbf{Perturbed Orlicz}: $\psi(x,t)=a(x)\varphi(t)$. 
		\item[(III)] \textbf{Double phase}: $\psi(x,t)=t^p+a(x)t^q$. 
		\item[(IV)] \textbf{Degenerate double phase}: $\psi(x,t)=t^p+a(x)t^p\log(\mathrm{e}+t)$.
		\item[(V)] \textbf{Triple phase}: $\psi(x,t)=t^p+a(x)t^q+b(x)t^r$. 
		\item[(VI)] \textbf{Variable exponent double phase}: $\psi(x,t)=t^{p(x)}+a(x)t^{q(x)}$. This type of generalized Orlicz function was studied recently in \cite{BlancoDoublephVarexpo}.
	\end{itemize}
	We say that a function $p \colon \Omega \to [1,+\infty)$ is log-H\"older continuous on $\Omega$ if 
	$$
	\exists \ C>0 \ \ \ \mbox{such that} \ \ \ |p(x)-p(y)| \leq \frac{C}{-\log|x-y|} \ \ \ \mbox{for every $x,y \in \Omega$ with $|x-y|$} \leq \frac{1}{2}.
	$$
	We use the usual shortcut $p \in C^{\mathrm{log}}(\Omega)$ to denote a log-H\"older continuous variable exponent on $\Omega$.
	In the Table below we collect conditions that are sufficient (sometimes necessary) for properties \hyperlink{idg}{(aInc)}, \hyperlink{idg}{(aDec)}, \hyperlink{A0}{(A0)}, \eqref{.} and \hyperlink{adA1}{(adA1)} in the special cases above. The usual notations $C^0$ and $C^{0, \beta}$ are used for continuous and H\"older continuous functions with exponent $\beta$, respectively. In each line, the checkmark is used to denote that a property needs no new assumption to be satisfied than the previously considered ones.
	
	\begin{table} [h!]
		\centering
		\begin{tabular}{| c || c | c | c | c | p{46mm} | }
			\hline
			 & (aInc) & (aDec) & (A0) & \eqref{.} & (adA1) \\ [0.5ex]
			\hline
			(I) & $\essinf_\Omega p>1$ & $\esssup_\Omega p<+\infty$ & $\checkmark$ & $p \in C^0$ & $p \in C^{\log}$ \\
			\hline
			(II) & $\varphi$ is (aInc) & $\varphi$ is (aDec) & $c_1 \leq a\leq c_2 $ & $\checkmark$ & $\checkmark$ \\
			\hline
			(III) & $a \geq 0$, $p>1$ & $q<+\infty$ & $a \in L^\infty$ & $a \in C^0$ & $a \in C^{0,q/p-1}$ \\
			\hline
			(IV) & $a \geq 0$, $p>1$ & $p<+\infty$ & $a \in L^\infty$ & $a \in C^0$ & $a \in C^{\log}$\\
			\hline
			(V) & $a,b \geq 0$, $p>1$ & $q \leq r <+\infty$ & $a,b \in L^\infty$ & $a,b \in C^0$ & $a \in C^{0,q/p-1}$, $b \in C^{0,r/p-1}$ \\ 
			\hline
			(VI) & $\essinf_\Omega p >1$ & $\esssup_\Omega q<+\infty$ & $a \in L^\infty$ & $a, p,q \in C^0$ & $p \in C^{\log}$, $a \in C^{0,\alpha}$, \newline $q \in C^{0,\alpha/q^-}$, $\frac{q}{p} \leq 1+\alpha$  \\[1ex]
			\hline
		\end{tabular}
	\end{table}

We remark that in the case of variable exponent, the sole continuity assumption will be enough to obtain \eqref{.}  and thus Theorem \ref{lsc of the functional}. This represents an improvement with respect to the result of \cite{DCLVlscvariable}, where log--H\"older continuity has been assumed. \\

We briefly discuss how our conditions can be checked in all the aforementioned cases. Checking  the validity of \hyperlink{idg}{(aInc)} and  \hyperlink{idg}{(aDec)} is straightforward, while we refer to Table 1 in \cite{Har-Karpp} for condition \hyperlink{A0}{(A0)}.
Concerning \eqref{.},  it follows from the following general claim: given $\varphi \in \Phi_w(\Omega)$ satisfying \hyperlink{A0}{$\textnormal{(A0)}$} on $\Omega$, we have that if $\varphi \in C^0(\Omega \times [0,+\infty))$, then it satisfies property \eqref{.}. Now, assume that $\varphi$ satisfies \hyperlink{A0}{$\textnormal{(A0)}$} with $\sigma \geq 1$ and let for simplicity $x_0=0$. Let $\zeta>\sigma$ fixed. Since $\varphi$ is uniformly continuous on $B_\delta \times [\sigma,\zeta]$ for some $\delta>0$ small enough, we observe that there exists a uniform modulus of continuity $\omega \colon [0,+\infty) \to [0,+\infty)$ such that
		$$
		|\varphi(x,t)-\varphi(y,t)| \leq \omega(|x-y|) \ \ \ \mbox{for every $x,y \in B_\delta$ and every $t \in [\sigma,\zeta]$.}
		$$
		Let $\varepsilon_0 \in (0,\delta)$ such that $\omega(2\varepsilon_0) \leq 1$. Thanks to  \hyperlink{A0}{$\textnormal{(A0)}$} and the fact that $\varphi(y,\cdot)$ is increasing,  we estimate
		$$
		\varphi(x,t) \leq \varphi(y,t) +\omega(|x-y|) \leq \varphi(y,t)+1 \leq 2\varphi(y,t), \ \ \ \mbox{for every $x,y \in B_\varepsilon$ and every $t \in [\sigma,\zeta]$}\,.
		$$
		Thus, property \eqref{.} is satisfied for every $x_0 \in \Omega$ with $C=2$.\\ 
As for condition \hyperlink{adA1}{(adA1)}, we discuss it here for example in the double phase case. Namely, we prove that if $\psi(x,t)=|t|^p+a(x)|t|^q$ with $a \in C^{0,\alpha}(\overline{\Omega})$, $a(x) \geq 0$ and $1<p<q$ with $\frac{q}{p} \leq 1+\alpha$ for some $\alpha \in (0,1]$, then $\psi$ satisfies \hyperlink{adA1}{(adA1)}.

Note that $\psi(x,t) \approx \max\{ t^p,a(x)t^q \}$. Let $B \Subset \Omega$ a ball with $\diam(B) \leq 1$, denote $a^-_B:=\min_B a(\cdot)$ and $a^+_B:=\max_B a(\cdot)$. We need to show that $\max\{ t^p,a_B^+ t^q \} \lesssim \max\{ t^p,a^-_B t^q \}$ whenever $\psi^-_B(t) \leq \frac{1}{\diam(B)}$. In view of the fact that $a \in C^{0,\alpha}(\overline{\Omega})$ and $\frac{q}{p} \leq 1+\alpha$, we have 
$$
a_B^+ \lesssim \max\{ \diam(B)^{\frac{q-p}{p}}, a_B^- \}.
$$
Now, when $t^p \leq \frac{1}{\diam(B)}$, one has in particular  $\psi^-_B(t) \leq \frac{1}{\diam(B)}$, so let us assume that $t^p \leq \frac{1}{\diam(B)}$. Hence, $a^+_B \lesssim \max\{t^{p-q},a^-_B\}$, so that  $a^+_B t^q\lesssim \max\{t^{p},a^-_B t^q\}$. From this, we conclude.

	\section{Poincar\'e inequality in $\sbv^\varphi$}
	\label{s:poincare}
	
	In this section we present a fundamental ingredient we need in the proofs of our main results and which is in our opinion of independent interest. Namely, a Poincar\'e inequality for functions belonging to $\sbv^\varphi$ with small jump set a là De Giorgi-Carriero-Leaci. We start with the case of $\Phi$-functions, while the extension to the generalized setting will be given in Subsection \ref{subs: poincaregeneralized}.
	
	\begin{thm}[Poincar\'e inequality in $\sbv^\varphi$]\label{POINCARE}
		Let $\varphi \in \Phi_w$ be a finite valued Orlicz function satisfying \hyperlink{id}{$\textnormal{(Inc)}_1$} and let $B_r(x)$ be a ball in $\R^d$ with radius $r$ and centred in $x \in \R^d$. Set $\phi:=\varphi^{\frac{d}{d-1}}t^{-\frac{1}{d-1}}$ and let $u \in \sbv^\varphi(B_r(x))$  be such that
		\begin{equation}\label{poinc small jump}
			\left( 2 \gammaiso\mathcal{H}^{d-1}(J_u \cap B_r(x)) \right)^{\frac{d}{d-1}} \leq \frac{1}{2} \mathcal{L}^d(B_r).
		\end{equation}
		Then, there exists a dimensional constant $C=C(d)$ such that the following inequality holds
		\begin{equation}\label{poinc ineq}
			\phi^{-1} \left( \fint_{B_r(x)}  \phi\left(\frac{|T_{B_r(x)}u(y)-\med(u,B_r(x))|}{r}\right) \, dy \right) \leq C \varphi^{-1} \left( \fint_{B_r(x)} \varphi(C|\nabla u(y)|) \, dy \right),
		\end{equation}
		where $T_{B_r(x)} u$ is defined in \eqref{truncat operat}.
	\end{thm}
	 Notice that since we have a truncation of a scalar $\sbv$ function, the result above applies also to real valued $\gsbv$ functions. Also observe that if $\varphi$ has linear growth, then the inequality in \eqref{poinc ineq} reduces to the standard homogeneous Poincaré inequality with $p=1$ for $\sbv$ functions. 
	
	\begin{rem}\label{rescaling rem}
		By translation invariance and rescaling, it is enough to verify inequality \eqref{poinc ineq} only on $B_1$. Indeed, let $u \colon B_r(x) \to \R^m$, define $v \colon B_1 \to \R^m$ as $v(y):=u(x+ry)/r$. We have $u_*(s;B_r)=rv_*(s/r^d;B_1)$ for every $s \in [0,\mathcal{L}^d(B_r)]$. Thus,
		$$
		\tau'(v,B_1)=\frac{\tau'(u,B_r)}{r}, \ \ \ \tau''(v,B_1)=\frac{\tau''(u,B_r)}{r}, \ \ \ \med(v,B_1)=\frac{\med(u,B_r)}{r},
		$$ 
		and $T_{B_1}v(y):=T_{B_r(x)}u(x+ry)/r$ for every $y \in B_1$.
		
		Then, if the Poincaré inequality \eqref{poinc ineq} holds on $B_1$ for $v$, we deduce
		\begin{align*}
			\phi^{-1} \left( \fint_{B_r(x)}  \phi\left(\frac{|T_{B_r(x)}u(y)-\med(u,B_r(x))|}{r}\right) \, dy \right) & = \phi^{-1} \left( \fint_{B_1}  \phi\left(|T_{B_1}v(z)-\med(v,B_1)|\right) \, dz \right) \\
			& \leq C \varphi^{-1} \left( \fint_{B_1} \varphi(C|\nabla v(z)|) \, dz \right) \\
			& = C \varphi^{-1} \left( \fint_{B_r} \varphi(C|\nabla u(y)|) \, dx \right).
		\end{align*}
	\end{rem}

The proof of Theorem \ref{POINCARE} will be given in Subsection \ref{subs: proofPoincare}, after recalling and proving some preliminary material.
	
	\subsection{Rearrangements of $BV$ functions} We are going to make use of the theory of rearrangements for $BV$ functions. Since they are used only in this section, we report here their main definitions and properties. For a complete treatment of the topic see \cite{CianchiFuscoBV}. We say that two functions $u$ and $v$ are equi-measurable if $\mathcal{L}^d(\{ u >t \})  =\mathcal{L}^d(\{ v >t \})$ for every $t \in \R$.
	
	For simplicity, let us denote by $B$ the unit ball in $\R^d$ with $\mathcal{L}^d(B)=\omega_d$. Let $u \colon B \to \R$ be a measurable function. We define the distribution function $\nu_u \colon \R \to [0,\omega_d]$ of $u$ as
	\begin{equation}\label{nu}
		\nu_u(t):=\mathcal{L}^d(\{ x \in B \colon u(x)>t \}) \ \ \ t \in \R.
	\end{equation}
	We have that $\nu_u$ is right-continuous and non-increasing. By definition, $\nu_u(\esssup u)=0$. Moreover, $\nu_u(t-)=\mathcal{L}^d(\{u \geq t\})$ and $\nu_u(t-)-\nu_u(t)=\mathcal{L}^d(\{ u=t \})$. Thus, $\nu_u$ is continuous in $t$ if and only if $\mathcal{L}^d(\{  u=t \})=0$. We can now define also the signed decreasing rearrangement of $u$ as a function $u^0 \colon [0,\omega_d] \to \R \cup \{\pm \infty\}$ given by
	\begin{equation}\label{sign rearr}
		u^0(s):=\sup\{ t \in \R \colon \nu_u(t)>s \} \ \ \ s \in [0,\omega_d].
	\end{equation}
	We have that $u^0$ is right-continuous, non-increasing and it is equi-measurable with $u$. By definition, $u^0(0)=\esssup u$. Moreover, notice that
	$$
	\{ s \in [0,\omega_d] \colon u^0(s)>t \}=[0,\nu_u(t)) \ \ \ \mbox{for every $t \in \R$}.
	$$
	Therefore, $u^0$ and $u$ are equi-distributed, that is $\nu_{u}=\nu_{u^0}$. We also have
	\begin{equation}\label{-1}
		u^0(\nu_u(t)) \leq t \ \ \mbox{for every $t \in \R$}  \ \ \ \mbox{and} \ \ \ t \leq u^0(\nu_u(t)-)  \ \ \mbox{for every $t \in (-\infty,\esssup u)$}
	\end{equation}
	and
	\begin{equation}\label{-2}
		\nu_u(u^0(s)) \leq s \ \ \mbox{for every $s \in [0,\omega_d]$}  \ \ \ \mbox{and} \ \ \ s \leq \nu_u(u^0(s)-)  \ \ \mbox{for every $s \in [0,\omega_d)$}.
	\end{equation}
	
	To summarize, $u^0$ is a monotone real valued function on $(0,\omega_d)$ and thus $u^0 \in BV_{\textnormal{loc}}((0,\omega_d))$. Moreover, $u^0$ is constant in the interval $(s_1,s_2)$ if and only if there exists $t \in \R$ where $\nu_u$ jumps and $(s_1,s_2) \subset (\nu_u(t),\nu_u(t-))$. Vice versa, $u^0$ jumps at some point $s_0 \in (0,\omega_d)$ and $(t_1,t_2) \subset (u^0(s_0),u^0(s_0-))$ if and only if $\nu_u$ is constant on $(t_1,t_2)$. Therefore, a countable family of left-open intervals $(\alpha_i,\beta_i]$ exists such that 
	\begin{equation}\label{-3}
		(\essinf u,\esssup u) \setminus \bigcup_{i \in I} (\alpha_i,\beta_i] \subseteq u^0((0,\omega_d)) \subseteq [\essinf u,\esssup u] \setminus \bigcup_{i \in I} (\alpha_i,\beta_i].
	\end{equation}
	
	Finally, if $u$ is a non-negative function on $B$, we indicate the distribution function of $u$ as
	\begin{equation}\label{muu}
		\mu_u(t):=\mathcal{L}^d(\{ x \in B \colon u(x)>t \}) \,, \ \ t \in [0,+\infty),
	\end{equation}
	and the decreasing rearrangement as $u^* \colon [0,\omega_d] \to [0,+\infty]$ is given by
	\begin{equation}\label{dec rearr}
		u^*(s):=\sup\{ t \in [0,+\infty) \colon \mu_u(t)>s \} \,, \ \ s \in [0,\omega_d].
	\end{equation}
	
	We start by proving the following fact.
	
	\begin{prop}\label{p1}
		Let $\psi \in \Phi_s$ be a finite valued Orlicz function satisfying \hyperlink{id}{$\textnormal{(Inc)}_1$}. Then, for every $u \in \sbv(B)$ such that \eqref{poinc small jump} holds, we have that $(T_B u)^0$ is absolutely continuous in $(0,\omega_d)$ and
		\begin{equation}\label{-7}
			\int_0^{\omega_d} \psi \left( c(d)\min\{s,1-s\}^{\frac{d-1}{d}} \left(-\frac{d(T_B u)^0}{ds}(s) \right) \right) \, ds \leq \int_B \psi\left(2|\nabla u(x)|\right) \, dx,
		\end{equation}
		where $c(d)$ is a positive constant depending only on the dimension.
	\end{prop}
	
	In order to prove Proposition \ref{p1}, we need the following two results, the first one is \cite[Equation (3.21)]{CianchiFuscoBV} and the second is contained in \cite[Proposition 2.1]{AlvinoLionTromb}.
	
	\begin{lem}\label{lem -2}
		Let $u \in BV(B)$ and set $C_u \subseteq B$ the set of points in $B$ where $u$ is approximately continuous. We have that
		$$
		\partial^M \{ u>t \} \cap C_u = \emptyset \ \ \mbox{for a.e. $t \in (\essinf u,\esssup u) \setminus u^0((0,\omega_d))$}.
		$$
	\end{lem}
	
	\begin{proof}
		The proof can be carried out exactly as the proof of \cite[Equation (3.21)]{CianchiFuscoBV}.
	\end{proof}

	\begin{lem}\label{lem -3}
		Let $f,g \in L^1(0,\omega_d)$ such that $f(x) \geq 0$ and $g(x) \geq 0$ for $\mathcal{L}^1$-a.e. $x \in (0,\omega_d)$ and let $F \colon [0,+\infty) \to [0,+\infty)$ be a convex function with $F(0)=0$. Assume also that
		$$
		\int_{0}^t f^*(s) \, ds \leq \int_0^t g^*(s) \, ds \ \ \ \mbox{for every $t \in [0,\omega_d]$,} \ \ \ \ \int_0^{\omega_d} F(g^*(s)) \, ds < +\infty.
		$$
		Then, 
		\begin{equation}\label{1000}
		\int_0^{\omega_d} F(f(x)) \, dx \leq \int_0^{\omega_d} F(g(x)) \, dx.
		\end{equation}
	\end{lem}

	\begin{proof}
		The result is essentially proved in \cite[Proposition 2.1]{AlvinoLionTromb}, which actually requires $F$ to be also Lipschitz. However, an inspection of the proof reveals that we need $F$ to be Lipschitz only on the set $\{ f^*(s) \colon s \in [0,\omega_d] \}$. Since $F$ is convex and thus locally Lipschitz, fixing $\eta>0$ arbitrarily small, we can apply \cite[Proposition 2.1]{AlvinoLionTromb} to $f \wedge f^*(\eta)$ and $g$ getting
		$$
		\int_\eta^{\omega_d} F(f^*(s)) \, ds \leq \int_0^{\omega_d} F(g^*(s)) \, ds.
		$$
		Letting $\eta \to 0^+$, by monotone convergence theorem and using the fact that $F$ is strictly increasing, we deduce
		$$
		\int_0^{\omega_d} (F(f))^*(s) \, ds \leq \int_0^{\omega_d} (F(g))^*(s) \, ds.
		$$
		Finally, using the equi-measurability of rearrangements we conclude \eqref{1000}.
	\end{proof}

	We are now in a position to prove Proposition \ref{p1}.
	
	\begin{proof}[Proof of Proposition \ref{p1}]
		In order to ease the notation we set $(2 \gammaiso \mathcal{H}^{d-1}(J_u))^{\frac{d}{d-1}}=:\Lambda$. We split the proof into three steps.
		
		\noindent \textbf{Step 1:} \textbf{$u^0$ is continuous in $[\Lambda, \omega_d-\Lambda]$}. We actually prove that $u^0$ is continuous in
		\begin{equation}\label{interval}
			\left(\left(\gammaiso\mathcal{H}^{d-1}(J_u)\right)^{\frac{d}{d-1}},\omega_d- \left(\gammaiso\mathcal{H}^{d-1}(J_u)\right)^{\frac{d}{d-1}}\right).
		\end{equation}
		Assume this is false. Then, we can find $s$ belonging to the interval in \eqref{interval} such that $u^0$ jumps in $s$ and thus $(u^0(s),u^0(s-)) \neq \emptyset$. Hence, owning Lemma \ref{lem -2} and the fact that $u \in \sbv(B)$, we can find $t \in (u^0(s),u^0(s-))$ such that $\partial^M \{u >t\} \cap C_u=\emptyset$ and the set $\{u>t\}$ has finite perimeter. This implies $\partial^* \{u >t\} \subset J_u$. Moreover, recall that $\mathcal{H}^{d-1}(\partial^M \{u>t\} \setminus \partial^* \{u>t\} )=0$. By definition of $u^0$ we have that $\mathcal{L}^d(\{ u>t \})=s$. Thus, both $\{u>t\}$ and $B \setminus \{u >t\}$ have measure strictly greater than $(\gammaiso \mathcal{H}^{d-1}(J_u))^{\frac{d}{d-1}}$. Recalling the isoperimetric inequality this gives 
		$$
		\left(\gammaiso\mathcal{H}^{d-1} (\partial^* \{u>t\})\right)^{\frac{d}{d-1}} = \left(\gammaiso P(\{ u>t \},B)\right)^{\frac{d}{d-1}} >\left(\gammaiso\mathcal{H}^{d-1}(J_u)\right)^{\frac{d}{d-1}}.
		$$
		This means that $\mathcal{H}^{d-1}(J_u) < \mathcal{H}^{d-1} (\partial^* \{u>t\})$, a contradiction. Hence, $u^0$ is continuous in the interval in \eqref{interval} and thus in $[\Lambda,\omega_d -\Lambda]$. Since $u^0$ admits a continuous representative in $[\Lambda,\omega_d-\Lambda]$, let us define 
		$$
		\tau':=u^0\left( \omega_d - \Lambda \right) \ \ \ \mbox{and} \ \ \ \tau'':=u^0 \left( \Lambda \right).
		$$
		Notice that $\tau'$ and $\tau''$ by definition of $u^0$ and $\Lambda$ coincide with $\tau'(u,B)$ and $\tau''(u,B)$, respectively, defined in \eqref{tau}.
		
		\noindent \textbf{Step 2:} \textbf{$u^0$ is absolutely continuous in $(\Lambda,\omega_d-\Lambda)$}. We shall use some arguments from \cite[Theorem 6.5 and Lemma 6.6]{CianchiNachrichten}. Assume $\mathcal{H}^{d-1}(J_u) >0$, otherwise the absolute continuity of $u^0$ follows by \cite[Lemma 6.6]{CianchiNachrichten}. Let 
		$$
		h_B(s):=\frac{1}{\gammaiso} \min\{ s,\omega_d-s \}^{\frac{d-1}{d}}.
		$$
		Fix $\varepsilon>0$. Since $|\nabla u| \in L^1(B)$, by the standard theory of rearrangements we have $|\nabla u|^* \in L^1((0,\omega_d))$. Thus, we can find $\delta>0$ such that
		\begin{equation}\label{-5}
			\int_0^\delta |\nabla u|^*(s) \, ds \leq \varepsilon.
		\end{equation}
		Let $\alpha, \beta \in u^0((\Lambda,\omega_d-\Lambda))$ with $\alpha \leq \beta$. We have that $\alpha,\beta \in (\tau',\tau'')$ by Step 1. 
		Set $u^\beta_\alpha \in \sbv(B)$ as $u^\beta_\alpha:=\alpha \vee u \wedge \beta$. Using the coarea formula we deduce
		\begin{align}\label{-6}
			\int_{\alpha}^{\beta} \! P(\{ u>t\},B) \, dt = \int_{-\infty}^{+\infty} \! P(\{ u_\alpha^\beta>t\},B) \, dt = |Du^\beta_\alpha| (B) \leq  \int_B |\nabla u^\alpha_\beta(x)| \,dx + (\beta-\alpha)\mathcal{H}^{d-1}(J_u).
		\end{align}
		On the other hand, by definition of $h_B$ and the isoperimetric inequality in $B$ we infer
		\begin{align}\label{-7bis}
			\int_{\alpha}^{\beta} P(\{ u>t\},B)  dt \geq \int_{\alpha}^\beta h_B(\nu_u(t)) \, dt.
		\end{align}
		Since $t \in (\alpha,\beta) \subset (\tau',\tau'')$, using \eqref{-2} we have that
		$$
		\lim_{t \to u^0(\Lambda)^-} \nu_u(t) = \nu_u(u^0(\Lambda)-) \geq \Lambda \ \ \ \mbox{and} \ \ \ \lim_{t \to u^0(\omega_d-\Lambda)^+} \nu_u(t)=\nu_u(u^0(\omega_d-\Lambda)) \leq \omega_d-\Lambda.
		$$
		This gives that $\nu_u(t) \in (\Lambda,\omega_d-\Lambda)$ for every $t \in (\alpha,\beta)$. Hence, for every $t \in (\alpha,\beta)$,
		\begin{equation}\label{-8}
			h_B(\nu_u(t)) \geq h_B(\Lambda) =  h_B\left( \left(2 \gammaiso \mathcal{H}^{d-1}(J_u) \right)^{\frac{d}{d-1}} \right) = 2\mathcal{H}^{d-1}(J_u).
		\end{equation}
		Combining \eqref{-6}--\eqref{-8} gives 
		$$
		\int_\alpha^\beta h_B(\nu_u(t)) \, dt \leq \int_B |\nabla u^\beta_\alpha(x)|\, dx+\frac{1}{2} \int_\alpha^\beta h_B(\nu_u(t)) \, dt,
		$$
		and finally,
		\begin{equation}\label{-9}
			\int_\alpha^\beta h_B(\nu_u(t)) \, dt \leq 2\int_{\{ \alpha \leq u \leq \beta \}} |\nabla u(x)| \, dx.
		\end{equation}
	
		Consider now a finite family of disjoint open intervals contained in $(\Lambda,\omega_d-\Lambda)$, namely $(a_i,b_i)$ for $i=1, \dots,N$, such that $\sum_{i=1}^N (b_i -a_i) \leq \delta$ where $\delta>0$ is as in \eqref{-5}.
		Now consider inequality \eqref{-9} with $\alpha:=u^0(b_i) \leq u^0(a_i)=:\beta$ for every $i=1,\dots,N$. If we add all these inequalities we get
		\begin{equation}\label{-11}
			\int_{\cup_{i=1}^N (u^0(b_i),u^0(a_i))} h_B(\nu_u(t)) \, dt \leq 2\int_{\cup_{i=1}^N \{ u^0(b_i) < u < u^0(a_i) \}} |\nabla u(x)|\, dx,
		\end{equation}
		where we have also used the fact that $\int_{\{ u=t \}} |\nabla u| \, dx=0$ for every $t \in \R$ (see \cite[Proposition 3.73 (c)]{Ambrosio2000FunctionsOB}). 
		Recalling \eqref{-2}, for every $i=1,\dots,N$ it holds
		\begin{equation}\label{-10}
			\lim_{t \to u^0(a_i)^-} \nu_u(t) = \nu_u(u^0(a_i)-) \geq a_i \ \ \ \mbox{and} \ \ \ \lim_{t \to u^0(b_i)^+} \nu_u(t)=\nu_u(u^0(b_i)) \leq b_i.
		\end{equation}
		Hence, by definition of $\nu_u$ we have that $\mathcal{L}^d(\{ u^0(b_i) < u < u^0(a_i) \})=\nu_u(u^0(b_i))-\nu_u(u^0(a_i)-) \leq b_i-a_i$ for every $i=1,\dots,N$.
		For every non-negative measurable function $f$ and every Borel set $E$, by the Hardy-Littlewood inequality, it holds
		$$
		\int_E f(x) \, dx \leq \int_0^{\mathcal{L}^d(E)} f^*(t) \, dt.
		$$
		Combining this result with inequality \eqref{-11} gives
		\begin{align}\label{-12}
			\begin{split}
				\int_{\cup_{i=1}^N (u^0(b_i),u^0(a_i))} h_B(\nu_u(t)) \, dt & \leq 2\int_0^{\sum_{i=1}^N [\nu_u(u^0(b_i))-\nu_u(u^0(a_i))]} |\nabla u|^* \, dt \leq 2\int_0^{\sum_{i=1}^N (b_i-a_i)} |\nabla u|^* \, dt  \\ & \leq 2\int_0^\delta |\nabla u|^* \, dt \leq 2\varepsilon.
			\end{split}
		\end{align}
		Again by \eqref{-10} we get that $\nu_u(t) \in (\Lambda, \omega_d-\Lambda)$ for every $t \in \cup_{i=1}^N (a_i,b_i)$. Therefore, arguing as before, 
		$$
		h_B(\nu_u(t)) \geq 2\mathcal{H}^{d-1}(J_u)
		$$ 
		for every $t \in (a_i,b_i)$, $i=1,\dots,N$. Thus,
		$$
		2\mathcal{H}^{d-1}(J_u) \sum_{i=1}^N (u^0(a_i)-u^0(b_i)) \leq 2\varepsilon,
		$$
		and this gives the absolute continuity of $u^0$ on $(\Lambda,\omega_d-\Lambda)$ since the family of sub-intervals and $\varepsilon>0$ where arbitrary.
		
		\noindent \textbf{Step 3: Conclusion.} By the previous steps we have that $u^0$ is uniformly continuous and absolutely continuous in $[\Lambda,\omega_d-\Lambda]$ and $(\Lambda,\omega_d-\Lambda)$, respectively. Hence, by definition of $\tau',\tau''$, the function $\tau' \vee u^0 \wedge \tau''$ is absolutely continuous in $(0,\omega_d)$. Observe that by definition of $u^0$ and $T_B u$ (see \eqref{truncat operat}), we have that 
		$$
		\tau' \vee u^0 \wedge \tau'' = (\tau' \vee u \wedge \tau'')^0=(T_B u)^0.
		$$
		Let $h_B$ be as in Step 2, and let us set $v:=T_B u$. Recalling inequalities \eqref{-11} and \eqref{-12}, since $v^0$ is constant in $(0,\omega_d) \setminus (\Lambda,\omega_d -\Lambda)$, without any additional effort, we can deduce that for every family of disjoint intervals $\{(a_i,b_i)\}_{i=1}^N$ in $(0,\omega_d)$, it holds
		\begin{equation*}
			\int_{\cup_{i=1}^N (v^0(b_i),v^0(a_i))} h_B(\nu_v(t)) \, dt \leq 2\int_0^{\sum_{i=1}^N (b_i-a_i)} |\nabla u|^* \, dt.
		\end{equation*}
		Moreover, since $v^0$ is absolutely continuous on $(0,\omega_d)$, by a change of variables we get
		$$
		\int_{\cup_{i=1}^N (a_i,b_i)} -\frac{dv^0}{dr}(r) h_B(r) \, dr \leq 2\int_0^{\sum_{i=1}^N (b_i-a_i)} |\nabla u|^* \, dt,
		$$
		Observe that by approximation the inequality holds even if the union is countable.
		Since every open subset of $(0,\omega_d)$ is the union of disjoint open intervals contained in $(0,\omega_d)$, given $\ell \in (0,\omega_d)$, it holds
		$$
		\sup_{A \in \mathcal{A}((0,\omega_d)), \ \mathcal{L}^1(A)=\ell \ } \int_A -\frac{dv^0}{dr}(r) h_B(r) \, dr \leq 2\int_0^\ell |\nabla u|^* \, dt.
		$$
		In particular, taking $A=(0,\omega_d)$ we deduce that $-\frac{dv^0}{dr}(r) h_B(r) \in L^1((0,\omega_d))$. Thus, since every Borel set $E$ can be approximated arbitrarily well in the sense of the Lebesgue measure by open sets, by the absolute continuity of the integral we conclude that for every $\ell \in (0,\omega_d)$
		$$
		\sup_{E \in \mathcal{B}((0,\omega_d)), \ \mathcal{L}^1(E)=\ell \ } \int_E -\frac{dv^0}{dr}(r) h_B(r) \, dr \leq 2\int_0^\ell |\nabla u|^* \, dt.
		$$
		In particular, this implies that for every $\ell \in (0,\omega_d)$
		\begin{equation}\label{-13}
			\int_0^\ell \left(-\frac{dv^0}{dr}(r) h_B(r)\right)^* \, dr \leq \int_0^\ell |2\nabla u|^* \, dt.
		\end{equation}
		Finally, since $\psi \in \Phi_s$ is convex and \eqref{-13} holds for every $\ell \in (0,\omega_d)$, we can apply Lemma \ref{lem -3} to the functions $-\frac{dv^0}{dr} h_B$ and $|2\nabla u|^*$ to infer
		$$
		\int_0^{\omega_d} \psi \left(-\frac{dv^0}{dr}(r) h_B(r)\right) \, dr \leq \int_0^{\omega_d} \psi(|2\nabla u|^*(t)) \, dt=\int_0^{\omega_d} (\psi(|2\nabla u|))^*(t) \, dt,
		$$
		where we have also used that $\psi$ is strictly increasing. Using the equi-measurability of rearrangements and recalling the definition of $h_B$, we get \eqref{-7} with $c(d):=1/\gammaiso$. This concludes the proof.
	\end{proof}
	
	\subsection{Proof of Theorem \ref{POINCARE}.}\label{subs: proofPoincare} To prove Theorem \ref{POINCARE} we  need the following Hardy type inequality proven in \cite{CianchiFuscoCrelle}.
	
	\begin{lem}[{\cite[Lemma 2.2]{CianchiFuscoCrelle}}]\label{hardy}
		Let $\ell \in (0,+\infty)$. Let $\varphi \in \Phi_w$ be a finite valued Orlicz function satisfying \hyperlink{id}{$\textnormal{(Inc)}_1$}, set $\phi:=\varphi^{\frac{d}{d-1}}t^{-\frac{1}{d-1}}$. Then, there exists a constant $c_1:=c_1(d,\ell)>0$ such that for every measurable functions $h \colon (0,\ell) \to [0,+\infty)$ it holds
		$$
		\phi^{-1} \left( \int_0^\ell \phi \left( \int_s^\ell h(t) \, dt \right) \, ds \right) \leq \varphi^{-1} \left( \int_0^\ell \varphi(c_1 s^{\frac{d-1}{d}}h(s)) \, ds \right).
		$$
	\end{lem} 
	
	Using the definitions of $u^0$, $T_B u$ and $\med(u,B)$ we have:
	\begin{align}\label{-16}
		\begin{split}
			& u^0\left( \frac{\omega_d}{2} \right)=\med(u,B)=:m(u), \\
			&((T_B u-m(u))_+)^*(s)=((T_B u)^0(s)-m(u))_+ \ \ \ \mbox{for a.e. $s \in (0,\omega_d)$}, \\
			& ((T_B u - m(u))_-)^*(s)=((T_B u)^0(\omega_d-s)-m(u))_-=(m(u)-(T_B u)^0(\omega_d-s))_+ \ \ \ \mbox{for a.e. $s \in (0,\omega_d)$}.
		\end{split}
	\end{align}
	
	Now we are in a position to prove Theorem \ref{POINCARE}. In this proof we proceed similarly to the proof of \cite[Theorem 1.2]{CianchiFuscoCrelle}.
	
	\begin{proof}[Proof of Theorem \ref{POINCARE}] We begin by recalling that since $\varphi$ satisfies \hyperlink{id}{$\textnormal{(Inc)}_1$}, 
		\begin{equation}\label{-14}
			\lambda \varphi(t) \leq \varphi(\lambda t) \ \ \mbox{and} \ \ \varphi^{-1}(s+t) \leq 2(\varphi^{-1}(s)+\varphi^{-1}(t)) \ \ \ \mbox{for every $\lambda \geq 1$, $s,t \geq 0$}.
		\end{equation}
		The same properties hold for $\phi$ by the very same definition of $\phi$. Moreover, in virtue of Theorems \ref{weak and strong equiv} and \ref{inverse equiv}, we have that there exists $\psi \in \Phi_s$ such that
		\begin{equation}\label{-15}
			\psi(t) \leq \varphi(t) \leq \psi(2t) \ \ \mbox{and} \ \ \psi^{-1}(t) \leq 2\varphi^{-1}(t) \leq 2\psi^{-1}(t) \ \ \ \mbox{for every $t \geq 0$}.
		\end{equation}
		We have the following chain of inequalities for the left hand side term in \eqref{poinc ineq}
		\begin{align}
			\phi^{-1}  \left( \int_B \right. & \phi(|T_B u(x)  -m(u)|) \left. \, \vphantom{\int} dx \right) = \phi^{-1} \left( \int_B \phi((T_B u(x)-m(u))_+) \, dx + \int_B \phi((T_B u(x)-m(u))_-) \, dx  \right) \nonumber \\
			& \leq 2 \phi^{-1} \left( \int_0^{\omega_d/2} \phi((T_B u-m(u))^*_+(s)) \, ds \right) + 2 \phi^{-1} \left( \int_0^{\omega_d/2} \phi((T_B u-m(u))^*_-(s)) \, ds  \right) \nonumber  \\
			& = 2 \phi^{-1} \left( \int_0^{\omega_d/2} \phi\left( \int_s^{\omega_d/2} -\frac{d(T_B u)^0}{dr}(r) \, dr \right) \, ds \right)\label{-17} \\
			&  \ \ \ + 2 \phi^{-1} \left( \int_0^{\omega_d/2} \phi\left(\int_s^{\omega_d/2} -\frac{d(T_B u)^0}{dr}(\omega_d-r) \, dr \right) \, ds  \right)\nonumber ,
		\end{align}
		where in the first inequality we have used the equi-measurability of the rearrangement, the fact that $\phi$ is strictly increasing and \eqref{-14} for $\phi$, in the last equality we have used \eqref{-16}.
		
		Now we estimate the right hand side term. Let $c=c(d)$ be the constant from Proposition \ref{p1} and $c_1=c_1(d,\omega_d/2)=c_1(d)$ be the constant from Lemma \ref{hardy}. We have 
		\begin{align}\label{-18}
			\begin{split}
			\varphi^{-1} & \left( \int_B \varphi\left(4\frac{c_1}{c}|\nabla u(x)|\right) \, dx \right) \geq \frac{1}{2} \psi^{-1} \left( \int_B \psi\left(4\frac{c_1}{c}|\nabla u(x)|\right) \, dx \right) \\ 
			&\geq \frac{1}{2} \psi^{-1} \left( \int_0^{\omega_d} \psi \left(2 c_1 \min\{s,1-s\}^{\frac{d-1}{d}} \left(-\frac{d(T_B u)^0}{ds}(s) \right) \right) \, ds \right) \\
			& \geq \frac{1}{4} \psi^{-1} \left( \int_0^{\omega_d/2} \psi \left( 2c_1 s^{\frac{d-1}{d}} \left(-\frac{d(T_B u)^0}{ds}(s) \right) \right) \, ds \right) \\
			& \ \ \ +\frac{1}{4} \psi^{-1} \left( \int_0^{\omega_d/2} \psi \left(2 c_1 s^{\frac{d-1}{d}} \left(-\frac{d(T_B u)^0}{ds}(\omega_d-s) \right) \right) \, ds \right) \\
			& \geq \frac{1}{4} \varphi^{-1} \left( \int_0^{\omega_d/2} \varphi \left( c_1  s^{\frac{d-1}{d}} \left(-\frac{d(T_B u)^0}{ds}(s) \right) \right) \, ds \right) \\
			& \ \ \ +\frac{1}{4} \varphi^{-1} \left( \int_0^{\omega_d/2} \varphi \left( c_1  s^{\frac{d-1}{d}} \left(-\frac{d(T_B u)^0}{ds}(\omega_d-s) \right) \right) \, ds \right),
			\end{split}
		\end{align}
		where we have used in this order: \eqref{-15}, Proposition \ref{p1}, the fact that $\psi^{-1}$ is increasing and, for the last inequality, we used \eqref{-15} together with \eqref{-14}. 
		
		Finally, combining \eqref{-17} and \eqref{-18} and using Lemma \ref{hardy}, we infer that
		$$
		8 \varphi^{-1} \left( \int_B \varphi\left(\frac{4c_1}{c}|\nabla u(x)|\right) \, dx \right) \geq 	\phi^{-1}  \left( \int_B \phi(|T_B u(x)  -m(u)|) \,  dx \right),
		$$
		which is \eqref{poinc ineq} up to taking $C=C(d)$ larger in order to obtain an averaged integral in the expression above. In view of Remark \ref{rescaling rem} we thus conclude.
	\end{proof}
	
	
	We will also make use of  the following weaker version of the Poincaré inequality in \eqref{poinc ineq} for functions which are also doubling.
	
	\begin{thm}\label{Poinc homogeneous form}
		Let $\varphi \in \Phi_w$ be an Orlicz function satisfying \hyperlink{id}{$\textnormal{(Inc)}_1$} and \hyperlink{id}{\textnormal{(aDec)}}. Then, there exists $C=C(d,K)>0$, where $K$ is the doubling constant of $\varphi$, such that for every ball $B_r(x) \subset \R^d$ and every $u \in \sbv(B_r(x))$ satisfying \eqref{poinc small jump}, the following inequality holds
		\begin{equation}\label{poinc ineq homog}
			\int_{B_r(x)} \varphi\left(\frac{|T_{B_r(x)}u(y)-\med(u,B_r(x))|}{r}\right) \, dy  \leq C \int_{B_r(x)} \varphi(|\nabla u(y)|) \, dy.
		\end{equation}
	\end{thm}
	
	In order to prove Theorem \ref{Poinc homogeneous form} we need the following elementary Lemma.
	
	\begin{lem}\label{orlicz hierarchy}
		Let $\varphi,\phi \in \Phi_s$ be two finite valued Orlicz functions such that
		\begin{equation}\label{-1000}
			\lim_{t \to 0^+} \phi(\varphi^{-1}(t))=0, \ \ \ t \mapsto \frac{\phi(t)}{\varphi(t)} \ \ \mbox{is increasing}.
		\end{equation}
		Then, for every ball $B_r \subset \R^d$ and every measurable function $u \in L^1(B_r)$ we have that
		\begin{equation}\label{-2000}
			\varphi^{-1} \left( \fint_{B_r} \varphi(|u(x)|)\,dx \right) \leq \phi^{-1} \left( \fint_{B_r} \phi(2|u(x)|)\,dx \right) 
		\end{equation}
	\end{lem}
	
	\begin{proof}
		Let us set $\psi:=\phi(\varphi^{-1})$. We claim that $\psi \in \Phi_w$ satisfies \hyperlink{id}{$\textnormal{(Inc)}_1$}. Indeed, since $\varphi$ is continuous, convex and finite valued, then it is a bijection, thus, $\varphi^{-1}$ is its inverse which is continuous. This gives that $\psi$ is continuous and finite valued on $(0,+\infty)$.  Moreover, by the assumptions in \eqref{-1000} and definition of $\psi$ we have that $\psi(0)=0=\lim_{t \to 0^+} \psi(t)$ and $t \mapsto \psi(t)/t$ is increasing. Therefore, by Theorem \ref{weak and strong equiv}, there exists $\tilde{\psi} \in \Phi_s$ such that
		$$
		\tilde{\psi}(t) \leq \psi(t) \leq \tilde{\psi}(2t) \ \ \ t \geq 0.
		$$
		Since $\tilde{\psi}$ is convex, by Jensen inequality we get for every $v \in L^1(B_r)$
		$$
		\psi \left(  \fint_{B_r} |v(x)| \, dx \right) \leq \tilde{\psi} \left( 2 \fint_{B_r} |v(x)| \, dx \right) \leq \fint_{B_r} \tilde{\psi}(2|v(x)|) \, dx \leq \fint_{B_r} \psi(2|v(x)|) \, dx.
		$$
		By substituting back the expression of $\psi$ with $v(x)=\varphi(|u(x)|)$ and recalling that $\varphi^{-1}(2t) \leq 2\varphi^{-1}(t)$ we get \eqref{-2000}.
	\end{proof}
	
	\begin{proof}[Proof of Theorem \ref{Poinc homogeneous form}]
		By translation invariance it is not restrictive to consider $x=0$. Using Proposition \ref{weak and strong equiv} we can find $\tilde{\varphi} \in \Phi_s$ such that $\tilde{\varphi}(t) \leq \varphi(t) \leq \tilde{\varphi}(2t)$. Setting $\tilde{\phi}:=\tilde{\varphi}^{\frac{d}{d-1}}t^{-\frac{1}{d-1}}$ we have that $\tilde{\varphi}$ and $\tilde{\phi}$ satisfy all the assumptions of Lemma \ref{orlicz hierarchy}. Using Lemma \ref{orlicz hierarchy} we thus infer
		$$
		\tilde{\varphi}^{-1} \left( \fint_{B_r} \tilde{\varphi} \left( \frac{|T_{B_r} u(x)-\med(u,B_r)|}{r} \right) \, dx \right) \leq  \tilde{\phi}^{-1} \left( \fint_{B_r} \tilde{\phi} \left( 2 \frac{|T_{B_r} u(x)-\med(u,B_r)|}{r} \right) \, dx \right);
		$$
		which gives
		$$
		\varphi^{-1} \left( \fint_{B_r} \varphi \left( \frac{|T_{B_r} u(x)-\med(u,B_r)|}{r} \right) \, dx \right) \leq 2 \phi^{-1} \left( \fint_{B_r} \phi \left( 8 \frac{|T_{B_r} u(x)-\med(u,B_r)|}{r} \right) \, dx \right).
		$$
		By Theorem \ref{POINCARE} this implies that for every $u \in \sbv(B_r)$ satisfying \eqref{poinc small jump} then
		$$
		\varphi^{-1} \left( \fint_{B_r} \varphi \left( \frac{|T_{B_r} u(x)-\med(u,B_r)|}{r} \right) \, dx \right) \leq C \varphi^{-1} \left( \fint_{B_r} \varphi(C|\nabla u(x)|) \, dx \right)
		$$
		with $C=C(d)$. Using the fact that $\varphi$ is doubling and $\varphi^{-1}$ is increasing, we get \eqref{Poinc homogeneous form} with $C=C(d,K)$ where $K \geq 2$ is the doubling constant of $\varphi$.
	\end{proof}
	
	\subsection{Poincaré inequality for generalized Orlicz spaces} \label{subs: poincaregeneralized}
	We now extend the Poincaré inequality for $\sbv(B_r(x))$ functions with small jump set to the cases where $\varphi \in \Phi_s(B_r(x))$ is a generalized Orlicz function satisfying \hyperlink{A0}{(A0)}, \hyperlink{adA1}{(adA1)}, \hyperlink{idg}{$\textnormal{(Inc)}_1$} and \hyperlink{idg}{$\textnormal{(Dec)}$} on $B_r(x)$. Notice that the statement below requires the introduction of a further truncation, depending on $\varphi$ and on the ball $B_r(x)$. In order to ease the notation, we only consider the case $x=0$.

	\begin{thm}\label{POINCARE GEN}
		Let $\varphi \in \Phi_s(B_r)$ be a generalized Orlicz function satisfying \hyperlink{A0}{\textnormal{(A0)}}, \hyperlink{adA1}{\textnormal{(adA1)}}, \hyperlink{idg}{$\textnormal{(Inc)}_1$} and \hyperlink{idg}{$\textnormal{(Dec)}$} on $B_r$. Let $\phi(x,t):=\varphi(x,t)^{\frac{d}{d-1}}t^{-\frac{1}{d-1}} \in \Phi_s(B_r)$. Moreover, let $K$ be the doubling constant of $\varphi$, $\sigma$ be the constant from \hyperlink{A0}{\textnormal{(A0)}} and $\beta$ be the constant appearing in \hyperlink{adA1}{\textnormal{(adA1)}}. For $u \in \sbv^\varphi(B_r)$, set 
\begin{equation}\label{other truncat op}
		\mathfrak{u}^{\varphi}_r(x):= \left(\med(u,B_r)-r(\varphi^-_{B_r})^{-1}\left(\frac{1}{2r}\right)\right) \vee T_{B_r}u(x) \wedge  \left(\med(u,B_r)+r(\varphi_{B_r}^-)^{-1}\left(\frac{1}{2r}\right)\right).
	\end{equation}
Then, there exists $C=C(d,K,\beta)$ such that for every $u \in \sbv^\varphi(B_r)$ satisfying \eqref{poinc small jump}, the following inequalities hold
		\begin{align}
				&\label{P1}  (\phi^-_{B_r})^{-1} \left( \fint_{B_r}  \phi_{B_r}^+ \left(\left|\frac{\mathfrak{u}^{\varphi}_r-\med(u,B_r)}{r} \right|\right) \, dx \right) \leq C (\varphi^-_{B_r})^{-1} \left( \fint_{B_r} \varphi_{B_r}^-(|\nabla u|) \, dx \right)+C\sigma, \\
				&\label{P2} \int_{B_r}  \varphi_{B_r}^+ \left( \left|\frac{\mathfrak{u}^{\varphi}_r-\med(u,B_r)}{r} \right| \right) \, dx \leq C \int_{B_r} \left( \varphi_{B_r}^-(|\nabla u|)+ \varphi^+_{B_r}\left(\left|\frac{u-\med(u,B_r)}{r}\right| \wedge \sigma\right) \right) \, dx . 
		\end{align}
	\end{thm}
	
	\begin{proof}
		By assumption we have that $\varphi^-_{B_r}$ satisfies  \hyperlink{idg}{$\textnormal{(Inc)}_1$}. Therefore inequality \eqref{poinc ineq} hold with $\varphi$ replaced by $\varphi^-_{B_r}$ and $\phi$ replaced by $\phi^-_{B_r}$. By definition of $\mathfrak{u}^{\varphi}_r$,
		we have that 
		\begin{equation}\label{,,,,}
		\frac{| \mathfrak{u}^{\varphi}_r(x)-\med(u,B_r)|}{r} \leq (\varphi^-_{B_r})^{-1} \left( \frac{1}{2r} \right).
		\end{equation}
		Hence, observing that $\phi_{B_r}^-$ is increasing and 
		$$
		\phi_{B_r}^+(\beta t) = (\varphi^+_{B_r}(\beta t))^{\frac{d}{d-1}}(\beta t)^{-\frac{1}{d-1}} \leq \frac{1}{\beta} (\varphi^-_B( t))^{\frac{d}{d-1}} t^{-\frac{1}{d-1}}=\frac{1}{\beta} \phi_{B_r}^-(t) \ \ \ \mbox{for } t \in \left(\sigma,(\varphi_{B_r}^-)^{-1} \left(\frac{1}{2r}\right)\right),
		$$
		inequality \eqref{P1} is a consequence of \hyperlink{adA1}{\textnormal{(adA1)}}. 
		
		Analogously, we have that inequality \eqref{P2} is a consequence of the fact that \eqref{poinc ineq homog} holds with $\varphi$ replaced by $\varphi_{B_r}^-$, \eqref{,,,,} and \hyperlink{adA1}{\textnormal{(adA1)}}.
	\end{proof}

	\section{Integral Representation}
	\label{s:integral-repr}

	This section is devoted to the proof of Theorem \ref{int rep teo} and its Corollary \ref{trans invar rep}. 
	Throughout this section, $\psi$ stands for a function in $\Phi_s(\Omega)$ satisfying \hyperlink{A0}{$\textnormal{(A0)}$}, \hyperlink{adA1}{$\textnormal{(adA1)}$}, \hyperlink{idg}{$\textnormal{(Inc)}$} and \hyperlink{idg}{$\textnormal{(Dec)}$} on $\Omega$. Notice that the additional requirements that $\psi \in \Phi_s(\Omega)$ instead of $\Phi_w(\Omega)$ and that \hyperlink{idg}{$\textnormal{(Inc)}$} and \hyperlink{idg}{$\textnormal{(Dec)}$} replace \hyperlink{idg}{$\textnormal{(aInc)}$} and \hyperlink{idg}{$\textnormal{(aDec)}$}, respectively, cause  no restriction, as we discussed in Remark \ref{rem: convexgratis}.

	\subsection{Fundamental estimate in $\gsbv^{\psi}$} 
	We begin by proving an important tool for integral representation, that is the fundamental estimate.

	\begin{lem}[Fundamental estimate in $\gsbv^{\psi}$] \label{fund est}
		Let $\eta>0$ and $B \Subset \Omega$ a ball, let $D',D'',E \in \mathcal{A}(B)$ with $D' \Subset D''$. For every functional $\mathcal{F}$ satisfying \eqref{H1}, \eqref{H3} and \eqref{H4} with $\psi$, for every $u \in \gsbv^{\psi}(D'',\R^m)$ and every $v \in \gsbv^{\psi}(E,\R^m)$ there exists a function $\varphi \in C^\infty(\R^d,[0,1])$ such that the function $w:=\varphi u+(1-\varphi)v \in \gsbv^{\psi}(E \cup D',\R^m)$ satisfies 
		\begin{align*}
			\begin{split}
				&\textnormal{(i)} \mbox{ $w=u$ on $D'$ and $w=v$ on $E \setminus D''$}; \\
				&\textnormal{(ii)} \mbox{ $\mathcal{F}(w,D' \cup E) \leq (1+\eta)(\mathcal{F}(u,D'')+\mathcal{F}(v,E))+M\displaystyle \int_F \psi\left(x, \displaystyle \frac{|u-v|}{\delta} \right) \, dx + \eta \mathcal{L}^d(D' \cup E)$,}
			\end{split}
		\end{align*}
		with $\delta:=\frac{1}{2}\dist(D',\partial D'')$, where $F:=(D'' \setminus D') \cap E$ and $M=M(D',D'',E,\psi,\eta)>0$ is independent of $u$ and $v$. Moreover, given $\varepsilon>0$ and $x_0 \in \R^d$ such that $D_{\varepsilon,x_0}',D_{\varepsilon,x_0}'',E_{\varepsilon,x_0} \Subset \Omega$, then
		$$
		M(D_{\varepsilon,x_0}',D_{\varepsilon,x_0}'',E_{\varepsilon,x_0},\psi,\eta)=M(D',D'',E,\psi,\eta)
		$$
		and the integral term in \textnormal{(ii)} becomes: $M \displaystyle \int_{F_{\varepsilon,x_0}}\psi\left(x, \displaystyle \frac{|u-v|}{\varepsilon\delta} \right) \, dx$.
	\end{lem}
	
	\begin{proof}
		Let $K$ be the doubling constant of $\psi$, take $k \in \N$ such that $k \geq \max \left\{ \frac{b(K+1)^2}{a\eta},\frac{b}{\eta} \right\}$. We set $D_1:=D'$ and, for $i=1,\dots,k$,
		$$
		D_{i+1}:=\left\{ x \in D'' \colon \dist(x,D') < \frac{i\delta}{k} \right\}.
		$$
		We thus have $D'=:D_1 \Subset D_2 \Subset \dots \Subset D_{k+1} \Subset D''$. Let $\varphi_i \in C^{\infty}_0(D_{i+1},[0,1])$ with $\varphi_i=1$ in a neighborhood $U_i$ of $\overline{D_i}$ and such that $\Vert \nabla \varphi_i \Vert_{L^\infty} \leq \frac{2k}{\delta}$. 
		
		Let $u \in \gsbv^{\psi}(D'',\R^m)$ and $v \in \gsbv^{\psi}(E,\R^m)$ such that $u-v \in L^{\psi}(F,\R^m)$, otherwise the thesis is trivial. We define the function $w_i:=\varphi_i u + (1-\varphi_i)v$. We notice that $w_i \in \gsbv^{\psi}(D' \cup E,\R^m)$, this follows from properties \hyperlink{A0}{(A0)} and \hyperlink{idg}{(Dec)} of $\psi$, with $u$ and $v$ extended arbitrarily outside $D''$ and $E$ respectively. Set $I_i:=E \cap (D_{i+1} \setminus \overline{D_i})$. Using \eqref{H1} and \eqref{H3} we infer
		\begin{align}\label{fe.1}
			\begin{split}
				\mathcal{F}(w_i,D' \cup E) & \leq \mathcal{F}(u,(D' \cup E) \cap U_i)+\mathcal{F}(v,E \setminus \supp \varphi_i)+\mathcal{F}(w_i,I_i) \\
				& \leq \mathcal{F}(u,D'')+\mathcal{F}(v,E)+\mathcal{F}(w_i,I_i).
			\end{split}
		\end{align}
		We now estimate the last term of \eqref{fe.1} using assumption \eqref{H4} and properties \hyperlink{A0}{(A0)}, \hyperlink{idg}{(Dec)} (i.e. doubling) of $\psi$ and the fact that $\psi$ is convex since it belongs to the class $\Phi_s$. We have
		\begin{align*}
			\mathcal{F}(w_i,I_i) & \leq b \mathcal{L}^d(I_i)+b(K+1)\int_{I_i} \! \psi(x,|\nabla u+\nabla v|) dx+b(K+1)\int_{I_i} \psi(x,|(u-v) \otimes\nabla \varphi|) dx \\
			& \ \ \ \ +b\mathcal{H}^{d-1}(J_u \cap I_i)+b \mathcal{H}^{d-1}(J_v \cap I_i) \\
			& \leq  b \mathcal{L}^d(I_i)+b(K+1)^2\int_{I_i} \psi(x,|\nabla u|) \, dx +b(K+1)^2\int_{I_i} \psi(x,|\nabla v|) \, dx \\
			& \ \ \ \ +b(K+1)K^{\log_2(k)+2}\int_{I_i} \psi\left(x,\frac{|u-v|}{\delta}\right) dx+b\mathcal{H}^{d-1}(J_u \cap I_i)+b \mathcal{H}^{d-1}(J_v \cap I_i) \\
			& \leq \frac{b}{a}(K+1)^2 (\mathcal{F}(u,I_i)+\mathcal{F}(v,I_i))+b(K+1)K^{\log_2(k)+2}\int_{I_i} \psi\left(x,\frac{|u-v|}{\delta}\right) dx+b\mathcal{L}^d(I_i).
		\end{align*}
		By our initial choice of $k$ and using \eqref{H1} we can find $i_0 \in \{1,\dots,k\}$ such that
		$$
		\mathcal{F}(w_{i_0},I_{i_0}) \leq \frac{1}{k} \sum_{i=1}^k \mathcal{F}(w_i,I_i) \leq \eta (\mathcal{F}(u,D'')+\mathcal{F}(v,E))+\eta \mathcal{L}^d(D' \cup E)+M \int_F \psi\left(x, \displaystyle \frac{|u-v|}{\delta} \right) dx
		$$
		with $M:=b(K+1)K^{\log_2(k)+2}k^{-1}$. This combined with \eqref{fe.1} concludes the proof by setting $w:=w_{i_0}$. For the second statement it is enough to consider cut-off functions $\varphi_i^\varepsilon \in C^{\infty}_0((D_{i+1})_{\varepsilon,x_0},[0,1])$ defined as $\varphi_i^\varepsilon(x):=\varphi_i(x_0+(x-x_0)/\varepsilon)$, for $i=1,\dots,k$ and proceed as before.
	\end{proof}
	
	\subsection{Proof of the integral representation theorem} Now we proceed with the proof of Theorem \ref{int rep teo}. We begin by showing that the Radon-Nikodym derivatives of $\mathcal{F}$ and $\textbf{m}_\mathcal{F}$ with respect to the measure 
	\begin{equation}\label{mu}
		\mu:=\mathcal{L}^d|_\Omega+\mathcal{H}^{d-1}|_{J_u \cap \Omega}
	\end{equation}
	coincide. To this aim, we state the following Lemma.
	
	\begin{lem}\label{lem 1}
		Let $\mathcal{F}$ satisfy \eqref{H1}--\eqref{H4}, let $u \in \gsbv^{\psi}(\Omega,\R^m)$ and let $\mu$ be as in \eqref{mu}. Then, for $\mu$-a.e. $x_0 \in \Omega$ it holds
		$$
		\lim_{\varepsilon \to 0^+} \frac{\mathcal{F}(u,B_\varepsilon(x_0))}{\mu(B_\varepsilon(x_0))}=\lim_{\varepsilon \to 0^+} \frac{\mathbf{m}_\mathcal{F}(u,B_\varepsilon(x_0))}{\mu(B_\varepsilon(x_0))}.
		$$
	\end{lem}
	We postpone the proof of this lemma at the end of this section. As a second step we prove that asymptotically as $\varepsilon \to 0^+$ the two minimization problems $\mathbf{m}_\mathcal{F}(u,B_\varepsilon(x_0))$ and $\mathbf{m}_\mathcal{F}(\overline{u}^\textnormal{bulk}_{x_0},B_\varepsilon(x_0))$ coincide for $\mathcal{L}^d$-a.e. $x_0 \in \Omega$, where $\overline{u}^\textnormal{bulk}_{x_0}:= \ell_{x_0,u(x_0),\nabla u(x_0)}$.
	
	\begin{lem}\label{lem 2}
		Let $\mathcal{F}$ satisfy \eqref{H1} and \eqref{H3}-\eqref{H4} and let $u \in \gsbv^{\psi}(\Omega,\R^m)$. Then, for $\mathcal{L}^d$-a.e. $x_0 \in \Omega$ we have
		\begin{equation}\label{vol pts equality}
			\lim_{\varepsilon \to 0^+} \frac{\mathbf{m}_\mathcal{F}(u,B_\varepsilon(x_0))}{\omega_d \varepsilon^d}=\limsup_{\varepsilon \to 0^+} \frac{\mathbf{m}_\mathcal{F}(\overline{u}^\textnormal{bulk}_{x_0},B_\varepsilon(x_0))}{\omega_d \varepsilon^d}.
		\end{equation}
	\end{lem}
	We defer the proof of this lemma to Section \ref{secti vol}. As a third and final step we prove that asymptotically as $\varepsilon \to 0^+$ the two minimization problems $\mathbf{m}_\mathcal{F}(u,B_\varepsilon(x_0))$ and $\mathbf{m}_\mathcal{F}(\overline{u}^\textnormal{surf}_{x_0},B_\varepsilon(x_0))$ coincide for $\mathcal{H}^{d-1}$-a.e. $x_0 \in J_u$, where $\overline{u}^\textnormal{surf}_{x_0}:= u_{x_0,u^+(x_0),u^-(x_0),\nu_u(x_0)}$. 
	
	\begin{lem}\label{lem 3}
		Let $\mathcal{F}$ satisfy \eqref{H1} and \eqref{H3}-\eqref{H4} and let $u \in \gsbv^{\psi}(\Omega,\R^m)$. Then for $\mathcal{L}^d$-a.e. $x_0 \in \Omega$ we have
		\begin{equation}\label{surf pts equality}
			\lim_{\varepsilon \to 0^+} \frac{\mathbf{m}_\mathcal{F}(u,B_\varepsilon(x_0))}{\omega_{d-1} \varepsilon^{d-1}}=\limsup_{\varepsilon \to 0^+} \frac{\mathbf{m}_\mathcal{F}(\overline{u}^\textnormal{surf}_{x_0},B_\varepsilon(x_0))}{\omega_{d-1} \varepsilon^{d-1}}
		\end{equation}
	\end{lem}
	We defer the proof of this lemma to Section \ref{secti jump}. We are now ready to prove Theorem \ref{int rep teo} and Corollary \ref{trans invar rep}.
	
	\begin{proof}[Proof of Theorem \ref{int rep teo}] 
		In view of the Besicovitch derivation theorem, we need to prove that
		\begin{align}
			&\frac{\textnormal{d}\mathcal{F}(u,\cdot)}{\textnormal{d}\mathcal{L}^d}(x_0)=f(x_0,u(x_0),\nabla u(x_0)), \ \ \ \mbox{for $\mathcal{L}^d$-a.e. $x_0 \in \Omega$}, \label{0.1} \\
			&\frac{\textnormal{d}\mathcal{F}(u,\cdot)}{\textnormal{d}\mathcal{H}^{d-1}|_{J_u}}(x_0)=g(x_0,u^+(x_0),u^-(x_0),\nu_u(x_0)), \ \ \ \mbox{for $\mathcal{H}^{d-1}$-a.e. $x_0 \in J_u$}, \label{0.2}
		\end{align}
		where $f$ and $g$ are defined in \eqref{f} and \eqref{g}, respectively.
		
		Using Lemma \ref{lem 1} and that $\lim_{\varepsilon \to 0^+}(\omega_d \varepsilon^d)^{-1} \mu(B_\varepsilon(x_0))=1$ for $\mathcal{L}^d$-a.e. $x_0 \in \Omega$, we infer that for $\mathcal{L}^{d}$-a.e. $x_0 \in \Omega$
		$$
		\frac{\textnormal{d}\mathcal{F}(u,\cdot)}{\textnormal{d}\mathcal{L}^d}(x_0)=\lim_{\varepsilon \to 0^+} \frac{\mathcal{F}(u,B_\varepsilon(x_0))}{\mu(B_\varepsilon(x_0))}=\lim_{\varepsilon \to 0^+} \frac{\mathbf{m}_\mathcal{F}(u,B_\varepsilon(x_0))}{\mu(B_\varepsilon(x_0))}
		=\lim_{\varepsilon \to 0^+} \frac{\mathbf{m}_\mathcal{F}(u,B_\varepsilon(x_0))}{\omega_d \varepsilon^d} < +\infty.
		$$
		Then, \eqref{0.1} follows by definition of $f$ (\eqref{f}) and Lemma \ref{lem 2}.
		
		By Lemma \ref{lem 1} and the fact that $\lim_{\varepsilon \to 0^+}(\omega_{d-1} \varepsilon^{d-1})^{-1} \mu(B_\varepsilon(x_0))=1$ for $\mathcal{H}^{d-1}$-a.e. $x_0 \in J_u$ we infer that for $\mathcal{H}^{d-1}$-a.e. $x_0 \in J_u$
		$$
		\frac{\textnormal{d}\mathcal{F}(u,\cdot)}{\textnormal{d}\mathcal{H}^{d-1}|_{J_u}}(x_0)=\lim_{\varepsilon \to 0^+} \frac{\mathcal{F}(u,B_\varepsilon(x_0))}{\mu(B_\varepsilon(x_0))}=\lim_{\varepsilon \to 0^+} \frac{\mathbf{m}_\mathcal{F}(u,B_\varepsilon(x_0))}{\mu(B_\varepsilon(x_0))}
		=\lim_{\varepsilon \to 0^+} \frac{\mathbf{m}_\mathcal{F}(u,B_\varepsilon(x_0))}{\omega_{d-1} \varepsilon^{d-1}} < +\infty,
		$$
		and \eqref{0.2} follows by definition of $g$ (\eqref{g}) and Lemma \ref{lem 3}.
	\end{proof}
	
	\begin{proof}[Proof of Corollary \ref{trans invar rep}]
		In view of Theorem \ref{int rep teo} we have that for all $u \in \gsbv^{\psi}(\Omega,\R^m)$ and all $A \in \mathcal{A}(\Omega)$
		$$
		\mathcal{F}(u,A)=\int_A \tilde{f}(x,u(x),\nabla u(x))\, dx+ \int_{J_u \cap A} \tilde{g}(x,u^+(x),u^-(x),\nu_u(x)) \, d\mathcal{H}^{d-1},
		$$
		where $\tilde{f}$ and $\tilde{g}$ are defined in \eqref{f} and \eqref{g}, respectively. Recalling the definition of $\mathbf{m}_\mathcal{F}$, since $\mathcal{F}$ satisfies assumption \eqref{H5}, by definition of $\tilde{f}$ and $\tilde{g}$ we deduce that for every $c \in \R^m$
		\begin{align*}
			\tilde{f}(x,u(x)+c,\nabla u(x))=\tilde{f}(x,0,\nabla u(x)), \ \ \ \ \ \tilde{g}(x,u^+(x)+c,u^-(x)+c,\nu_u(x))=\tilde{g}(x,[u](x),0,\nu_u(x)).
		\end{align*}
		Hence, $\tilde{f}(x,u,\nabla u)=:f(x,\nabla u)$ and $\tilde{g}(x,u^+(x),u^-(x),\nu_u(x))=:g(x,[u](x),\nu_u(x))$.
	\end{proof}
	
	In the remaining part of the section we prove Lemma \ref{lem 1} following the lines of \cite{BFLMrelaxSBVp}. We start by fixing some notations. Given $\delta>0$ and $A \in \mathcal{A}(\Omega)$ we define
	\begin{align}\label{emmedelta}
		\mathbf{m}_\mathcal{F}^\delta(u,A) :=\inf \left\{ \sum_{i=1}^\infty \mathbf{m}_\mathcal{F}(u,B_i) \colon B_i \subset A \  \mbox{pairwise disjoint balls},  \ \diam(B_i)<\delta, \ \mu \left( A \setminus \cup_{i=1}^\infty B_i \right)=0 \right\}
	\end{align}
	where $\mu$ is the measure introduced in \eqref{mu}. Since $\mathbf{m}_\mathcal{F}^\delta(u,A)$ is decreasing in $\delta$ we can also introduce
	\begin{equation}\label{m star}
		\mathbf{m}_\mathcal{F}^*(u,A):=\lim_{\delta \to 0^+} \mathbf{m}_\mathcal{F}^\delta(u,A).
	\end{equation}
	In the next lemma we prove that under assumptions \eqref{H1}--\eqref{H4} $\mathcal{F}$ and $\mathbf{m}_\mathcal{F}^*$ coincide.
	
	\begin{lem}\label{lem 4 for lem 1}
		Let $\mathcal{F}$ satisfy \eqref{H1}--\eqref{H4} and let $u \in \gsbv^{\psi}(\Omega,\R^m)$. Then, for all $A \in \mathcal{A}(\Omega)$ we have $\mathcal{F}(u,A)=\mathbf{m}_\mathcal{F}^*(u,A)$.
	\end{lem}
	
	\begin{proof}
		We mainly follow the lines of \cite[Lemma 3.3]{BFMrelaxglobal}. We start by proving the inequality $\mathcal{F}(u,A) \geq \mathbf{m}_\mathcal{F}^*(u,A)$. For every ball $B \subset A$ we have that $\mathbf{m}_\mathcal{F}(u,B) \leq \mathcal{F}(u,B)$ by definition. Using \eqref{H1} we infer $\mathcal{F}(u,A) \geq \mathbf{m}_\mathcal{F}^\delta(u,A)$ for all $\delta>0$. Thus the desired inequality follows by definition of $\mathbf{m}_\mathcal{F}^*$ in \eqref{m star}. 
		
		We now prove the reverse inequality. Fix $\delta>0$ and $A \in \mathcal{A}(\Omega)$ and let $\{ B_i^\delta \}_i$ be a family of balls as in \eqref{emmedelta} such that
		\begin{equation}\label{2.1}
			\sum_{i=1}^\infty \mathbf{m}_\mathcal{F}(u,B^\delta_i) \leq \mathbf{m}^\delta_\mathcal{F}(u,A)+\delta.
		\end{equation}
		By definition of $\mathbf{m}_\mathcal{F}$ we can find $v_i^\delta \in \gsbv^{\psi}(B_i^\delta,\R^m)$ such that $v_i^\delta=u$ in a neighborhood of $\partial B_i^\delta$ and
		\begin{equation}\label{2.2}
			\mathcal{F}(v_i^\delta,B_i^\delta) \leq \mathbf{m}_\mathcal{F}(u,B_i^\delta)+\delta\mathcal{L}^d(B_i^\delta).
		\end{equation}
		We define
		\begin{equation}\label{2.3}
			v^{\delta,n}:=\sum_{i=1}^n v_i^\delta \chi_{B_i^\delta}+u \chi_{N_0^{\delta,n}} \ \ n \in \N, \ \ \ \ v^{\delta}:=\sum_{i=1}^\infty v_i^\delta \chi_{B_i^\delta}+u \chi_{N_0^{\delta}},
		\end{equation}
		where $N_0^{\delta,n}:= \Omega \setminus \cup_{i=1}^n B_i^\delta$ and $N_0^{\delta}:= \Omega \setminus \cup_{i=1}^\infty B_i^\delta$. By construction, we have that $v^{\delta,n} \in \gsbv^{\psi}(\Omega,\R^m)$ for every $n \in \N$ and, using \eqref{2.1}, \eqref{2.2} and \eqref{H4}, 
		\begin{equation}\label{2.4}
			\sup_{n \in N} \left( \int_\Omega \psi\left(x,|\nabla v^{\delta,n}(x)|\right) \, dx +\mathcal{H}^{d-1}(J_{v^{\delta,n}}) \right) < +\infty.
		\end{equation}
		Moreover, we have that $v^{\delta,n} \to v^\delta$ pointwise $\mathcal{L}^d$-a.e. and in measure on $\Omega$ by construction. Recalling that $\psi$ satisfies \hyperlink{idg}{$\textnormal{(Inc)}_\gamma$} for some $\gamma>1$, using \cite[Theorem 2.2]{AmbrosioExistence} (see also \cite{AmbrosioCompact1} and \cite{AmbrosioCompact2}) together with the compactness in $L^0$ of $\{ v^{\delta,n} \}_n$ gives that $v^\delta \in \gsbv^\gamma(\Omega,\R^m)$ and $\nabla v^{\delta ,n} \rightharpoonup \nabla v^\delta$ weakly in $L^\gamma$. Now using Ioffe's theorem and \eqref{2.4}, we infer that
		$$
		\int_\Omega \psi \left(x,|\nabla v^\delta(x)| \right) \, dx \leq \liminf_{n \to +\infty} \int_\Omega \psi \left(x,|\nabla v^{\delta,n} (x)| \right) \, dx < +\infty.
		$$
		Therefore, $v^\delta \in \gsbv^{\psi}(\Omega,\R^m)$. Observe that it holds
		\begin{align}\label{2.5}
			\begin{split}
				\mathcal{F}(v^\delta,A) & = \sum_{i=1}^\infty \mathcal{F}(v^\delta_i,B^\delta_i)+\mathcal{F}(u,A \cap N_0^\delta) \leq \sum_{i=1}^\infty (\mathbf{m}_\mathcal{F}(u,B_i^\delta)+\delta \mathcal{L}^d(B_i^\delta)) \\
				& \leq \mathbf{m}_\mathcal{F}^\delta(u,A)+\delta(1+\mathcal{L}^d(A)),
			\end{split}
		\end{align} 
		where we have used that $\mu(A \cap N_0^\delta)=\mathcal{F}(u,A \cap N_0^\delta)=0$ by definition of $\{ B_i^\delta \}_i$ and \eqref{H4}. Thanks to \hyperlink{idg}{$\textnormal{(Inc)}_\gamma$} we have $\psi^+_\Omega(1)t^\gamma \leq \psi(x,t)$ for $\mathcal{L}^d$-a.e. $x \in \Omega$ and every $t \geq 1$. Therefore, \eqref{2.5} together with \eqref{H4} implies that
		\begin{equation}\label{2.6}
			\int_A |\nabla v^\delta|^\gamma \, dx  + \mathcal{H}^{d-1}(J_{v^\delta} \cap A) \leq \frac{\psi_\Omega^+(1)}{a} \left(\mathbf{m}_\mathcal{F}^\delta(u,A)+\delta(1+\mathcal{L}^d(A))\right)+\mathcal{L}^d(A).
		\end{equation}
		We now claim that 
		\begin{equation}\label{2.7}
			w^\delta:=u-v^\delta \to 0 \ \ \mbox{in measure on $A$}.
		\end{equation}
		Notice that if \eqref{2.7} holds, then \eqref{H2}, \eqref{m star} and \eqref{2.5} imply that $\mathbf{m}^*_\mathcal{F}(u,A) \geq \mathcal{F}(u,A)$ in the limit $\delta \to 0^+$. 
		
		In order to prove \eqref{2.7}, we first notice that $w^\delta|_{B_i^\delta} \in \gsbv^\gamma(B_i^\delta,\R^m)$ has zero trace on $\partial B_i^\delta$. Setting $w^{\delta,M}:=-M \vee w^\delta \wedge M$ with $M>0$, by the homogeneous Poincaré inequality in $BV$ we have
		$$
		\Vert w^{\delta,M} \Vert_{L^1(B_i^\delta)} \leq C \delta |Dw^{\delta,M}|(B_i^\delta).
		$$
		Therefore, by definition of $\{B_i^\delta\}_i$ we deduce that
		$$
		\Vert w^{\delta,M} \Vert_{L^1(A)} \leq C \delta |D w^{\delta,M}|(\cup_{i=1}^\infty B_i^\delta) \leq C \delta |D w^{\delta,M}|(A).
		$$
		The quantity $|D w^{\delta,M}|(A)$ is bounded in view of \eqref{2.6} since $u \in \gsbv(A,\R^m)$. Indeed,
		$$
		|D w^{\delta,M}|(A) \leq \int_A |\nabla v^\delta| \, dx+\int_A |\nabla u| \, dx+ 2M \left( \mathcal{H}^{d-1}(J_u \cap A)+\mathcal{H}^{d-1}(J_{v^\delta} \cap A) \right) < +\infty.
		$$
		This implies $w^{\delta,M} \to 0$ in $L^1(A,\R^m)$ as $\delta \to 0^+$ and thus in measure on $A$ for every $M>0$. Observe that if we take $M=1$ then we have that for every $\varepsilon \in (0,1)$
		$$
		\{ x \in A \colon |w^{\delta}(x)|> \varepsilon \} \subseteq \{ x \in A \colon |w^{\delta,1}(x)|> \varepsilon \},
		$$
		and this gives \eqref{2.7}.
	\end{proof}
	
	We finally prove Lemma \ref{lem 1}.
	
	\begin{proof}[Proof of Lemma \ref{lem 1}]
		The proof follows the same arguments of \cite[Lemma 5 and Lemma 6]{BFLMrelaxSBVp}. It essentially relies on Lemma \ref{lem 4 for lem 1} and on the assumptions on $\mathcal{F}$ but it is not hinged on the growth conditions. Hence we omit it.
	\end{proof}
	
	It remains to prove Lemmas \ref{lem 2} and \ref{lem 3}. This is the subject of the following two sections.
	
	\subsection{Lebesgue points} \label{secti vol}
	
	This section is devoted to the proof of Lemma \ref{lem 2}.
	We now proceed with the bulk part of the energy, that is, we analyze the blow up at points where the approximate gradient exists.
	
	\begin{lem}\label{vol seq}
		Let $u \in \gsbv^{\psi}(\Omega,\R^m)$. Then for $\mathcal{L}^d$-a.e. $x_0 \in \Omega$ and $\mathcal{L}^1$-a.e. $\lambda \in (0,1)$ there exists a sequence ${u}_\varepsilon \in \gsbv^{\psi}(B_\varepsilon(x_0))$ such that the following properties hold:
		\begin{align}\label{bulk seq props}
			\begin{split}
				&(i) \ \ {u}_\varepsilon=u \mbox{ on } B_\varepsilon(x_0) \setminus \overline{B_{\lambda\varepsilon}(x_0)}, \ \ \ \ \lim_{\varepsilon \to 0^+} \varepsilon^{-(d+1)} \mathcal{L}^d(\{ {u}_\varepsilon \neq u \} \cap B_\varepsilon(x_0))=0, \\
				&(ii) \ \ \lim_{\varepsilon \to 0^+} \frac{1}{\varepsilon^d} \int_{B_{\lambda \varepsilon}(x_0)} \psi \left(x,\frac{|{u}_\varepsilon(x)-u(x_0)-\nabla u(x_0)(x-x_0)|}{\varepsilon} \right) \, dx=0; \\
				&(iii) \ \ \lim_{\varepsilon \to 0^+} \frac{1}{\varepsilon^{d}} \mathcal{H}^{d-1}(J_{{u}_\varepsilon} \cap B_\varepsilon(x_0))=0.
			\end{split} 
		\end{align}
	\end{lem}
	
	\begin{proof}
		Without loss of generality we can assume $m=1$. Take $x_0 \in \Omega$ satisfying
		\begin{subequations}
				\begin{align}
				& \lim_{\varepsilon \to 0^+} \frac{1}{\varepsilon^d} \int_{B_\varepsilon(x_0)} \! \psi(x,|\nabla u(x)-\nabla u(x_0)|) \, dx=0; \label{vpta a} \\
				& \lim_{\varepsilon \to 0^+} \frac{1}{\varepsilon^d} \mathcal{H}^{d-1}(J_u \cap B_\varepsilon(x_0))=0; \label{vpta b} \\
				& \lim_{\varepsilon \to 0^+} \frac{1}{\varepsilon^d} \mathcal{L}^d \left( \left\{ x \in B_\varepsilon(x_0) \colon \frac{|u(x)-u(x_0)-\nabla u(x_0)(x-x_0)|}{|x-x_0|}>\rho \right\} \right)=0, \ \ \ \mbox{$\rho>0$} \label{vpta c}.
				\end{align}
		\end{subequations}
		Notice that \eqref{vpta a}--\eqref{vpta c} are satisfied by $\mathcal{L}^d$-a.e $x_0 \in \Omega$ by Proposition \ref{leb pts prop}. Set for brevity $T_\varepsilon:=T_{B_\varepsilon(x_0)}$. Observe that for $\varepsilon$ small enough, \eqref{vpta b} implies \eqref{poinc small jump}. For every $x \in B_\varepsilon(x_0)$ we can thus define the truncated functions 
		$$
		\overline{u}_\varepsilon(x):=T_\varepsilon \left(u(x)-u(x_0)-\nabla u(x_0)(x-x_0)\right)
		$$
		and, similarly as \eqref{other truncat op},
		$$
		\mathfrak{u}_\varepsilon^{\psi}(x):= \left(\med(\overline{u}_\varepsilon(x),B_\varepsilon(x_0))-\varepsilon (\psi^-_{B_\varepsilon(x_0)})^{-1}\left(\frac{1}{2\varepsilon}\right)\right) \vee \overline{u}_\varepsilon(x) \wedge \left(\med(\overline{u}_\varepsilon,B_\varepsilon(x_0))+\varepsilon (\psi_{B_\varepsilon(x_0)}^-)^{-1}\left(\frac{1}{2\varepsilon}\right)\right).
		$$
		Finally, we set $v_\varepsilon(x):=u(x_0)+\nabla u(x_0)(x-x_0)+\mathfrak{u}_\varepsilon^\psi(x)$. 
		
		We now want to estimate the quantity $\mathcal{L}^d(\{ v_\varepsilon \neq u \} \cap B_\varepsilon(x_0))$. We have 
		\begin{align}
			\begin{split}\label{blblbllb}
		\mathcal{L}^d(\{ v_\varepsilon \neq u \} \cap B_\varepsilon(x_0)) \leq & \ \mathcal{L}^d(\{ \mathfrak{u}^\psi_\varepsilon \neq \overline{u}_\varepsilon \} \cap B_\varepsilon(x_0)) \\
		& + \mathcal{L}^d(\{ \overline{u}_\varepsilon \neq u(x)-u(x_0)-\nabla u(x_0)(x-x_0)\} \cap B_\varepsilon(x_0)).
			\end{split}
		\end{align}
		By definition of $T_\varepsilon$ and \eqref{truncated difference}, we can estimate the second quantity on the right hand side of \eqref{blblbllb} with 
		$$
		\mathcal{L}^d(\{ \overline{u}_\varepsilon \neq u(x)-u(x_0)-\nabla u(x_0)(x-x_0)\} \cap B_\varepsilon(x_0)) \leq \left(2\gammaiso \mathcal{H}^{d-1}(J_u \cap B_\varepsilon(x_0)) \right)^{\frac{d}{d-1}}.
		$$
		For the first term of \eqref{blblbllb} recall that by Theorem \ref{weak and strong equiv} and Lemma \ref{equiv properties} we have that for every $t \geq 0$
		$$
		\psi^-_{B_\varepsilon(x_0)}\left((\psi^-_{B_\varepsilon(x_0)})^{-1}(t)\right) \geq \frac{t}{2K},
		$$ 
		where $K \geq 2 $ is the doubling constant of $\psi$. 
		
		We set $m_\varepsilon:=\med(u-\nabla u(x_0)(\cdot-x_0),B_\varepsilon(x_0))$. Using the definition of $\mathfrak{u}^\psi_\varepsilon$, Chebychev inequality and Theorem \ref{Poinc homogeneous form},
		\begin{align*}
			\mathcal{L}^d\left(\left\{  \mathfrak{u}^\psi_\varepsilon \neq \overline{u}_\varepsilon \right\} \cap B_\varepsilon(x_0)\right) &  \leq \mathcal{L}^d\left(B_\varepsilon(x_0) \cap\left\{ |\overline{u}_\varepsilon-\med(\overline{u}_\varepsilon,B_\varepsilon(x_0))|\geq \varepsilon (\psi^-_{B_\varepsilon(x_0)})^{-1}\left( \frac{1}{2\varepsilon} \right) \right\}\right) \\
			& \leq \mathcal{L}^d \left(B_\varepsilon(x_0)\cap \left\{ \psi^-_{B_\varepsilon(x_0)} \left( \frac{|\overline{u}_\varepsilon-\med(\overline{u}_\varepsilon,B_\varepsilon(x_0))|}{\varepsilon}\right) \geq \frac{1}{4K\varepsilon} \right\} \right) \\
			&\leq  4K\varepsilon \int_{B_\varepsilon(x_0)} \psi^-_{B_\varepsilon(x_0)} \left( \frac{|\overline{u}_\varepsilon(x)-\med(\overline{u}_\varepsilon,B_\varepsilon(x_0))|}{\varepsilon}\right) \, dx \\
			&=4K\varepsilon \int_{B_\varepsilon(x_0)} \psi^-_{B_\varepsilon(x_0)} \left( \frac{|T_\varepsilon(u(x)-\nabla u(x_0)(x-x_0))-m_\varepsilon|}{\varepsilon}\right) \, dx \\
			& \leq \varepsilon C \int_{B_\varepsilon(x_0)} \psi_{B_\varepsilon(x_0)}^-(|\nabla u(x)-\nabla u(x_0)|) \, dx \\
			& \leq \varepsilon C \int_{B_\varepsilon(x_0)} \psi(x,|\nabla u(x)-\nabla u(x_0)|) \, dx.
		\end{align*}
		where $C=C(d,K)$. Thus, using Fubini Theorem and \eqref{vpta a}--\eqref{vpta b}, we deduce that
		\begin{align}\label{1.2}
			\begin{split}
				\limsup_{\varepsilon \to 0^+} & \frac{1}{\varepsilon^d} \int_0^1 \! \mathcal{H}^{d-1}(\{ v_\varepsilon \neq u \} \cap \partial B_{\lambda \varepsilon}(x_0)) d\lambda = \limsup_{\varepsilon \to 0^+} \frac{2}{\varepsilon^{d+1}} \mathcal{L}^d(\{ v_\varepsilon \neq u \} \cap B_\varepsilon(x_0)) \\
				& \leq \limsup_{\varepsilon \to 0^+} \frac{1}{\varepsilon^{d+1}}\left( \varepsilon C \int_{B_\varepsilon(x_0)} \psi(x,|\nabla u(x)-\nabla u(x_0)|) \, dx+\left(2 \gammaiso \mathcal{H}^{d-1}(J_u \cap B_\varepsilon(x_0)) \right)^{\frac{d}{d-1}} \right) \\
				& =0.
			\end{split}
		\end{align}
		Therefore, since $d \geq 2$, for every sequence $\varepsilon \to 0$, one can find a subsequence such that for $\mathcal{L}^1$-a.e. $\lambda \in (0,1)$ it holds
		\begin{align}\label{1.1}
			\begin{split}
				&\mathcal{H}^{d-1}(\partial B_{\lambda\varepsilon}(x_0)\cap J_{v_\varepsilon})=0, \\
				&\lim_{\varepsilon \to 0^+} \varepsilon^{-d} \mathcal{H}^{d-1}(\{ v_\varepsilon \neq u \} \cap \partial B_{\lambda \varepsilon}(x_0))=0.
			\end{split}
		\end{align}
		Now let us fix a subsequence $\varepsilon \to 0$ (not relabelled) and $\lambda \in (0,1)$ such that \eqref{1.1} holds. We set
		\begin{equation*}
			u_\varepsilon(x):=
			\begin{cases*}
				v_\varepsilon(x) & if $x \in B_{\lambda \varepsilon}(x_0)$, \\
				u(x) & if $x \in B_\varepsilon(x_0) \setminus \overline{B_{\lambda \varepsilon}(x_0)}$.
			\end{cases*}
		\end{equation*}
		From the definition of $u_\varepsilon$ and the estimates in \eqref{1.2}, we deduce \eqref{bulk seq props}(i) and \eqref{bulk seq props}(iii). We are left to prove \eqref{bulk seq props}(ii).
		
		We actually will prove a slightly stronger statement, namely 
		\begin{equation}\label{1.3}
			\lim_{\varepsilon \to 0^+} \frac{1}{\varepsilon^d} \int_{B_{\varepsilon}(x_0)} \psi \left(x,\frac{|v_\varepsilon(x)-u(x_0)-\nabla u(x_0)(x-x_0)|}{\varepsilon} \right) \, dx=\lim_{\varepsilon \to 0^+} \frac{1}{\varepsilon^d} \int_{B_{\varepsilon}(x_0)} \psi \left(x,\frac{|\mathfrak{u}^\psi_\varepsilon(x)|}{\varepsilon}\right) \, dx=0.
		\end{equation}
		Let $\varepsilon_0>0$ be fixed. Set $\psi_\varepsilon^-:=\psi^-_{B_\varepsilon(x_0)}$ and $\psi_\varepsilon^+:=\psi^+_{B_\varepsilon(x_0)}$ for brevity. Using property \hyperlink{adA1}{(adA1)} and the fact that $\psi$ satisfies \hyperlink{idg}{$\textnormal{(aDec)}$}, by definition of $\mathfrak{u}_\varepsilon^\psi$ and by inequality \eqref{P2} we estimate, for every $\varepsilon<\varepsilon_0$,
		\begin{align*}
			 & \int_{B_{\varepsilon}(x_0)}   \psi \left(x,\frac{| \mathfrak{u}^\psi_\varepsilon(x)-\med(\overline{u}_\varepsilon,B_\varepsilon(x_0))|}{\varepsilon}\right) \, dx \\ 
			 & \leq C \int_{B_{\varepsilon}(x_0)} \psi^-_\varepsilon \left(|\nabla u(x)-\nabla u(x_0)|\right) \, dx +\int_{B_\varepsilon(x_0)} \psi^+_{\varepsilon_0}\left(\frac{|\overline{u}_\varepsilon(x)-\med(\overline{u}_\varepsilon,B_\varepsilon(x_0))|}{\varepsilon}\wedge \sigma\right) \, dx.
		\end{align*}
		with $C=C(d,K,\beta)$, where $\sigma$ and $\beta$ are the constants appearing in \hyperlink{A0}{(A0)} and \hyperlink{adA1}{(adA1)}, respectively. 
		
		\noindent By \eqref{vpta a}, in order to prove \eqref{1.3} we only need to show that 
		\begin{align*}
		&\lim_{\varepsilon \to 0^+} \frac{\med(\overline{u}_\varepsilon,B_\varepsilon(x_0))}{\varepsilon}=\lim_{\varepsilon \to 0^+} \med\left( \frac{u-u(x_0)-\nabla u(x_0)(\cdot-x_0)}{\varepsilon},B_\varepsilon(x_0) \right)=0, \\
		&\lim_{\varepsilon \to 0^+} \frac{1}{\varepsilon^d} \int_{B_\varepsilon(x_0)} \psi^+_{\varepsilon_0}\left(\frac{|\overline{u}_\varepsilon(x)|}{\varepsilon} \wedge \sigma\right) \, dx=0.
		\end{align*}
		The first limit follows by \eqref{vpta c} The second limit follows again by \eqref{vpta c} and dominate convergence theorem after rescaling.
	\end{proof}

	We can now prove Lemma \ref{lem 2} which will actually follow as a consequence of the next two Lemmas.
	
	\begin{lem}\label{lem 2.1}
		Let $\mathcal{F}$ satisfy \eqref{H1} and \eqref{H3}-\eqref{H4} and let $u \in \gsbv^{\psi}(\Omega,\R^m)$. Then, for $\mathcal{L}^d$-a.e $x_0 \in \Omega$
		\begin{equation}\label{2.8}
			\lim_{\varepsilon \to 0^+} \frac{\mathbf{m}_\mathcal{F}(u,B_\varepsilon(x_0))}{\omega_d \varepsilon^d} \leq \limsup_{\varepsilon \to 0^+} \frac{\mathbf{m}_\mathcal{F}(\overline{u}^\textnormal{bulk}_{x_0},B_\varepsilon(x_0))}{\omega_d \varepsilon^d}.
		\end{equation}
	\end{lem}
	
	\begin{proof}
		We will prove the assertion for all the points $x_0 \in \Omega$ for which the statement of Lemma \ref{vol seq} holds, $\lim_{\varepsilon \to 0^+} (\omega_d\varepsilon^d)^{-1} \mu(B_\varepsilon(x_0))=1$ and 
		\begin{equation}\label{2.10}
			\lim_{\varepsilon \to 0^+} \frac{\mathcal{F}(u,B_\varepsilon(x_0))}{\omega_d \varepsilon^d} = \lim_{\varepsilon \to 0^+} \frac{\mathbf{m}_\mathcal{F}(u,B_\varepsilon(x_0))}{\omega_d\varepsilon^d} <+\infty.
		\end{equation}  
		By Lemma \ref{lem 1}, property \eqref{2.10} hold for $\mathcal{L}^{d}$-a.e. $x_0 \in \Omega$.
		Let $(u_\varepsilon)_\varepsilon$ be the sequence given by Lemma \ref{vol seq} and fix $\lambda \in (0,1)$ such that \eqref{bulk seq props}(ii) holds. We write $\lambda=1-\theta$ for some $\theta \in (0,1)$.
		
		Given $z_\varepsilon \in \gsbv^{\psi}(B_{(1-3\theta)\varepsilon}(x_0),\R^m)$ such that $z_\varepsilon=\ubulk$ in a neighborhood of $\partial B_{(1-3\theta)\varepsilon}(x_0)$ and
		\begin{equation}\label{2.10a}
			\mathcal{F}(z_\varepsilon,B_{(1-3\theta)\varepsilon}(x_0)) \leq \mathbf{m}_\mathcal{F}(\ubulk,B_{(1-3\theta)\varepsilon}(x_0))+\omega_d \varepsilon^{d+1},
		\end{equation}
		we extend it to $z_\varepsilon \in \gsbv(B_\varepsilon(x_0),\R^m)$ by setting $z_\varepsilon=\ubulk$ outside of $B_{(1-3\theta)\varepsilon}(x_0)$. We further set $C_{\varepsilon,\theta}(x_0):=B_\varepsilon(x_0) \setminus \overline{B_{(1-4\theta)\varepsilon}(x_0)}$. Now we use Lemma \ref{fund est} with $u$ and $v$ replaced by $z_\varepsilon$ and $u_\varepsilon$, respectively, and with the sets
		\begin{equation}\label{2.18}
			D'_{\varepsilon,x_0}:=B_{(1-2\theta)\varepsilon}(x_0), \ \ \ D''_{\varepsilon,x_0}:=B_{(1-\theta)\varepsilon}(x_0), \ \ \ E_{\varepsilon,x_0}:=C_{\varepsilon,\theta}(x_0).
		\end{equation}
		Notice that $C_{\varepsilon,\theta}(x_0)=(C_{1,\theta}(x_0))_{\varepsilon,x_0}$ according to the notation introduced in \eqref{000}, where $C_{1,\theta}(x_0):=B_1(x_0) \setminus \overline{B_{(1-4\theta)}(x_0)}$. Moreover, we observe that $\mathcal{L}^d(C_{1,\theta}(x_0))=\omega_d(1-(1-4\theta)) \to 0$ as $\theta \to 0$.
		For an arbitrarily fixed $\eta>0$, we find $w_\varepsilon \in \gsbv^{\psi}(B_\varepsilon(x_0),\R^m)$ such that $w_\varepsilon=u_\varepsilon$ on $B_\varepsilon(x_0) \setminus B_{(1-\theta)\varepsilon}(x_0)$ and
		\begin{align}\label{2.11}
			\begin{split}
				\mathcal{F}(w_\varepsilon,B_\varepsilon(x_0)) & \leq (1+\eta)(\mathcal{F}(z_\varepsilon,B_{(1-\theta)\varepsilon}(x_0))+\mathcal{F}(u_\varepsilon,C_{\varepsilon,\theta}(x_0)))+\eta \mathcal{L}^d(B_\varepsilon(x_0)) \\
				&+M \int_{B_{(1-\theta)\varepsilon}(x_0) \setminus B_{(1-2\theta)\varepsilon}(x_0)} \psi \left(x, \frac{|z_\varepsilon-u_\varepsilon|}{\varepsilon} \right) \, dx
			\end{split}
		\end{align}
		Recalling \eqref{bulk seq props}(i) in Lemma \ref{vol seq}, we have that $w_\varepsilon=u_\varepsilon=u$ in a neighborhood of $\partial B_\varepsilon(x_0)$. Moreover, since $z_\varepsilon=\ubulk$ outside $B_{(1-3\theta)\varepsilon}(x_0)$, using \eqref{bulk seq props}(ii) we infer that
		\begin{align}\label{2.12}
			\begin{split}
				\lim_{\varepsilon \to 0^+} & \frac{1}{\varepsilon^d} \int_{B_{(1-\theta)\varepsilon}(x_0) \setminus B_{(1-2\theta)\varepsilon}(x_0)} \psi \left(x, \frac{|z_\varepsilon-u_\varepsilon|}{\varepsilon} \right) \, dx \\  & \leq \lim_{\varepsilon \to 0^+} \frac{1}{\varepsilon^d} \int_{B_{(1-\theta)\varepsilon}(x_0)} \psi \left(x, \frac{|\ubulk-u_\varepsilon|}{\varepsilon} \right) \, dx =0.
			\end{split}
		\end{align}
		From \eqref{2.11} and \eqref{2.12} we deduce that there exists a non-negative infinitesimal sequence $(\rho_\varepsilon)_\varepsilon$ such that
		\begin{equation}\label{2.13}
			\mathcal{F}(w_\varepsilon,B_\varepsilon(x_0)) \leq (1+\eta)(\mathcal{F}(z_\varepsilon,B_{(1-\theta)\varepsilon}(x_0))+\mathcal{F}(u_\varepsilon,C_{\varepsilon,\theta}(x_0)))+\varepsilon^d \rho_\varepsilon+\eta \omega_d \varepsilon^d.
		\end{equation}
		Then, using the fact that $z_\varepsilon = \ubulk$ in $B_\varepsilon(x_0) \setminus B_{(1-3\theta)\varepsilon}(x_0) \subset C_{\varepsilon,\theta}(x_0)$, \eqref{H1}, \eqref{H4} and \eqref{2.10a}, setting $\psi_\varepsilon^+(t):=\psi^+_{B_\varepsilon(x_0)}(\cdot,t)$ for $t>0$, we deduce that there exists $\varepsilon_0>0$ small enough such that for every $\varepsilon<\varepsilon_0$
		\begin{equation}\label{1700}
			\frac{\mathcal{F}(\ubulk,C_{\varepsilon,\theta}(x_0))}{\varepsilon^d} \leq b \mathcal{L}^d(C_{\varepsilon,\theta}(x_0))(1+\psi^+_{\varepsilon_0}(|\nabla u(x_0)|)).
		\end{equation}
		Thus,
		\begin{align}
				\limsup_{\varepsilon \to 0^+} \frac{\mathcal{F}(z_\varepsilon,B_{(1-\theta)\varepsilon}(x_0))}{\varepsilon^d}& \leq \limsup_{\varepsilon \to 0^+} \left( \frac{\mathcal{F}(z_\varepsilon,B_{(1-3\theta)\varepsilon}(x_0))}{\varepsilon^d}+\frac{\mathcal{F}(\ubulk,C_{\varepsilon,\theta}(x_0))}{\varepsilon^d} \right) \label{2.14} \\
				& \leq \limsup_{\varepsilon \to 0^+} \left(\frac{\mathbf{m}_\mathcal{F}(\ubulk,B_{(1-3\theta)\varepsilon}(x_0))}{\varepsilon^d}+b \mathcal{L}^d(C_{\varepsilon,\theta}(x_0))(1+\psi^+_{\varepsilon_0}(|\nabla u(x_0)|)) \right) \nonumber \\
				& \leq (1-3\theta)^d \limsup_{\varepsilon \to 0^+} \frac{\mathbf{m}_\mathcal{F}(\ubulk,B_{(1-3\theta)\varepsilon}(x_0))}{(1-3\theta)^d\varepsilon^d} \nonumber \\
				& \ \ \ \ +\limsup_{\varepsilon \to 0^+} b \mathcal{L}^d(C_{\varepsilon,\theta}(x_0))(1+\psi^+_{\varepsilon_0}(|\nabla u(x_0)|)) \nonumber .
		\end{align}
		On the other hand, using again \eqref{H4} we also deduce
		\begin{align*}
			\mathcal{F}(u_\varepsilon,C_{\varepsilon,\theta}(x_0))  \leq b \mathcal{L}^d(C_{\varepsilon,\theta}(x_0))+ b\int_{C_{\varepsilon,\theta}(x_0)} \psi(x,|\nabla u_\varepsilon(x)|) \, dx+b\mathcal{H}^{d-1}(J_{u_\varepsilon} \cap C_{\varepsilon,\theta}(x_0)).
		\end{align*}
		Now, using \eqref{bulk seq props}(iii), we get
		\begin{equation}\label{2.15}
			\limsup_{\varepsilon \to 0^+} \frac{\mathcal{F}(u_\varepsilon,C_{\varepsilon,\theta}(x_0))}{\varepsilon^d} \leq b \mathcal{L}^d(C_{1,\theta}(x_0))+\limsup_{\varepsilon \to 0^+} \frac{b}{\varepsilon^d} \int_{C_{\varepsilon,\theta}(x_0)}\psi(x,|\nabla u_\varepsilon(x)|) \, dx.
		\end{equation}
		By exploiting the construction in Lemma \ref{vol seq}, we notice that $|\nabla u_\varepsilon| \leq |\nabla u|$ $\mathcal{L}^d$-a.e. in $B_\varepsilon(x_0)$ for any $\varepsilon>0$. Thus, using \eqref{vpta a} and the fact that $\psi$ is doubling with constant $K \geq 2$, we have
		\begin{align}
				& \limsup_{\varepsilon \to 0^+} \ \frac{b}{\varepsilon^d} \int_{C_{\varepsilon,\theta}(x_0)} \psi(x,|\nabla u_\varepsilon|) \, dx \label{2.16} \\
				& \ \leq \limsup_{\varepsilon \to 0^+} \ \frac{b}{\varepsilon^d} \int_{C_{\varepsilon,\theta}(x_0)} \psi(x,|\nabla u|) \, dx \nonumber \\
				& \ \leq \limsup_{\varepsilon \to 0^+} \ (1+K) \left( \frac{b}{\varepsilon^d} \int_{B_{ \varepsilon}(x_0)} \psi(x,|\nabla u(x)-\nabla u(x_0)|) \, dx+b\mathcal{L}^d(C_{\varepsilon,\theta}(x_0))\psi^+_{\varepsilon_0}(|\nabla u(x_0)|) \right) \nonumber \\
				& \ \leq b(1+K)\mathcal{L}^d(C_{1,\theta}(x_0))\psi^+_{\varepsilon_0}(|\nabla u(x_0)|). \nonumber
		\end{align}
		Combining \eqref{2.15} and \eqref{2.16} we infer
		\begin{equation}\label{2.17}
			\limsup_{\varepsilon \to 0^+} \frac{\mathcal{F}(u_\varepsilon,C_{\varepsilon,\theta}(x_0))}{\varepsilon^d} \leq b\mathcal{L}^d(C_{1,\theta}(x_0))\left(1+(1+K)\psi^+_{\varepsilon_0}(|\nabla u(x_0)|)\right).
		\end{equation}
		Recalling that $w_\varepsilon=u_\varepsilon$ in a neighborhood of $\partial B_\varepsilon(x_0)$, by \eqref{2.13}, \eqref{2.14} and \eqref{2.17} we finally get
		\begin{align*}
			\lim_{\varepsilon \to 0^+} \frac{\mathbf{m}_\mathcal{F}(u,B_\varepsilon(x_0))}{\omega_d \varepsilon^d} & \leq \limsup_{\varepsilon \to 0^+} \frac{\mathcal{F}(w_\varepsilon,B_\varepsilon(x_0))}{\omega_d\varepsilon^d} \\
			& \leq (1+\eta)(1-3\theta)^d \limsup_{\varepsilon \to 0^+} \frac{\mathbf{m}_\mathcal{F}(\ubulk,B_\varepsilon(x_0))}{\omega_d\varepsilon^d} \\
			& \ \ \ + (1+\eta)b\omega_d^{-1} \mathcal{L}^d(C_{1,\theta}(x_0))\left(1+(1+K)\psi^+_{\varepsilon_0}(|\nabla u(x_0)|)\right)+\eta.
		\end{align*}
		Letting $\eta \to 0$ and $\theta \to 0$ we get \eqref{2.8}.
		
	\end{proof}
	
	\begin{lem}\label{lem 2.2}
			Let $\mathcal{F}$ satisfy \eqref{H1} and \eqref{H3}-\eqref{H4} and let $u \in \gsbv^{\psi}(\Omega,\R^m)$. Then, for $\mathcal{L}^d$-a.e $x_0 \in \Omega$ we have
		\begin{equation}\label{2.9}
			\lim_{\varepsilon \to 0^+} \frac{\mathbf{m}_\mathcal{F}(u,B_\varepsilon(x_0))}{\omega_d \varepsilon^d} \geq \limsup_{\varepsilon \to 0^+} \frac{\mathbf{m}_\mathcal{F}(\overline{u}^\textnormal{bulk}_{x_0},B_\varepsilon(x_0))}{\omega_d \varepsilon^d}.
		\end{equation}
	\end{lem}
	
	\begin{proof}
		We prove the statement for points $x_0 \in \Omega$ considered in Lemma \ref{lem 2.1}. Let $\eta>0$ and $\lambda=1-\theta$ be fixed as in the proof of Lemma \ref{lem 2.1}, and let $(u_\varepsilon)_\varepsilon$ be the sequence given by Lemma \ref{vol seq}. Using Fubini's Theorem (as in \eqref{1.2} and \eqref{1.1}), we have that for every $\varepsilon>0$ we can find $s \in (1-4\theta,1-3\theta)$ such that
		\begin{align}\label{2.19}
			\begin{split}
				& \mathcal{H}^{d-1}(\partial B_{s\varepsilon}(x_0) \cap (J_u \cup J_{u_\varepsilon}))=0 \ \ \ \mbox{for every $\varepsilon>0$}, \\
				& \lim_{\varepsilon \to 0^+} \varepsilon^{-d} \mathcal{H}^{d-1}(\{ u_\varepsilon \neq u \} \cap \partial B_{s\varepsilon}(x_0))=0.
			\end{split}
		\end{align}
		From now on, we will follow mainly the arguments of the proof of Lemma \ref{lem 2.1}. We choose a sequence $z_\varepsilon \in \gsbv^{\psi}(B_{s \varepsilon}(x_0),\R^m)$ such that $z_\varepsilon =u$ in a neighborhood of $\partial B_{s \varepsilon}(x_0)$ and
		\begin{equation}\label{2.20}
			\mathcal{F}(z_\varepsilon,B_{s\varepsilon}(x_0)) \leq \mathbf{m}_\mathcal{F}(u,B_{s\varepsilon}(x_0))+\omega_d\varepsilon^{d+1}.
		\end{equation}
		Setting $z_\varepsilon =u_\varepsilon$ outside $B_{s\varepsilon}(x_0)$ we extend it to $z_\varepsilon \in \gsbv^{\psi}(B_\varepsilon(x_0),\R^m)$. We now use Lemma \ref{fund est} applied with $u$ and $v$ replaced by $z_\varepsilon$ and $\ubulk$, respectively, and the same choice of sets $D'_{\varepsilon,x_0},D''_{\varepsilon,x_0},E_{\varepsilon,x_0}$ as in \eqref{2.18}.
		
		Hence, we find $w_\varepsilon \in \gsbv^{\psi}(B_\varepsilon(x_0),\R^m)$ such that $w_\varepsilon = \ubulk$ on $B_\varepsilon(x_0) \setminus B_{(1-\theta)\varepsilon} (x_0)$ and 
		\begin{align*}
			\mathcal{F}(w_\varepsilon,B_\varepsilon(x_0)) \leq & (1+\eta)\left( \mathcal{F}(z_\varepsilon,B_{(1-\theta)\varepsilon}(x_0))+\mathcal{F}(\ubulk,C_{\varepsilon,\theta}(x_0)) \right)+\eta \mathcal{L}^d(B_\varepsilon(x_0))+ \\
			& + M \int_{B_{(1-\theta)\varepsilon}(x_0) \setminus B_{(1-2\theta)\varepsilon}(x_0)} \psi \left(x,\frac{|z_\varepsilon-\ubulk|}{\varepsilon}\right) \, dx. 
		\end{align*}
		Since from the initial choice of $s$ we have that $z_\varepsilon =u_\varepsilon$ outside $B_{(1-3\theta)\varepsilon}(x_0)$, by arguing as in the proof of Lemma \ref{lem 2.1} (see \eqref{2.12} and \eqref{2.13}), we find an non-negative infinitesimal sequence $(\rho_\varepsilon)_\varepsilon$ such that
		\begin{equation}\label{2.21}
			\mathcal{F}(w_\varepsilon,B_\varepsilon(x_0)) \leq (1+\eta)(\mathcal{F}(z_\varepsilon,B_{(1-\theta)\varepsilon}(x_0))+\mathcal{F}(u_\varepsilon,C_{\varepsilon,\theta}(x_0)))+\varepsilon^d \rho_\varepsilon+\eta \omega_d \varepsilon^d.
		\end{equation}
		We now estimate the terms in \eqref{2.21}. Using that $z_\varepsilon=u_\varepsilon$ on $B_\varepsilon(x_0) \setminus B_{s\varepsilon}(x_0) \subset C_{\varepsilon,\theta}(x_0)$, \eqref{H1}, \eqref{H4} and \eqref{2.20} we get
		\begin{align}\label{2.22}
			\begin{split}
				\mathcal{F}(z_\varepsilon,B_{(1-\theta)\varepsilon}(x_0)) \leq & \mathbf{m}_\mathcal{F}(u,B_{s\varepsilon}(x_0))+\omega_d \varepsilon^{d+1}+b \mathcal{H}^{d-1}(\partial B_{s\varepsilon}(x_0) \cap (J_{u_\varepsilon} \cup J_u))  \\ 
				&+\mathcal{F}(u_\varepsilon,C_{\varepsilon,\theta}(x_0)).
			\end{split}
		\end{align}
		Now, recalling \eqref{2.17} and using \eqref{2.19} together with the fact that $s\varepsilon \leq (1-3\theta)\varepsilon$, we estimate
		\begin{align}
				\limsup_{\varepsilon \to 0^+} \label{2.23} \frac{\mathcal{F}(z_\varepsilon,B_{(1-\theta)\varepsilon}(x_0))}{\varepsilon^d} & \leq s^d \limsup_{\varepsilon \to 0^+} \frac{\mathbf{m}_\mathcal{F}(u,B_{s\varepsilon}(x_0))}{s^d\varepsilon^d}+b\mathcal{L}^d(C_{1,\theta}(x_0))\left(1+C\psi^+_\varepsilon(|\nabla u(x_0)|)\right) \\
				& \leq (1-3\theta)^d \limsup_{\varepsilon \to 0^+} \frac{\mathbf{m}_\mathcal{F}(u,B_{\varepsilon}(x_0))}{\varepsilon^d}+b\mathcal{L}^d(C_{1,\theta}(x_0))\left(1+C\psi^+_\varepsilon(|\nabla u(x_0)|)\right) \nonumber
		\end{align}
		where $C>0$ is a constant depending only on $\psi$.
		
		The estimate for $\mathcal{F}(\ubulk,C_{\varepsilon,\theta}(x_0))$ in \eqref{1700} together with the estimates \eqref{2.21} and \eqref{2.23} gives, recalling that $\rho_\varepsilon \to 0$ as $\varepsilon \to 0$,
		\begin{align*}
			\begin{split}
				\limsup_{\varepsilon \to 0^+} \frac{\mathcal{F}(w_\varepsilon,B_\varepsilon(x_0))}{\varepsilon^d} \leq & (1+\eta)(1-3\theta)^d\limsup_{\varepsilon \to 0^+} \frac{\mathbf{m}_\mathcal{F}(u,B_\varepsilon(x_0))}{\varepsilon^d} \\
				& +2(1+\eta)b\mathcal{L}^d(C_{1,\theta})\left(1+C\psi^+_\varepsilon(|\nabla u(x_0)|) \right) +\omega_d \eta.
			\end{split}
		\end{align*}
		Finally, letting $\eta \to 0$ and $\theta \to 0$ and recalling that $w_\varepsilon=\ubulk$ in a neighborhood of $\partial B_\varepsilon(x_0)$, we deduce
		\begin{align*}
			\begin{split}
				\limsup_{\varepsilon \to 0^+} \frac{\mathbf{m}_\mathcal{F}(\ubulk,B_\varepsilon(x_0))}{\omega_d \varepsilon} & \leq \limsup_{\varepsilon \to 0^+} \frac{\mathcal{F}(w_\varepsilon,B_\varepsilon(x_0))}{\omega_d \varepsilon} \\ 
				& \leq \limsup_{\varepsilon \to 0^+} \frac{\mathbf{m}_\mathcal{F}(u,B_\varepsilon(x_0))}{\omega_d \varepsilon}
				= \limsup_{\varepsilon \to 0^+} \frac{\mathbf{m}_\mathcal{F}(u,B_\varepsilon(x_0))}{\omega_d \varepsilon},
			\end{split}
		\end{align*}
		this proves \eqref{2.9}.
	\end{proof}
	
	\subsection{Surface points} \label{secti jump}
	
	This section is devoted to the proof of Lemma \ref{lem 2}.
	We can now perform a similar analysis for the jump points.
	
	\begin{lem}\label{surf seq}
		Let $u \in \gsbv^{\psi}(\Omega,\R^m)$. Then, for $\mathcal{H}^{d-1}$-a.e. $x_0 \in J_u$ and $\mathcal{L}^1$-a.e. $\lambda \in (0,1)$ there exists a sequence ${u}_\varepsilon \in \gsbv^{\psi}(B_\varepsilon(x_0))$ such that the following properties hold
		\begin{align}\label{jump seq props}
			\begin{split}
				&(i) \ \ {u}_\varepsilon=u \mbox{ on } B_\varepsilon(x_0) \setminus \overline{B_{\lambda\varepsilon}(x_0)}, \ \ \ \ \lim_{\varepsilon \to 0^+} \varepsilon^{-d} \mathcal{L}^d(\{ {u}_\varepsilon \neq u \} \cap B_\varepsilon(x_0))=0; \\
				&(ii) \ \ J_{u_\varepsilon} \setminus J_u \subset \partial B_{\lambda\varepsilon}(x_0), \ \ \ \lim_{\varepsilon \to 0^+} \varepsilon^{-(d-1)} \mathcal{H}^{d-1}(J_{u_\varepsilon} \setminus J_u)=0; \\
				&(iii) \ \ \lim_{\varepsilon \to 0^+} \frac{1}{\varepsilon^{d-1}} \int_{B_{ \varepsilon}(x_0)} \psi(x,|\nabla u_\varepsilon(x)|) \, dx=0; \\
				&(iv) \ \ \lim_{\varepsilon \to 0^+} \frac{1}{\varepsilon^{d-1}} \int_{B_{\lambda \varepsilon}(x_0)} \psi \left(x,\frac{|{u}_\varepsilon(x)-u_{x_0,\nu}(x)|}{\varepsilon} \right) \, dx=0. \\
			\end{split} 
		\end{align}
	\end{lem}
	
	\begin{proof}
		As before we assume without loss of generality $m=1$. Since $\rho_\psi(\nabla u)<+\infty$ we have that for $\mathcal{H}^{d-1}$-a.e $x_0 \in J_u$ it holds
		\begin{equation}\label{1.8}
			\lim_{\varepsilon \to 0^+} \frac{1}{\varepsilon^{d-1}} \int_{B_{\varepsilon}(x_0)} \psi(x,|\nabla u(x)|) \, dx=0;
		\end{equation}
		see e.g. \cite[Section 2.4.3, Theorem 2.10]{EvansGariepy}. Since $J_u$ is $(d-1)$-rectifiable, there exists a sequence of compact sets $K_j \subset \Omega$ and a set $N \subset \Omega$ such that
		$$
		J_u = \bigcup_{j=1}^\infty K_j \cup N, \ \ \ \ \  \mathcal{H}^{d-1}(N)=0;
		$$
		and each $K_j$ is a subset of a $C^1$-hypersurface. Then, taking $\varepsilon_0>0$ small enough, in a neighborhood $B_{\varepsilon_0}(y)$ of each point $y \in K_j$, up to rotations, we may find a $C^1$-function $\Gamma_j:\R^{d-1} \to \R$ such that
		$$
		K_j \cap B_{\varepsilon_0}(y) \subseteq \{ x=(x',x_d) \in B_{\varepsilon_0}(x_0) \colon x_d=\Gamma_j(x') \}.
		$$ 
		Since $u \in \gsbv(\Omega)$, if we restrict it to the Lipschitz set $\Omega_j:=\{ (x',x_d) \in \Omega \colon x_d > \Gamma_j(x') \}$ then it has unique measurable trace on its boundary, that is, for $\mathcal{H}^{d-1}$-a.e. $x_0 \in \partial \Omega_j$ there exists tr$(u)(x_0) \in \R^d$ such that 
		\begin{equation}\label{1.10}
			\lim_{\varepsilon \to 0^+} \varepsilon^{-d} \mathcal{L}^d \left( \Omega_j \cap B_\varepsilon(x_0) \cap \{ |u-\mbox{tr}(u)(x_0)|>\rho \} \right)=0 \ \ \ \mbox{for all $\rho>0$};
		\end{equation}
		see e.g. \cite[Theorem 5.5]{GDMjems}.
		Up to taking $\varepsilon_0$ smaller, we can define a function $w \in \gsbv^{\psi}(B_{\varepsilon_0}(y))$ as
		\begin{equation*}
			w(x):=
			\begin{cases*}
				u(x',x_d) & if $x_d > \Gamma_j(x')$, \\
				u(x',-x_d+2\Gamma_j(x')) & if $x_d < \Gamma_j(x')$.
			\end{cases*}
		\end{equation*}
		By construction we have that $|\nabla w| \leq C|\nabla u|$ $\mathcal{L}^d$-a.e. on $B_{\varepsilon_0}(y)$, thus indeed $w \in \gsbv^{\psi}(B_{\varepsilon_0}(y))$. Moreover, $J_w \cap B_{\varepsilon_0}(y) \subset B_{\varepsilon_0}(y) \setminus K_j$. Now we claim that for $\mathcal{H}^{d-1}$-a.e. $x_0 \in B_{\varepsilon_0}(y) \cap K_j$ we have either $w(x_0)=u^+(x_0)$ or $w(x_0)=u^-(x_0)$. Indeed, by definition of $w$ and \eqref{1.10} we infer that for $\mathcal{H}^{d-1}$-a.e. $x_0 \in B_{\varepsilon_0}(y) \cap K_j$ it holds
		\begin{equation}\label{1.11}
			\lim_{\varepsilon \to 0^+} \varepsilon^{-d} \mathcal{L}^d \left( B_\varepsilon(x_0) \cap \{ |w-w(x_0)|>\rho \} \right)=0 \ \ \ \mbox{for all $\rho>0$}.
		\end{equation}
		Thus, combining \eqref{1.10} and \eqref{1.11} we deduce the claim, being tr$(u)(x_0)$ either $u^+(x_0)$ or $u^-(x_0)$. 
		
		Thanks to Theorem \ref{DGCL 1}, we have that for $\mathcal{H}^{d-1}$-a.e. $x_0 \in K_j \cap B_{\varepsilon_0}(y)$
		\begin{equation}\label{1.12}
			\lim_{\varepsilon \to 0^+} \frac{1}{\varepsilon^{d-1}} \mathcal{H}^{d-1}(J_w \cap B_\varepsilon(x_0)) =0.
		\end{equation}
		Moreover, by \eqref{1.8} and definition of $w$, we have that for $\mathcal{H}^{d-1}$-a.e. $x_0 \in K_j \cap B_{\varepsilon_0}(y)$
		\begin{equation}\label{1.13}
			\lim_{\varepsilon \to 0^+} \frac{1}{\varepsilon^{d-1}} \int_{B_{\varepsilon}(x_0)} \psi_\varepsilon^- (|\nabla w(x)|) \, dx=0,
		\end{equation}
		where we recall that $\psi_\varepsilon^-:=\inf_{B_\varepsilon(x_0)} \psi(x,\cdot)$. 
		
		Let us fix $x_0 \in K_j \cap B_{\varepsilon_0}(y)$ such that properties \eqref{1.11}--\eqref{1.13} hold and assume without loss of generality that $w(x_0)=u^+(x_0)$.
		We fix $\eta>0$ arbitrarily small as before, by \eqref{1.12} and \eqref{1.13}, for every $\varepsilon>0$ small enough we have
		\begin{equation}\label{1.14}
			\int_{B_{\varepsilon}(x_0)} \psi_\varepsilon^- (|\nabla w(x)|) \, dx+\mathcal{H}^{d-1}(J_w \cap B_\varepsilon(x_0)) < \eta \varepsilon^{d-1}.
		\end{equation}
		Set $T_\varepsilon:=T_{B_\varepsilon(x_0)}$ and $m_\varepsilon:=\med(w,B_\varepsilon(x_0))$. We define similarly to \eqref{other truncat op}
		\begin{equation}\label{dp}
		\mathfrak{v}_\varepsilon^\psi = \left(m_\varepsilon-\varepsilon(\psi^-_{B_\varepsilon(x_0)})^{-1}\left(\frac{1}{2\varepsilon}\right)\right) \wedge T_\varepsilon w(x) \vee \left(m_\varepsilon+\varepsilon(\psi_{B_\varepsilon(x_0)}^-)^{-1}\left(\frac{1}{2\varepsilon}\right)\right).
		\end{equation}
		Using property \hyperlink{adA1}{(adA1)} of $\psi$, Remark \ref{RESCALE} and inequality \eqref{P2}, we infer that there exist $\varepsilon_0>0$ and $C>0$ depending on the dimension and on $\psi$ but not on $\varepsilon$ such that, for every $\varepsilon<\varepsilon_0$,
		\begin{align*}
			\int_{B_\varepsilon(x_0)} \psi\left(x,\frac{|\mathfrak{v}_\varepsilon^\psi-m_\varepsilon|}{\varepsilon}\right) dx  
			\leq C \int_{B_\varepsilon(x_0)} \psi_\varepsilon^- \left( \left| \nabla w(x) \right| \right) \, dx+\omega_d\varepsilon^d \psi_{\varepsilon_0}^+(\sigma).
		\end{align*}
		Hence, keeping in mind \eqref{1.13},
		\begin{equation}\label{1.15}
			\lim_{\varepsilon \to 0^+} \frac{1}{\varepsilon^{d-1}} \int_{B_\varepsilon(x_0)} \psi\left(x,\frac{|\mathfrak{v}_\varepsilon^\psi-m_\varepsilon|}{\varepsilon}\right) \, dx=0.
		\end{equation}
		We now want to show that
		\begin{equation}\label{1.16}
			\lim_{\varepsilon \to 0^+} \frac{1}{\varepsilon^{d-1}} \int_{B_\varepsilon(x_0)} \psi^+_\varepsilon \left(\frac{|m_\varepsilon-u^+(x_0)|}{\varepsilon} \right) \, dx=\lim_{\varepsilon \to 0^+} \varepsilon\omega_d \psi^+_\varepsilon \left(\frac{|m_\varepsilon-u^+(x_0)|}{\varepsilon} \right)= 0.
		\end{equation}
		To this aim observe that condition \eqref{1.11} and the fact that $w(x_0)=u^+(x_0)$ imply that
		\begin{equation}\label{1.17}
			\lim_{\varepsilon \to 0^+} \med(w,B_\varepsilon(x_0))=u^+(x_0).
		\end{equation}
		Assume that $|m_\varepsilon-u^+(x_0)| \geq \varepsilon \sigma$ for every $\varepsilon>0$ small enough, otherwise the thesis is obvious. The function $\psi_\varepsilon^- \in \Phi_w$ satisfies \hyperlink{id}{$\textnormal{(Inc)}_\gamma$} with the same $\gamma>1$ of $\psi$ and it is doubling with the same constant $K \geq 2$ as $\psi$. Therefore, by Theorem \ref{weak and strong equiv} and Lemma \ref{equiv properties}, for every $\varepsilon>0$ there exists an Orlicz function $\Psi_\varepsilon \in \Phi_s$ which satisfies \hyperlink{id}{$\textnormal{(Inc)}_\gamma$}, it is doubling with the same constant $K$ and
		$$
		\Psi_\varepsilon(t) \leq \psi_\varepsilon^-(t) \leq \Psi_\varepsilon(2t), \qquad \Psi_\varepsilon^{-1}(t) \leq 2(\psi_\varepsilon^-)^{-1}(t) \leq 2 \Psi_\varepsilon^{-1}(t) \ \ \ \ \mbox{for every $t \geq 0$.}
		$$
		In particular, by continuity, $\Psi_\varepsilon(\Psi_\varepsilon^{-1}(t))=\Psi_\varepsilon^{-1}(\Psi_\varepsilon(t))=t$ for every $t \geq 0$.
		We claim that for every $\delta \in (0,1)$ and for every $\varepsilon>0$ it holds
		\begin{equation}\label{1.18}
			(\Psi_\varepsilon)^{-1}\left(\frac{t}{\delta}\right) \leq \delta^{-\frac{1}{\gamma}} (\Psi_\varepsilon)^{-1}(t) \ \ \ \mbox{for every $t>0$.}
		\end{equation}
		Indeed, this can be checked using the relation $\Psi_\varepsilon(\lambda s) \geq \lambda^\gamma\Psi_\varepsilon(s)$ for $\lambda \geq 1$ given by the property \hyperlink{id}{$\textnormal{(Inc)}_\gamma$}, with $\lambda:=\delta^{-\frac{1}{\gamma}}$ and $s:=(\Psi_\varepsilon)^{-1}(t)$. Fix $\delta \in (3/4,1)$ and $\delta \varepsilon<\rho<\varepsilon$. We set
		$$
		\widehat{w}:=(w \wedge \tau''(w,B_{\varepsilon}(x_0)) \wedge \tau''(w,B_\rho(x_0))) \vee (\tau'(w,B_{\varepsilon}(x_0)) \vee \tau'(w,B_\rho(x_0))).
		$$ 
		By Theorem \ref{DGCL 3}, the following inequalities hold
		\begin{equation}\label{masi}
			|\widehat{w}-m_\varepsilon |\leq |T_\varepsilon w-m_\varepsilon| \qquad |\widehat{w}-m_\rho| \leq |T_\rho w - m_\rho|.
		\end{equation}
		Using property \hyperlink{id}{$\textnormal{(Inc)}_\gamma$}, the inequalities in \eqref{masi}, Theorem \ref{Poinc homogeneous form} applied to $\Psi_\varepsilon$ and \eqref{1.14}, we get that
		\begin{align*}
			\Psi_\varepsilon \left(\frac{|m_\varepsilon - m_\rho|}{\varepsilon}\right) & = \fint_{B_\rho(x_0)} \Psi_\varepsilon \left(\frac{|m_\varepsilon - m_\rho|}{\varepsilon}\right) \, dx \\
			&\leq  2^q\fint_{B_\rho(x_0)} \Psi_\varepsilon \left(\frac{|\widehat{w}-m_\rho |}{\varepsilon}\right) \, dx +2^q \fint_{B_\rho(x_0)} \Psi_\varepsilon \left(\frac{|\widehat{w} - m_\varepsilon|}{\varepsilon}\right) \, dx \\
			& \leq 2^q \fint_{B_\rho(x_0)} \Psi_\varepsilon \left(\frac{|\widehat{w}-m_\rho |}{\rho}\right) \, dx + \frac{2^q}{\delta^d}\fint_{B_\varepsilon(x_0)} \Psi_\varepsilon \left(\frac{|\widehat{w} - m_\varepsilon|}{\varepsilon}\right) \, dx \\
			& \leq C\fint_{B_\rho(x_0)}  \Psi_\varepsilon \left(|\nabla w(x)| \right) \, dx + \frac{C}{\delta^d}\fint_{B_\varepsilon(x_0)} \Psi_\varepsilon \left(|\nabla w(x)|\right) \, dx \\
			& \leq C \eta \varepsilon^{-1}.
		\end{align*}
		where $C=C(d,\psi)$. Reasoning analogously one also gets that there exists $C>0$ not depending on $\varepsilon$ and on $k \geq 1$ such that it holds
		\begin{equation*}
			\Psi_\varepsilon \left( \frac{|m_{\delta^k \varepsilon}-m_{\delta^{k+1}\varepsilon}|}{\delta^k \varepsilon} \right) \leq C\eta (\delta^k\varepsilon)^{-1}.
		\end{equation*}
		Let $\rho_k$ be an infinitesimal sequence such that $\delta^{k+1} \varepsilon \leq \rho_k \leq \delta^k \varepsilon$ for every $k \geq 1$. Combining the estimates above, we have that
		\begin{align}\label{kjhrk}
			\begin{split}
			|m_\varepsilon - m_{\rho_k}| & \leq |m_{\delta^k \varepsilon}-m_{\rho_k}|+\sum_{j=0}^{k-1} |m_{\delta^{j+1} \varepsilon}-m_{\delta^j \varepsilon}| \\
			& \leq \delta^k \varepsilon \Psi_\varepsilon^{-1}\left(C \eta (\delta^k\varepsilon)^{-1} \right)+\sum_{j=0}^{k-1} \delta^j \varepsilon \Psi_\varepsilon^{-1}\left( C\eta (\delta^j\varepsilon)^{-1} \right).
			\end{split}
		\end{align}
		In the following $C>0$ will denote a generic constant depending only on the dimension $d$ and on $\psi$. Let $k \to +\infty$ in \eqref{kjhrk}, keeping in mind \eqref{1.17} and \eqref{1.18}, we have for every $\varepsilon>0$ small enough
		\begin{align*}
			|m_\varepsilon - u^+(x_0)| & \leq \sum_{j=0}^{\infty} \delta^j \varepsilon \Psi_\varepsilon^{-1}\left( C \eta (\delta^j\varepsilon)^{-1} \right) \leq C \sum_{j=0}^{\infty} \delta^j \varepsilon \Psi_\varepsilon^{-1}\left( \eta (\delta^j\varepsilon)^{-1} \right) \\
			& \leq \varepsilon C  \sum_{j=0}^{\infty} (\delta^j)^{1-\frac{1}{\gamma}}  \Psi_\varepsilon^{-1}\left( \eta \varepsilon^{-1} \right)=\varepsilon C \frac{1}{1-\delta^{(\gamma-1)/\gamma}} \Psi_\varepsilon^{-1}\left( \eta \varepsilon^{-1} \right) \\
			& \leq \varepsilon C \Psi_\varepsilon^{-1}\left( \eta \varepsilon^{-1} \right).
		\end{align*}
		Taking $\eta$ smaller if necessary, we get that 
		$$
		C\Psi_\varepsilon^{-1}(\eta \varepsilon^{-1}) \leq \frac{1}{2}\Psi_\varepsilon^{-1} (2^{-1} \varepsilon^{-1}) \leq (\psi_\varepsilon^-)^{-1}(2^{-1} \varepsilon^{-1}).
		$$ 
		Thus, by property \hyperlink{adA1}{(adA1)} of $\psi$ and using that $\Psi_\varepsilon$ is doubling, we deduce that 
		$$
		\psi^+_\varepsilon \left( \frac{|m_\varepsilon-u^+(x_0)|}{\varepsilon} \right) \leq C  \Psi_\varepsilon \left( 2\frac{|m_\varepsilon-u^+(x_0)|}{\varepsilon} \right) \leq C \Psi_\varepsilon \left( \frac{|m_\varepsilon-u^+(x_0)|}{\varepsilon} \right).
		$$
		Therefore,
		\begin{align*}
			\lim_{\varepsilon \to 0^+} \varepsilon \psi^+_\varepsilon \left( \frac{|m_\varepsilon-u^+(x_0)|}{\varepsilon} \right) & \leq \lim_{\varepsilon \to 0^+} \varepsilon C \Psi_\varepsilon \left( \frac{|m_\varepsilon-u^+(x_0)|}{\varepsilon} \right) \\
			& \leq \lim_{\varepsilon \to 0^+} \varepsilon C \Psi_\varepsilon \left( C \Psi_\varepsilon^{-1} \left( \frac{\eta}{\varepsilon} \right) \right) \leq C \eta.
		\end{align*}
		Since $\eta>0$ was arbitrarily small, letting $\eta \to 0^+$ we infer \eqref{1.16}.
		Combining \eqref{1.15} and \eqref{1.16} we deduce that
		\begin{equation}\label{1.19}
			\lim_{\varepsilon \to 0^+} \frac{1}{\varepsilon^{d-1}} \int_{B_\varepsilon(x_0)} \psi\left(x,\frac{|\mathfrak{v}_\varepsilon^\psi(x)-u^+(x_0)|}{\varepsilon}\right) \, dx=0.
		\end{equation}
		
		We repeat the above arguments for the function defined as
		\begin{equation*}
			z(x):=
			\begin{cases*}
				u(x',x_d) & if $x_d < \Gamma_j(x')$, \\
				u(x',-x_d+2\Gamma_j(x')) & if $x_d > \Gamma_j(x')$.
			\end{cases*}
		\end{equation*}
		We have that \eqref{1.12} and \eqref{1.13} are satisfied with $w$ replaced by $z$. Moreover, let $\mathfrak{z}^\psi_\varepsilon$ be defined as in \eqref{dp} with $w$ replaced by $z$, reasoning analogously we deduce 
		\begin{equation}\label{1.19bis}
			\lim_{\varepsilon \to 0^+} \frac{1}{\varepsilon^{d-1}} \int_{B_\varepsilon(x_0)} \psi\left(x,\frac{|\mathfrak{z}_\varepsilon^\psi(x)-u^-(x_0)|}{\varepsilon}\right) \, dx=0.
		\end{equation}
		We set
		\begin{equation*}
			v_\varepsilon(x):=
			\begin{cases*}
				\displaystyle \mathfrak{v}_\varepsilon^\psi(x) & if $x_d > \Gamma_j(x')$, \\
				\displaystyle \mathfrak{z}_\varepsilon^\psi(x) & if $x_d < \Gamma_j(x')$.
			\end{cases*}
		\end{equation*}
		Using \eqref{1.8}, since $|\nabla v_\varepsilon| \leq |\nabla u|$ on $B_\varepsilon(x_0)$ by definition, we get
		\begin{equation}\label{1.20}
			\lim_{\varepsilon \to 0^+} \frac{1}{\varepsilon^{d-1}} \int_{B_{\varepsilon}(x_0)} \psi(x,|\nabla v_\varepsilon(x)|) \, dx=0.
		\end{equation}
		From \eqref{1.19} and \eqref{1.19bis} we also infer that
		\begin{equation}\label{1.21}
			\lim_{\varepsilon \to 0^+} \frac{1}{\varepsilon^{d-1}} \int_{B_\varepsilon(x_0)} \psi\left(x,\frac{|v_\varepsilon(x)-u_{x_0,\nu}(x)|}{\varepsilon} \right) \, dx=0.
		\end{equation}
		By definition of $\mathfrak{v}_\varepsilon^\psi$, $\mathfrak{z}_\varepsilon^\psi$, by Chebychev inequality, by \eqref{truncated difference} and by Theorem \ref{Poinc homogeneous form}, we estimate
		\begin{align*}
		\mathcal{L}^d \left( \left\{  \mathfrak{v}_\varepsilon^\psi \neq   w \right\} \cup \left\{ \mathfrak{z}_\varepsilon^\psi \neq  z \right\} \right)
		& \leq \mathcal{L}^d \left( \left\{ \mathfrak{v}_\varepsilon^\psi \neq T_\varepsilon w \right\}\right)+\mathcal{L}^d \left( \left\{ \mathfrak{z}_\varepsilon^\psi \neq T_\varepsilon z \right\}\right)+\mathcal{L}^d \left( \{T_\varepsilon w \neq w  \} \right)+\mathcal{L}^d \left( \{T_\varepsilon z\neq z  \} \right) \\
		& \leq \varepsilon C\int_{B_\varepsilon(x_0)} \psi^-_\varepsilon(|\nabla w(x)|) \, dx +\varepsilon C\int_{B_\varepsilon(x_0)} \psi^-_\varepsilon(|\nabla z(x)|) \, dx \\ 
		&  \ \ \ \ \ + \left( 2\gammaiso\mathcal{H}^{d-1}(J_w \cap B_\varepsilon(x_0)) \right)^{\frac{d}{d-1}}+\left( 2\gammaiso\mathcal{H}^{d-1}(J_z \cap B_\varepsilon(x_0)) \right)^{\frac{d}{d-1}},
		\end{align*}
		where $C=C(d,K)$. Thus, by definition of $v_\varepsilon$, \eqref{1.12} and \eqref{1.13} for $w$ and $z$, we infer that 
		\begin{equation}\label{1.22}
			\lim_{\varepsilon \to 0^+} \varepsilon^{-d} \mathcal{L}^d (\{x \in B_\varepsilon(x_0) \colon v_\varepsilon(x) \neq u(x) \})=0.
		\end{equation}
		Now, an analogous argument as in \eqref{1.2} shows that for every sequence $\varepsilon \to 0^+$ one can find a subsequence (not relabeled) such that for $\mathcal{L}^1$-a.e. $\lambda \in (0,1)$
		\begin{equation}\label{1.23}
			\mathcal{H}^{d-1}(\partial B_{\lambda \varepsilon}(x_0) \cap J_{v_\varepsilon})=0 \qquad \lim_{\varepsilon \to 0^+} \varepsilon^{1-d} \mathcal{H}^{d-1}(\{ v_\varepsilon \neq u \} \cap \partial B_{\lambda \varepsilon}(x_0))=0.
		\end{equation}
		
		We finally define
		\begin{equation*}
			u_\varepsilon(x):=
			\begin{cases*}
				v_\varepsilon(x) & if $x \in B_{\lambda\varepsilon}(x_0) $, \\
				u(x) & if $x \in B_\varepsilon(x_0) \setminus \overline{B_{\lambda \varepsilon}(x_0)}$.
			\end{cases*}
		\end{equation*}
		By definition it is clear that $u_\varepsilon \in \gsbv^{\psi}(B_\varepsilon(x_0))$.
		Properties \eqref{jump seq props}(i)--(ii) follow by definition of $u_\varepsilon$, \eqref{1.22} and \eqref{1.23}. Property \eqref{jump seq props}(iii) follows by \eqref{1.8} and \eqref{1.20}. Finally, property \eqref{jump seq props}(iv) is a consequence of the definition of $u_\varepsilon$ and \eqref{1.21}.
	\end{proof}

	With this approximation tool, we can now prove Lemma \ref{lem 3} which will follow as a consequence of Lemma \ref{lem 3.1} and Lemma \ref{lem 3.2}.
	
	\begin{lem}\label{lem 3.1}
		Let $\mathcal{F}$ satisfy \eqref{H1} and \eqref{H3}-\eqref{H4} and let $u \in \gsbv^{\psi}(\Omega,\R^m)$. Then, for $\mathcal{H}^{d-1}$-a.e. $x_0 \in J_u$ it holds
		\begin{equation}\label{3.1}
			\lim_{\varepsilon \to 0^+} \frac{\mathbf{m}_\mathcal{F}(u,B_\varepsilon(x_0))}{\omega_{d-1} \varepsilon^{d-1}} \leq \limsup_{\varepsilon \to 0^+} \frac{\mathbf{m}_\mathcal{F}(\overline{u}^\textnormal{surf}_{x_0},B_\varepsilon(x_0))}{\omega_{d-1} \varepsilon^{d-1}}.
		\end{equation}
	\end{lem}
	
	\begin{proof}
		Let $x_0 \in J_u$ such that the statement of Lemma \ref{surf seq} holds and
		\begin{equation}\label{3.3}
			\lim_{\varepsilon \to 0^+} \frac{\mathcal{F}(u,B_\varepsilon(x_0))}{\mu(B_\varepsilon(x_0))} \leq \lim_{\varepsilon \to 0^+} \frac{\mathbf{m}_\mathcal{F}(u,B_\varepsilon(x_0))}{\omega_{d-1}\varepsilon^{d-1}} < +\infty.
		\end{equation} 
		By Lemma \ref{lem 1} and recalling the definition of $\mu$, \eqref{3.3} holds $\mathcal{H}^{d-1}$-a.e. in $ J_u \cap \Omega$.
		Let $(u_\varepsilon)_\varepsilon$ be the sequence given by said Lemma; in the following we set $\nu:=\nu_u(x_0)$ for brevity. 
		Fix $\eta>0$, take $\lambda \in (0,1)$ such that \eqref{jump seq props} holds, and set $\lambda=1-\theta$ with $\theta \in (0,1)$. Take a sequence $z_\varepsilon \in \gsbv^{\psi}(B_{(1-3\theta)\varepsilon}(x_0),\R^m)$ with $z_\varepsilon=\usurf$ in a neighborhood of $\partial B_{(1-3\theta)\varepsilon}(x_0)$ and such that
		\begin{equation}\label{3.4}
			\mathcal{F}(z_\varepsilon,B_{(1-3\theta)\varepsilon}(x_0)) \leq \mathbf{m}_\mathcal{F}(\usurf,B_{(1-3\theta)\varepsilon}(x_0))+\omega_{d-1}\varepsilon^d.
		\end{equation}
		We extend $z_\varepsilon$ to a function in $\gsbv^{\psi}(B_\varepsilon(x_0),\R^m)$ by setting $z_\varepsilon=\usurf$ outside $B_{(1-3\theta)\varepsilon}(x_0)$. Now we apply Lemma \ref{fund est} with $u$ and $v$ replaced by $z_\varepsilon$ and $u_\varepsilon$, respectively, and $D'_{\varepsilon,x_0}$, $D''_{\varepsilon,x_0}$, $E_{\varepsilon,x_0}$ defined as in \eqref{2.18}. Thus, we find $w_\varepsilon \in \gsbv^{\psi}(B_\varepsilon(x_0),\R^m)$ such that $w_\varepsilon=u_\varepsilon$ on $B_\varepsilon(x_0) \setminus B_{(1-\theta)\varepsilon}(x_0)$ and
		\begin{align}\label{3.5}
			\begin{split}
				\mathcal{F}(w_\varepsilon,B_\varepsilon(x_0)) \leq & (1+\eta)\left(\mathcal{F}(z_\varepsilon,B_{(1-\theta)\varepsilon}(x_0))+\mathcal{F}(u_\varepsilon,C_{\varepsilon,\theta}(x_0))\right)+\eta\mathcal{L}^d(B_\varepsilon(x_0)) \\
				& +M\int_{B_{(1-\theta)\varepsilon}(x_0) \setminus B_{(1-2\theta)\varepsilon}(x_0)} \psi \left(x,\frac{|z_\varepsilon-u_\varepsilon|}{\varepsilon}\right) \, dx.
			\end{split}
		\end{align}
		By \eqref{jump seq props}(i) we have that $w_\varepsilon=u_\varepsilon=u$ on a neighborhood of $\partial B_\varepsilon(x_0)$. Moreover, using the fact that $z_\varepsilon=\usurf$ outside $B_{(1-3\theta)\varepsilon}(x_0)$, by \eqref{jump seq props}(iv) we deduce
		\begin{align*}
			\lim_{\varepsilon \to 0^+} \frac{1}{\varepsilon^{d-1}} & \int_{B_{(1-\theta)\varepsilon}(x_0) \setminus B_{(1-2\theta)\varepsilon}(x_0)} \psi \left(x,\frac{|z_\varepsilon-u_\varepsilon|}{\varepsilon}\right) \, dx  \\
			& \leq  \lim_{\varepsilon \to 0^+} \frac{1}{\varepsilon^{d-1}} \int_{B_{(1-\theta)\varepsilon}(x_0)} \psi \left(x,\frac{|u_\varepsilon-\usurf|}{\varepsilon}\right) \, dx=0.
		\end{align*}
		Recall that $C_{\varepsilon,\theta}(x_0):=B_\varepsilon(x_0) \setminus \overline{B_{(1-4\theta)\varepsilon}(x_0)}$ as defined in Lemma \ref{2.1}. Together with \eqref{3.5} we infer that there exists a non-negative infinitesimal sequence $(\rho_\varepsilon)_\varepsilon$ such that
		\begin{equation}\label{3.6}
			\mathcal{F}(w_\varepsilon,B_\varepsilon(x_0)) \leq  (1+\eta)\left(\mathcal{F}(z_\varepsilon,B_{(1-\theta)\varepsilon}(x_0))+\mathcal{F}(u_\varepsilon,C_{\varepsilon,\theta}(x_0))\right)+\varepsilon^{d-1}\rho_\varepsilon+\eta\omega_d\varepsilon^d.
		\end{equation}
		We now estimate the terms in \eqref{3.6}. Let $\Pi_0$ be the hyperplane passing through $x_0$ with normal $\nu$. 
		Using that $z_\varepsilon=u_\varepsilon$ on $B_\varepsilon(x_0) \setminus B_{(1-3\theta)\varepsilon}(x_0) \subset C_{\varepsilon,\theta}(x_0)$ together with \eqref{H1}, \eqref{H4} and \eqref{3.4} we get
		\begin{equation}\label{3.15}
			\limsup_{\varepsilon \to 0^+} \frac{\mathcal{F}(\usurf,C_{\varepsilon,\theta}(x_0))}{\omega_{d-1}\varepsilon^{d-1}} \leq b(1-(1-4\theta)^{d-1}),
		\end{equation}
		and
		\begin{align}
				\limsup_{\varepsilon \to 0^+} \label{3.7} \frac{\mathcal{F}(z_\varepsilon,B_{(1-\theta)\varepsilon}(x_0))}{\omega_{d-1}\varepsilon^{d-1}} & \leq \limsup_{\varepsilon \to 0^+} \frac{\mathcal{F}(z_\varepsilon,B_{(1-3\theta)\varepsilon}(x_0))}{\omega_{d-1}\varepsilon^{d-1}}+\limsup_{\varepsilon \to 0^+} \frac{\mathcal{F}(\usurf,C_{\varepsilon,\theta}(x_0))}{\omega_{d-1}\varepsilon^{d-1}} \\
				& \leq \limsup_{\varepsilon \to 0^+} \frac{\mathbf{m}_\mathcal{F}(\usurf,B_{(1-3\theta)\varepsilon}(x_0))}{\omega_{d-1}\varepsilon^{d-1}}+\frac{b}{\omega_{d-1}}\mathcal{H}^{d-1}(C_{1,\theta}(x_0)\cap \Pi_0) \nonumber \\
				& \leq (1-3\theta)^{d-1} \limsup_{\varepsilon \to 0^+} \frac{\mathbf{m}_\mathcal{F}(\usurf,B_{\varepsilon}(x_0))}{\omega_{d-1}\varepsilon^{d-1}}+\frac{b}{\omega_{d-1}}(1-(1-4\theta)^{d-1}) \nonumber .
		\end{align}
		Notice that by rectifiability of $J_u$ and \eqref{jump seq props}(ii) it holds
		\begin{align}\label{3.8}
			\begin{split}
				\limsup_{\varepsilon \to 0^+} \frac{\mathcal{H}^{d-1}(J_{u_\varepsilon} \cap C_{\varepsilon,\theta}(x_0))}{\omega_{d-1}\varepsilon^{d-1}} \leq \lim_{\varepsilon \to 0^+} \frac{\mathcal{H}^{d-1}(J_{u} \cap C_{\varepsilon,\theta}(x_0))}{\omega_{d-1}\varepsilon^{d-1}}=\frac{\mathcal{H}^{d-1}(\Pi_0 \cap C_{1,\theta}(x_0))}{\omega_{d-1}}
			\end{split}
		\end{align}
		Therefore, using \eqref{jump seq props}(iii) and \eqref{H4} again we obtain
		\begin{align}\label{3.9}
			\begin{split}
				\limsup_{\varepsilon \to 0^+} \frac{\mathcal{F}(u_\varepsilon,C_{\varepsilon,\theta}(x_0))}{\omega_{d-1}\varepsilon^{d-1}} & \leq \limsup_{\varepsilon \to 0^+} \frac{b}{\omega_{d-1}\varepsilon^{d-1}}\left( \int_{C_{\varepsilon,\theta}(x_0)} (1+\psi(x,|\nabla u_\varepsilon|)) \, dx + \mathcal{H}^{d-1}(J_{u_\varepsilon} \cap C_{\varepsilon,\theta}(x_0)) \right) \\
				&  \leq b(1-(1-4\theta)^{d-1}).
			\end{split}
		\end{align}
		Finally, recalling that $\rho_\varepsilon \to 0$ as $\varepsilon \to 0$, that $w_\varepsilon=u$ in a neighborhood of $\partial B_\varepsilon(x_0)$, and using \eqref{3.6}, \eqref{3.7} and \eqref{3.9}, we deduce
		\begin{align*}
			\lim_{\varepsilon \to 0^+} \frac{\mathbf{m}_\mathcal{F}(u,B_\varepsilon(x_0))}{\omega_{d-1}\varepsilon^{d-1}} & \leq \limsup_{\varepsilon \to 0^+} \frac{\mathcal{F}(w_\varepsilon,B_\varepsilon(x_0))}{\omega_{d-1}\varepsilon^{d-1}} \\
			& \leq (1+\eta)(1-3\theta)^{d-1} \limsup_{\varepsilon \to 0^+} \frac{\mathbf{m}_\mathcal{F}(\usurf,B_\varepsilon(x_0))}{\omega_{d-1}\varepsilon^{d-1}}
			+2b(1+\eta)(1-(1-4\theta)^{d-1}).
		\end{align*}
		Letting $\eta \to 0$ and $\theta \to 0$ we get \eqref{3.1}.
	\end{proof}
	
	\begin{lem}\label{lem 3.2}
		Let $\mathcal{F}$ satisfy \eqref{H1} and \eqref{H3}-\eqref{H4} and let $u \in \gsbv^{\psi}(\Omega,\R^m)$. Then, for $\mathcal{H}^d$-a.e $x_0 \in J_u$ we have
		\begin{equation}\label{3.2}
			\lim_{\varepsilon \to 0^+} \frac{\mathbf{m}_\mathcal{F}(u,B_\varepsilon(x_0))}{\omega_{d-1} \varepsilon^{d-1}} \geq \limsup_{\varepsilon \to 0^+} \frac{\mathbf{m}_\mathcal{F}(\overline{u}^\textnormal{surf}_{x_0},B_\varepsilon(x_0))}{\omega_{d-1} \varepsilon^{d-1}}.
		\end{equation}
	\end{lem}
	
	\begin{proof}
		Again, we prove the assertion for the points $x_0 \in J_u$ considered in the proof of Lemma \ref{lem 3.1}. Fix $\eta>0$ and let $(u_\varepsilon)_\varepsilon$ be the sequence given by Lemma \ref{surf seq}. By \eqref{jump seq props}(i) and Fubini's Theorem it follows that
		\begin{equation*}
			\lim_{\varepsilon \to 0^+} \frac{1}{\varepsilon^{d-1}}\int_0^1 \mathcal{H}^{d-1}(\{ u_\varepsilon \neq u \} \cap \partial B_{\lambda \varepsilon}(x_0))\,d\lambda = \lim_{\varepsilon \to 0^+} \frac{2}{\varepsilon^d} \mathcal{L}^d(\{ u_\varepsilon \neq u \} \cap  B_{ \varepsilon}(x_0)) = 0.
		\end{equation*}
		Thus, given $\theta \in (0,1)$, for every $\varepsilon >0$ there exists $\lambda \in (1-4\theta,1-3\theta)$ such that
		\begin{align}\label{3.10}
			\begin{split}
				& \mathcal{H}^{d-1}((J_{u_\varepsilon} \cup J_u) \cap \partial B_{\lambda \varepsilon}(x_0))=0 \ \ \ \mbox{for every $\varepsilon>0$}, \\
				& \lim_{\varepsilon \to 0^+} \varepsilon^{-(d-1)} \mathcal{H}^{d-1}(\{ u_\varepsilon \neq u \} \cap \partial B_{\lambda \varepsilon}(x_0))=0.
			\end{split}
		\end{align}
		Take a sequence $z_\varepsilon \in \gsbv^{\psi}(B_{\lambda \varepsilon}(x_0),\R^m)$ with $z_\varepsilon = u$ in a neighborhood of $\partial B_{\lambda \varepsilon}(x_0) $ and such that
		\begin{equation}\label{3.11}
			\mathcal{F}(z_\varepsilon,B_{\lambda\varepsilon}(x_0)) \leq \mathbf{m}_\mathcal{F}(u,B_{\lambda\varepsilon}(x_0))+\omega_{d-1} \varepsilon^{d}.
		\end{equation}
		Again, we can extend $z_\varepsilon$ to a function in $\gsbv^{\psi}(B_\varepsilon(x_0),\R^m)$ by setting $z_\varepsilon=u$ outside $B_{\lambda\varepsilon}(x_0)$. We apply Lemma \ref{fund est} with $u$ and $v$ replaced by $z_\varepsilon$ and $\usurf$ respectively, with the sets $D'_{\varepsilon,x_0}$, $D''_{\varepsilon,x_0}$ and $E_{\varepsilon,x_0}$ defined as in \eqref{2.18}. We thus find $w_\varepsilon \in \gsbv^{\psi}(B_\varepsilon(x_0),\R^m)$ such that $w_\varepsilon = \usurf$ in a neighborhood of $\partial B_\varepsilon(x_0)$ and 
		\begin{align*}
			\mathcal{F}(w_\varepsilon,B_\varepsilon(x_0)) \leq & (1+\eta) \left( \mathcal{F}(z_\varepsilon,B_{(1-\theta)\varepsilon}(x_0))+\mathcal{F}(\usurf,C_{\varepsilon,\theta}(x_0)) \right)+\eta \mathcal{L}^d(B_\varepsilon(x_0)) \\
			& + M \int_{B_{(1-\theta)\varepsilon}(x_0) \setminus B_{(1-2\theta)\varepsilon}(x_0)} \psi \left( x,\frac{|z_\varepsilon-\usurf|}{\varepsilon} \right) \, dx.
		\end{align*}
		By our initial choice of $\lambda$ then, $z_\varepsilon=u_\varepsilon$ outside of $B_{(1-3\theta)\varepsilon}(x_0)$. Therefore, using \eqref{jump seq props}(iv), there exists a non-negative infinitesimal sequence $(\rho_\varepsilon)_\varepsilon$ such that 
		\begin{align}\label{3.12}
			\mathcal{F}(w_\varepsilon,B_\varepsilon(x_0)) \leq & (1+\eta) \left( \mathcal{F}(z_\varepsilon,B_{(1-\theta)\varepsilon}(x_0))+\mathcal{F}(\usurf,C_{\varepsilon,\theta}(x_0)) \right)+\rho_\varepsilon\varepsilon^{d-1}+\eta \omega_d\varepsilon^d.
		\end{align}
		Now what is left is to estimate the terms in \eqref{3.12}. Recalling that by our choice of $\lambda$ we have $z_\varepsilon=u_\varepsilon$ on $B_\varepsilon(x_0) \setminus B_{\lambda\varepsilon}(x_0)$ and using \eqref{H1}, \eqref{H4} and \eqref{3.11} we deduce
		\begin{align}\label{3.13}
			\begin{split}
				\mathcal{F}(z_\varepsilon,B_{(1-\theta)\varepsilon}(x_0)) \leq & \mathbf{m}_\mathcal{F}(u,B_{\lambda\varepsilon}(x_0))+\omega_{d-1}\varepsilon^d+\mathcal{F}(u_\varepsilon,C_{\varepsilon,\theta}(x_0)) \\
				& +b \mathcal{H}^{d-1}\left((\{u_\varepsilon \neq u\} \cup J_{u_\varepsilon} \cup J_u) \cap \partial B_{\lambda\varepsilon}(x_0)\right).
			\end{split}
		\end{align}
		Since $\lambda \leq 1-3\theta$, using the estimate in \eqref{3.9} and \eqref{3.10} we get
		\begin{align}\label{3.14}
			\begin{split}
				\limsup_{\varepsilon \to 0^+} \frac{\mathcal{F}(z_\varepsilon,B_{(1-\theta)\varepsilon}(x_0))}{\omega_{d-1}\varepsilon^{d-1}} & \leq \limsup_{\varepsilon \to 0^+} \frac{\mathbf{m}_\mathcal{F}(u,B_{\lambda\varepsilon}(x_0))}{\omega_{d-1}\varepsilon^{d-1}}+b(1-(1-4\theta)^{d-1}) \\
				& \leq (1-3\theta)^{d-1}\limsup_{\varepsilon \to 0^+} \frac{\mathbf{m}_\mathcal{F}(u,B_{\varepsilon}(x_0))}{\omega_{d-1}\varepsilon^{d-1}}+b(1-(1-4\theta)^{d-1}).
			\end{split}
		\end{align}
		Thus, collecting the estimates \eqref{3.15}, \eqref{3.12} and \eqref{3.14} we obtain
		$$
		\limsup_{\varepsilon \to 0^+} \frac{\mathcal{F}(w_\varepsilon,B_\varepsilon(x_0))}{\omega_{d-1}\varepsilon^{d-1}} \leq (1+\eta) \left( (1-3\theta)^{d-1}\limsup_{\varepsilon \to 0^+}\frac{\mathbf{m}_\mathcal{F}(u,B_\varepsilon(x_0))}{\omega_{d-1}\varepsilon^{d-1}}+2b(1-(1-4\theta)^{d-1}) \right).
		$$
		Finally, since $w_\varepsilon = \usurf$ in a neighborhood of $\partial B_\varepsilon(x_0)$ and letting $\eta \to 0$ and $\theta \to 0$ we infer that
		\begin{align*}
			\limsup_{\varepsilon \to 0^+} \frac{\mathbf{m}_\mathcal{F}(\usurf,B_\varepsilon(x_0))}{\omega_{d-1}\varepsilon^{d-1}}& \leq \limsup_{\varepsilon \to 0^+} \frac{\mathcal{F}(w_\varepsilon,B_\varepsilon(x_0))}{\omega_{d-1}\varepsilon^{d-1}} \leq \limsup_{\varepsilon \to 0^+} \frac{\mathbf{m}_\mathcal{F}(u,B_\varepsilon(x_0))}{\omega_{d-1}\varepsilon^{d-1}} \\
			& = \lim_{\varepsilon \to 0^+} \frac{\mathbf{m}_\mathcal{F}(u,B_\varepsilon(x_0))}{\omega_{d-1}\varepsilon^{d-1}}.
		\end{align*}
		This gives \eqref{3.2}. 
	\end{proof}

		\section{Lower Semicontinuity}
		\label{s:lower}
	
	Recalling the assumptions on $\psi \in \Phi_w(\Omega)$ required by Theorem \ref{lsc of the functional}, and Remark \ref{rem: convexgratis}, throughout this section we will assume wiht no loss of generality that $\psi$ is a  function in $\Phi_s(\Omega)$ satisfying \hyperlink{A0}{$\textnormal{(A0)}$}, \hyperlink{idg}{$\textnormal{(Inc)}$}, \hyperlink{idg}{$\textnormal{(Dec)}$} and \eqref{.} on $\Omega$.  Before proving the lower semicontinuity, we need some preliminary results.
	We start by presenting a truncation Lemma. The original one \cite[Lemma 4.1]{CDMSZ} is presented in terms of a fixed exponent. We formulate it here in our setting. The adaptation of the original proof requires only some minor changes and it is hinged on the fact that $\psi(x,\cdot)$ is strictly increasing for every $x \in \Omega$.
	
	\begin{lem}\label{truncation lemma}
		Let $\mathcal{G}$ be as in \eqref{functional} and $f \colon \Omega \times \R^{m \times d} \to [0,+\infty)$ satisfy \eqref{f1}--\eqref{f2}. Let $\eta,\lambda>0$. There exists $\mu>\lambda$ depending on $\eta,\lambda,a,b$ such that the following holds: for every $A \in \mathcal{A}(\Omega)$ and every $u \in L^0(\R^d,\R^m)$ such that $u|_A \in \gsbv^{\psi}(A,\R^m)$, there exists $\tilde{u} \in L^\infty(\R^d,\R^m)$ such that $\tilde{u}|_A \in \sbv^{\psi}(A,\R^m)$ and
		\begin{enumerate}
			\item[\textnormal{(i)}] $|\tilde{u}| \leq \mu$ on $\R^d$;
			\item[\textnormal{(ii)}] $\tilde{u}=u $ $\mathcal{L}^d$-a.e. in $\{ |u| \leq \lambda \}$;
			\item[\textnormal{(iii)}] $\mathcal{G}(\tilde{u},A) \leq (1+\eta)\mathcal{G}(u,A)+b \mathcal{L}^d(A \cap \{ |u| \geq \lambda \})$.
		\end{enumerate}
	\end{lem}
	
	The next lemma concerns the approximation of $\sbv^\psi$ functions with Lipschitz functions in the unit ball.
	
	\begin{lem}[Lusin approximation in $\sbv^{\psi}$]\label{lusin approx}
		For every $u \in \sbv^{\psi}(B_1,\R^m)$ and every $\lambda >0$ there exists a Lipschitz function $v:B_1 \to \R^m$ satisfying $\lip(v) \leq \tau \lambda$ with $\tau=\tau(d,m)$, such that $v=u$ in $\{ M(Du) \leq \lambda \}$ and 
		\begin{equation}\label{lusin ineq}
			\mathcal{L}^d(A \cap \{ M(Du)>\lambda \}) \leq \frac{\tau}{\lambda} \int_{J_u} |u^+-u^-| \, d\mathcal{H}^{d-1}+ \frac{1}{ \psi_A^-(\lambda)} \int_{A \cap \{ M(|\nabla u|) > \lambda  \}} \psi(x,M(|\nabla u|)) \, dx,
		\end{equation}
		for any $A \in \mathcal{B}(B_1)$, where $M$ is the restricted maximal function to $B_1$ (see Definition \ref{maximal ope}).
	\end{lem}
	
	\begin{proof}
		The proof of this fact is no different than the standard Lusin approximation on the whole space $\R^d$ for $\sbv^p$ and builds upon the fact that
		$$
		\inf_{x,x' \in B_1, \ \rho=|x-x'|} \frac{\mathcal{L}^d(B_\rho(x) \cap B_\rho(x') \cap B_1)}{\rho^d}>0,
		$$
		see e.g. \cite[Theorem 5.34 and Theorem 5.36]{Ambrosio2000FunctionsOB}.
	\end{proof}
	
	We are finally ready to prove the lower semicontinuity. In our proof we will combine some arguments contained in \cite{Ambrosiolsc} with some ideas from \cite{marc} about the approximation of the integrand from below.
	
	\begin{prop}\label{100}
		Let $\psi \in \Phi_s(B_1)$ satisfy \hyperlink{A0}{$\textnormal{(A0)}$}, \hyperlink{idg}{$\textnormal{(Inc)}$}, \hyperlink{idg}{$\textnormal{(Dec)}$} on $B_1$ and property \eqref{.} in $x_0=0$.
		Consider $\{ \varepsilon_k \}_k$ an infinitesimal sequence and set $\psi_k:B_{1} \times [0,+\infty) \to [0,+\infty)$ as $$\psi_k(x,t):=\psi\left({\varepsilon_k}x,t\right).$$
		Let $a,b>0$ and $\{ f_k \}_k $ be a sequence of Carathéodory functions such that for each $k \geq 1$
		\begin{equation}\label{growth of seq}
			a \psi_k(x,|\xi|) \leq f_k(x,\xi) \leq b(1+\psi_k(x,|\xi|)) \ \ \ (x,\xi) \in B_1 \times \R^{m \times d}.
		\end{equation}
		Assume also that there exists a quasi-convex function $f$ such that
		\begin{equation}\label{conv to the qc}
			\lim_{k \to +\infty} f_k(x,\xi)=f(\xi)
		\end{equation}
		locally uniformly in $\R^{m \times d}$ for a.e. $x \in B_1$.
		Then,
		\begin{equation}\label{lminf}
			\liminf_{k \to +\infty} \int_{B_1} f_k(x,\nabla u_k) \, dx \geq \int_{B_1} f(\nabla u) \, dx
		\end{equation}
		for any sequence $u_k \in \gsbv(B_1,\R^m)$ converging in measure to a linear function $u \colon B_1 \to \R^m$ and satisfying $\mathcal{H}^{d-1}(J_{u_k} \cap B_1) \to 0$ as $k \to +\infty$.
	\end{prop}
	
	\begin{proof}
		Assume that $\psi $ satisfies \hyperlink{idg}{$\textnormal{(Inc)}_\gamma$} and \hyperlink{idg}{$\textnormal{(Dec)}_q$} with $\gamma>1$ and $q \in (1,\infty)$. Taking into account Remark \ref{RESCALE}, we have that for every $k \geq 1$, $\psi_k \in \Phi_s(B_{1})$ satisfies properties \hyperlink{A0}{$\textnormal{(A0)}$} with $\sigma=\sigma(\psi) \geq 1$, \hyperlink{idg}{$\textnormal{(Inc)}_\gamma$} and \hyperlink{idg}{$\textnormal{(Dec)}_q$} on $B_{1}$. 
		
		Assuming the liminf in \eqref{lminf} to be finite, from \eqref{growth of seq} we deduce that
		$$
		\sup_{k \geq 1} \int_{B_1} \psi_k (x,|\nabla u_k(x)|) \, dx < +\infty.
		$$
		Hence, $u_k \in \gsbv^{\psi_k}(B_1,\R^m)$. 
		
		We first use the truncation argument in Lemma \ref{truncation lemma} applied to $u_k$ with $\eta>0$ small and $\theta>1$ large enough, in order to find $\{ v_k \}_k \subset \sbv^{\psi_k}(B_1,\R^m)$ and $\mu>\theta$, such that $v_k \to u$ in measure, $\Vert v_k \Vert_{L^\infty} \leq \mu$ and, for every $k \geq 1$,
		\begin{equation}\label{1}
			\int_{B_1} f_k(x,\nabla v_k(x)) \, dx \leq (1+\eta) \int_{B_1} f_k(x,\nabla u_k(x)) + b \mathcal{L}^d(\{ |u_k| \geq \theta \}).
		\end{equation}
		Notice that for $\theta $ large enough $\mathcal{L}^d(\{ |u_k| \geq \theta \}) \to 0$ as $k \to +\infty$, since $u_k$ converges to a linear bounded function in measure. Therefore, if we prove the liminf inequality with $v_k$ in place of $u_k$, the thesis follows from \eqref{1} sending $\eta \to 0$.
		From now on $C>0$ will indicate a positive constant depending only on $\psi$. Taking into account \eqref{growth of seq} and \eqref{1}, we have
		\begin{equation}\label{2}
			\sup_{k \geq 1} \int_{B_1} \psi_k(x,|\nabla v_k(x)|) \, dx < +\infty
		\end{equation}
		and $|D^s v_k|(B_1) \leq 2 \mu \mathcal{H}^{d-1}(J_{v_k}) \to 0$ as $k \to +\infty$. Moreover, by (ii) of Lemma \ref{truncation lemma}, we have that for $\theta>0$ large enough, $v_k \to u$ in measure on $B_1$.
		
		
		Let $\psi_k^-(t):= \inf_{B_1} \psi_k(\cdot,t)$ and $\psi_k^+(t):= \sup_{B_1} \psi_k(\cdot,t)$. By Remark \ref{RESCALE} and Corollary \ref{max op bdd cor} applied to $\psi_k^-$, given a function $w \in L^{\psi_k}(B_1,\R^m)$, we have that there exists $C=C(d,\sigma,\gamma,q)>0$ such that for every $k \geq 1$
		\begin{equation}\label{30}
			\int_{B_1} {\psi}^-_k\left(M(w)(x)\right) \, dx \leq C	\left(\int_{B_1} {\psi}_k^-\left(|w(x)|\right) \, dx+1 \right)^q.
		\end{equation}
		Thus, by \eqref{2} and \eqref{30} we get that the sequence $\{ {\psi}_k^-(M(|\nabla v_k|)(x)) \}_k$ is bounded in $L^1(B_1)$. Hence, by Chacon biting lemma, we can find a sequence of sets $E_h \in \mathcal{B}(B_1)$ such that $\mathcal{L}^d(E_h) \to 0$ as $h \to +\infty$ and $\{{\psi}^-_k(M(|\nabla v_k|)) \chi_{B_1 \setminus E_h}\}_k$ is equiintegrable for every $h \geq 1$. Let
		$$
		\Lambda_h(s):= \sup \left\{ \limsup_{k \to +\infty} \int_{F} {\psi}^-_k(M(|\nabla v_k|)(x)) \, dx: \ F \in \mathcal{B}(B_1), \ F \subset B_1 \cap E_h, \ \mathcal{L}^d(F) \leq s  \right\}.
		$$
		Due to equiintegrability, we deduce that $\Lambda_h(s) \to 0$ as $s \searrow 0$ for every $h \geq 1$.
		
		Let $\lambda > 0$. Thanks to Lemma \ref{lusin approx} applied to $v_k$, we find functions $u_k^\lambda:B_1 \to \R^m$ and $\mathcal{L}^d$-measurable sets $E_k^\lambda$ such that 
		\begin{equation}\label{10}
			\lip(u_k^\lambda) \leq \tau \lambda \ \ \  \mbox{and} \ \ \ u_k^\lambda=v_k \mbox{ in $B_1 \setminus E_k^\lambda$},
		\end{equation}
		where $\tau=\tau(d,m)$.
		Using \eqref{lusin ineq} we deduce that for every $\lambda>0$ large enough and every $A \in \mathcal{B}(B_1)$ it holds
		\begin{equation}\label{3}
			\mathcal{L}^d(E_k^\lambda \setminus A) \leq \frac{\tau}{\lambda}  2 \mu \mathcal{H}^{d-1}(J_{v_k})+ \frac{1}{{\psi}_k^-(\lambda)} \int_{\{ M(|\nabla v_k|) > \lambda  \} \setminus A} {\psi}^-_k(M(|\nabla v_k|)) \, dx.
		\end{equation}
		Since $\{ v_k\}_k$ is bounded in $L^\infty$, using a truncation argument, we can moreover assume that $\Vert u_k^\lambda \Vert_{L^\infty} \leq m\mu$ for every $k \geq 1$ and every $\lambda>0$.
		
		Observe that for every $k \geq 1$ and every $\lambda>0$ large enough, using Chebychev inequality, \eqref{2} and \eqref{30}, we get
		$$
		\mathcal{L}^d(\{ M(|\nabla v_k|) > \lambda  \}) \leq \frac{C}{\lambda^\gamma} \int_{B_1} {\psi}_k^-(M(|\nabla v_k|)(x)) \, dx \leq \frac{C}{\lambda^\gamma},
		$$
		where in the first inequality we have also used that $\psi$ satisfies \hyperlink{A0}{\textnormal{(A0)}} and \hyperlink{idg}{$\textnormal{(Inc)}_\gamma$} on $B_1$ with $\gamma>1$. Hence, keeping in mind the fact that $|D^s v_k|(B_1) \to 0$ as $k \to +\infty$, taking $A=E_h$ in \eqref{3}, we obtain
		\begin{equation}\label{4}
			\limsup_{k \to + \infty} \left( {\psi}^-_k(\lambda) \ \mathcal{L}^d(E_k^\lambda \setminus E_h) \right) \leq \Lambda_h\left( \frac{C}{\lambda^\gamma} \right),
		\end{equation}
		for every $\lambda>0$ large enough and every $h \geq 1$.
		
		For every fixed $\lambda >0$ large enough, the sequence $\{u_k^\lambda\}_k$ is equibounded and equicontinuous. Therefore, it converges uniformly in $\overline{B}_1$ as $k \to +\infty$ to a function $u_\lambda \in C(B_1,\R^m)$. Moreover, by the lower semicontinuity under convergence in measure of the map
		$$
		w \mapsto \mathcal{L}^d(\{ x \in B_1 \setminus E_h: \ w(x) \neq 0 \}),
		$$
		for a fixed $h \geq 1$, using \eqref{4} we have that for every $\lambda >0$ large enough
		\begin{equation*}
			\lambda^\gamma\mathcal{L}^d(\{x \in B_1 \setminus E_h: \ u_\lambda(x) \neq u(x) \}) \leq \limsup_{k \to +\infty} \lambda^\gamma\mathcal{L}^d(\{x \in B_1 \setminus E_h: \ u_k^\lambda(x) \neq v_k(x) \}| \leq  \Lambda_h \left( \frac{C}{\lambda^\gamma} \right).
		\end{equation*}
		Hence, if we set $L_\lambda:=\{x \in B_1 : u_\lambda(x) \neq u(x) \}$, we have $\mathcal{L}^d(L_\lambda \setminus E_h) \to 0$ as $\lambda \to +\infty$ for every $h \geq 1$.
		
		We are finally in a position to conclude. Since $\mathcal{L}^d(E_h)\to 0$ as $h \to +\infty$, we need to prove that for every $h \geq 1$
		\begin{equation}\label{4000}
		\liminf_{k \to + \infty} \int_{B_\rho} f_k(x,\nabla v_k) \, dx \geq \mathcal{L}^d(B_1 \setminus E_h)  f(\nabla u).
		\end{equation}
		We now fix $\zeta > \tau \lambda>\sigma$. Let $g \in C^\infty( [0,+\infty); [0,1])$ such that $g(s)=1$ for $s \leq \zeta$ and $g(s)=0$ for $s \geq 2 \zeta$. We define for every $k \geq 1$ the Carathéodory function 
		\begin{equation*}
			f_k^\zeta(x,\xi):=g(|\xi|)f_k(x,\xi)+a(1-g(s))\psi_k^-(|\xi|),
		\end{equation*}
		where $a>0$ is the constant appearing in \eqref{growth of seq}. Notice that in view of \eqref{growth of seq} and property \eqref{.} of $\psi$ in $0$, there exists $k_0 \geq 1$ such that for every $k \geq k_0$ it holds
		\begin{equation}\label{10000}
			a \psi^-_k(|\xi|) \leq f^\zeta_k(x,\xi) \leq 2\mathfrak{C}b(1+\psi_k^-(|\xi|)),\ \ \ (x,\xi) \in B_1 \times \R^{m \times d},
		\end{equation}
		where $\mathfrak{C}$ does not depend on $\zeta$ nor on $k$.
		Using \eqref{10} and \eqref{10000}, for every $k \geq k_0$ we have
		\begin{align}\label{5000}
			\begin{split}
			\int_{B_1} f^\zeta_k(x,\nabla v_k) \, dx & \geq \int_{B_1 \setminus (E_h \cup E_k^\lambda)} f_k^\zeta(x,\nabla v_k) \, dx \\
			& = \int_{B_1 \setminus (E_h \cup E_k^\lambda)} f_k^\zeta(x,\nabla u_k^\lambda) \, dx \\
			& \geq \int_{B_1 \setminus E_h} f_k^\zeta(x,\nabla u_k^\lambda) \, dx-\int_{E_k^\lambda \setminus E_h} 2\mathfrak{C}b (1+{\psi}_k^-(\lambda)) \, dx.
			\end{split}
		\end{align}
		By \eqref{5000}, the fact that $\zeta>\tau \lambda$, the locally uniform convergence of $f_k$ to $f$ for $\mathcal{L}^d$-a.e. $x \in B_1$ and the quasiconvexity of $f$, we deduce
		\begin{align*}
			\begin{split}
			\liminf_{k \to + \infty} \int_{B_1} f_k(x,\nabla v_k) \, dx & \geq \liminf_{k \to + \infty} \int_{B_1} f_k^\zeta(x,\nabla v_k) \, dx \\
			& \geq \liminf_{k \to + \infty} \int_{B_1 \setminus E_h} f(\nabla u_k^\lambda) +\left(f_k^\zeta(x,\nabla u_k^\lambda)-f(\nabla u_k^\lambda)\right) \, dx \\ & \ \ \ \ \ -\int_{E_k^\lambda \setminus E_h} 2\mathfrak{C} b (1+{\psi}_k^-(\lambda)) \, dx \\ 
			& \geq  \liminf_{k \to + \infty} \int_{B_1 \setminus E_h} f(\nabla u_k^\lambda) \, dx - 4\mathfrak{C}b \int_{E_k^\lambda \setminus E_h} {\psi}_k^-(\lambda) \, dx, \\
			& \geq \int_{B_1 \setminus E_h} f(\nabla u_\lambda) \, dx - 4\mathfrak{C}b \limsup_{k \to + \infty} \left( {\psi}_k^-(\lambda) \ \mathcal{L}^d(E_k^\lambda \setminus E_h) \right).
			\end{split}
		\end{align*}
		Therefore, using \eqref{4}, we have that for every $\lambda>0$ large enough it holds
		\begin{equation}\label{7000}
		\liminf_{k \to + \infty} \int_{B_1} f_k(x,\nabla v_k) \, dx \geq \int_{B_1 \setminus E_h} f(\nabla u_\lambda) \, dx - 4\mathfrak{C}b \Lambda_h \left( \frac{C}{\lambda^\gamma} \right).
		\end{equation}
		For the first term in the right hand side of \eqref{7000} we estimate
		\begin{equation}\label{8000}
		\int_{B_1 \setminus E_h} f(\nabla u_\lambda) \, dx \geq \int_{B_1 \setminus (E_h \cup L_\lambda)} f(\nabla u) \, dx \geq (\mathcal{L}^d(B_1 \setminus E_h)-\mathcal{L}^d(L_\lambda \setminus E_h))f(\nabla u).
		\end{equation}
		Thus, for every $h \geq 1$ fixed, sending $\lambda \to +\infty$ from \eqref{7000} and \eqref{8000} we conclude \eqref{4000}.
	\end{proof}
	
	We now prove the lower semicontinuity for the bulk energy.
	
	\begin{prop}\label{lsc bulk}
		Let $f \colon \Omega \times \R^{m \times d} \to [0,+\infty)$ be a Borel measurable function satisfying \eqref{f1}--\eqref{f2} with $a,b>0$ and such that $z \mapsto f(x,z)$ is quasiconvex in $\R^{m \times d}$ for every $x \in \Omega$. Let $\psi \in \Phi_s(\Omega)$ satisfy \hyperlink{A0}{$\textnormal{(A0)}$}, \hyperlink{idg}{$\textnormal{(Inc)}$}, \hyperlink{idg}{$\textnormal{(Dec)}$} and property \eqref{.} for $\mathcal{L}^d$-a.e $x_0 \in \Omega$.
		Given $A \in \mathcal{A}(\Omega)$,
		$$
		\liminf_{k \to +\infty} \int_{A} f(x,\nabla u_k) \, dx \geq \int_{A} f(x,\nabla u) \, dx
		$$
		for every sequence $\{ u_k \}_k \subset \gsbv^{\psi}(A,\R^m)$ converging to a function $u \in \gsbv^{\psi}(A,\R^m)$ in measure and such that $\sup_k \mathcal{H}^{d-1}(J_{u_k} \cap A) < +\infty$.
	\end{prop}
	
	\begin{proof}
		The proof is standard and we only sketch it. Possibly extracting a subsequence we can assume that $u_k$ converges to $u$ $\mathcal{L}^d$-a.e. in $A$ and that $\mathcal{H}^{d-1}(J_{u_k})$ and $f(x,\nabla u_k(x))\mathcal{L}^d$ weakly* converge in $A$ to Radon measures $\mu$ and $\lambda$ respectively. By Besicovitch derivation theorem, in order to prove the statement, we need to show that for a.e. $x_0 \in A$ it holds
		\begin{equation}\label{40}
			\limsup_{\varepsilon \searrow 0} \frac{\lambda(B_\varepsilon(x_0))}{\varepsilon^d} \geq \omega_d f(x_0,\nabla u(x_0)).
		\end{equation}  
		Take $x_0 \in \Omega$ such that property \eqref{.} holds and let $\{\varepsilon_k\}_k$ be such that $\varepsilon_k \searrow 0$ and $\lambda(\partial B_{\varepsilon_k}(x_0))=0$ for every $k \geq 1$. Take $\varepsilon_1$ suitably small so that $B_{\varepsilon_1}(x_0) \Subset A$ and define $\psi_1: B_1 \times [0,+\infty) \to [0,+\infty)$ as
		$$
		\psi_1(x,t):=\psi\left({\varepsilon_1}x+x_0,t\right)
		$$ 
		for every $x \in B_1$ and every $t \geq 0$. By Remark \ref{RESCALE} we have that $\psi_1 \in \Phi_s(B_1)$ satisfies the same properties of $\psi$ with the same constants in $B_1$ and property \eqref{.} in $0 \in B_1$.
		
		Define $f_k(x,\xi):=f(x_0+\varepsilon_k x,\xi)$ for any $x \in B_1$. We have that
		\begin{equation*}
			\lim_{k \to +\infty} f_k(x,\xi)=f(x_0,\xi)
		\end{equation*}
		locally uniformly in $\R^{m \times d}$ for $\mathcal{L}^d$-a.e. $x \in B_1$. We can also find a sequence $\{ w_k \}_k \subset \gsbv(B_1,\R^m)$ such that $w_k$ converges in measure to the map $y \mapsto \nabla u(x_0) y$ in $B_1$, $\mathcal{H}^{d-1}(J_{w_k}) \to 0$ as $k \to +\infty$ and 
		\begin{equation}\label{50}
			\int_{B_1} f_k(x,\nabla w_k) \, dx \leq \frac{\lambda(B_{\varepsilon_k}(x_0))}{\varepsilon_k^d}+\varepsilon_k.
		\end{equation} 
		Since $a\psi(x_0+\varepsilon_k x,|\xi|) \leq f(x_0+\varepsilon_k x,\xi) \leq b(1+\psi(x_0+\varepsilon_k x,|\xi|))$ by \eqref{f2}, using Proposition \ref{100} with $\psi_1$, $f_k$ and $w_k$, we deduce that
		$$
		\liminf_{k \to + \infty} \int_{B_1} f_k(x,\nabla w_k) \, dx \geq \omega_d  f(x_0,\nabla u(x_0)).
		$$
		This together with \eqref{50} gives \eqref{40}.
	\end{proof}
	
	We conclude with the surface part. The following result is contained in \cite[Theorem 3.3]{AmbrosioExistence}.
	\begin{thm}\label{lsc surf}
		Let $g \colon \Omega \times \R^m_0 \times \Sf^{d-1} \to [0, +\infty)$ satisfy \eqref{g1}--\eqref{g2} and such that $(\zeta,\nu) \to g(x,\zeta,\nu)$ is BV-elliptic for every $x \in \Omega$. Let $A \in \mathcal{A}(\Omega)$. Then, for every $u \in \gsbvp(A,\R^m)$ and any sequence $u_k \in \gsbvp(A,\R^m)$ converging in measure to $u$ in $A$ and such that
		\begin{equation}\label{equiint cond}
			\sup_k  \int_A |\nabla u_k|^p \, dx<+\infty
		\end{equation}
		for some $p>1$, the following inequality holds
		\begin{equation*}
			\int_{J_{u } \cap A} g(x,[u],\nu_u) \, d\mathcal{H}^{d-1} \leq \liminf_{k \to + \infty} 	\int_{J_{u_k }\cap A} g(x,[u_k],\nu_{u_k}) \, d\mathcal{H}^{d-1}.
		\end{equation*}
	\end{thm}
	
	Using Proposition \ref{lsc bulk} and Theorem \ref{lsc surf} we are now ready to conclude the proof of Theorem \ref{lsc of the functional}.
	
	\begin{proof}[Proof of Theorem \ref{lsc of the functional}]
		Thanks to the assumptions on $\psi$, can apply Proposition \ref{lsc bulk} and Theorem \ref{lsc surf}
		once we notice that $\sup_k \int_A \psi(x,|\nabla u_k|) \, dx<+\infty$ and that $\psi$ satisfies \hyperlink{idg}{$\textnormal{(Inc)}_\gamma$} with $\gamma>1$. Therefore, condition \eqref{equiint cond} holds.
	\end{proof}

	\section{Relaxation}
	\label{s:relax}
	
	Using the integral representation and the lower semicontinuity result for quasiconvex and BV-elliptic integrands, we can now prove the relaxation Theorem \ref{relax}.
	
	\begin{proof}[Proof of Theorem \ref{relax}]
		Using Theorem \ref{lsc of the functional} and taking the infimum over all the sequences, for every function $u \in \gsbv^{\psi}(A,\R^m)$ and $A \in \mathcal{A}(\Omega)$ we get
		\begin{equation}
			\int_A \mathcal{Q}f(x,\nabla u(x)) \, dx +\int_{J_u \cap A} \mathcal{R}g(x,[u](x),\nu_u(x)) \, d\mathcal{H}^{d-1} \leq \overline{\mathcal{G}}(u,A).
		\end{equation}
		We now want to show that $\overline{\mathcal{G}}$ satisfies \eqref{H1}--\eqref{H5}. By definition of $\overline{\mathcal{G}}$ we deduce immediately that it satisfies \eqref{H2}, \eqref{H3} and \eqref{H5}. The upper bound of \eqref{H4} is obvious since we have by definition $\overline{\mathcal{G}}(u,A) \leq \mathcal{G}(u,A)$ for every $u \in \gsbv^\psi(\Omega,\R^m)$ and $A \in \mathcal{A}(\Omega)$. On the other hand, the lower bound of \eqref{H4} for $\overline{\mathcal{G}}$ can be deduced observing that by Ioffe's Theorem the functional
		$$
		u \mapsto \int_A \psi(x,|\nabla u(x)|) \, dx + \mathcal{H}^{d-1}(J_u \cap A)
		$$
		is lower semicontinuous with respect to the convergence in measure on $A$. Finally, the proof of \eqref{H1} for $\overline{\mathcal{G}}$ is standard and builds upon the fundamental estimate (see e.g. \cite{BFLMrelaxSBVp} proof of Theorem 10), we omit it.
		
		Hence, by Corollary \ref{trans invar rep} there exist two Borel functions $\overline{f} \colon \Omega \times \R^{m \times d} \to [0,+\infty)$ and $\overline{g}\colon \Omega \times  \R^m_0 \times \Sf^{d-1} \to [0,+\infty)$ such that, for every $u \in \gsbv^{\psi}(A,\R^m)$ and $A \in \mathcal{A}(\Omega)$, 
		\begin{equation*}
			\overline{\mathcal{G}}(u,A)=\int_A \overline{f}(x,\nabla u(x)) \, dx +\int_{J_u \cap A} \overline{g}(x,[u](x),\nu_u(x)) \, d\mathcal{H}^{d-1}.
		\end{equation*}
		We now proceed as in \cite[Proof of Theorem 4]{BFLMrelaxSBVp} and show that
		\begin{equation}\label{101}
			\overline{f}(x,\xi) \leq \mathcal{Q}f(x,\xi) \ \ \ \mbox{for $\mathcal{L}^d$ -a.e. $x \in \Omega$ and every $\xi \in \R^{m \times d}$}.
		\end{equation}
		Let $\eta>0$, since $f$ is a Carathéodory function, by the Scorza-Dragoni theorem there exists a compact set $K \subset \Omega$ with $\mathcal{L}^d(\Omega \setminus K) \leq \eta$, such that the function
		$$
		f \colon K \times \R^{m \times d} \to [0,+\infty)
		$$
		is continuous. Let $K^1$ be the set of points with density one for $K$. Fix $(x,\xi) \in K^1 \times \R^{m \times d}$ and let $\varphi \in C^\infty_0(B,\R^m)$ be such that
		\begin{equation}\label{102}
			\mathcal{Q}f(x,\xi) +\eta \geq \int_B f(x,\xi+\nabla \varphi(y)) \, dy.
		\end{equation}
		For any $\varepsilon >0$ such that $B_\varepsilon(x) \subset \Omega$, set $v_\varepsilon(y):=\varepsilon \varphi((y-x)/\varepsilon)$. Then, $v_\varepsilon \in W^{1,\infty}_0(B_\varepsilon(x),\R^m)$ and, by definition of $\mathbf{m}_\mathcal{G}$ and $\overline{f}$ (see \eqref{m} and \eqref{f}),
		\begin{align}\label{103}
			\overline{f}(x,\xi) \leq \limsup_{\varepsilon \searrow 0} \frac{\mathbf{m}_\mathcal{G}(\xi(\cdot-x)+v_\varepsilon,B_\varepsilon(x))}{\varepsilon^d} \leq \limsup_{\varepsilon \searrow 0} \frac{1}{\varepsilon^d} \int_{B_\varepsilon(x)} f(y,\xi+\nabla v_\varepsilon(y)) \, dy.
		\end{align}
		Using the modulus of uniform continuity for $f$ we get that for $\varepsilon$ small enough
		\begin{equation}\label{104}
			|f(x,\xi+\nabla v_\varepsilon(y))-f(y,\xi+\nabla v_\varepsilon(y))| \leq \eta
		\end{equation}
		for every $y \in B_\varepsilon(x)$. Thus, by \eqref{102}, \eqref{103}, \eqref{104} and recalling properties \hyperlink{A0}{(A0)} and \hyperlink{idg}{(Dec)} of $\psi$ together with \eqref{f2},
		\begin{align*}
			\overline{f}(x,\xi) & \leq \limsup_{\varepsilon \searrow 0} \left( \frac{1}{\varepsilon^d} \int_{B_\varepsilon(x) \cap K} \! f(y,\xi+\nabla v_\varepsilon(y)) \, dy+ \frac{1}{\varepsilon^d} \int_{B_\varepsilon(x) \setminus K} \! b(1+\psi(x,|\xi|+\Vert \nabla \varphi \Vert_{L^\infty})) \, dy \right) \\
			& \leq \limsup_{\varepsilon \searrow 0} \frac{1}{\varepsilon^d} \int_{B_\varepsilon(x) \cap K} f(x,\xi+\nabla v_\varepsilon(y)) \, dy+\eta \\
			& \leq \int_{B} f(x,\xi+\nabla \varphi(y)) \, dy  +\eta \leq \mathcal{Q}f(x,\xi) +2\eta,
		\end{align*}
		where we have used also that $x \in K^1$. Letting $\eta \to 0$, gives \eqref{101}.
		
		We now deal the surface part, showing that, for $\mathcal{H}^{d-1}$-a.e. $x \in \Omega$ and every $(\zeta,\nu) \in \R^m_0 \times \Sf^{d-1}$,
		\begin{equation}\label{99}
			\overline{g}(x,\zeta,\nu) \leq \mathcal{R}g(x,\zeta,\nu).
		\end{equation}
		Take $\eta>0$ and fix $(x,\zeta,\nu) \in \Omega \times \R^m_0 \times \Sf^{d-1}$. Since $g$ satisfies \eqref{g1}--\eqref{g4}, we can find a function $w \in \sbv^{\psi}(Q_\nu,\R^m) \cap L^\infty(B,\R^m)$ such that $\nabla w=0$ a.e. in $B$, $w=u_{x,\zeta,0,\nu}$ in a neighborhood of $\partial B$ and
		\begin{equation}\label{105}
			\mathcal{R}g(x,\zeta,\nu) +\eta \geq \int_{B} g(x,[w](y),\nu_w(y)) \, d\mathcal{H}^{d-1}(y).
		\end{equation}
		We recall that
		\begin{equation*}
			u_{x,\zeta,0,\nu}(y)=
			\begin{cases*}
				\zeta & if $(y-x)\cdot \nu >0$, \\
				0 & if $(y-x)\cdot \nu \leq 0$.
			\end{cases*}
		\end{equation*}
		For any $\varepsilon >0$ such that $B_\varepsilon(x) \subset \Omega$, set $v_\varepsilon(y):= w((y-x)/\varepsilon)$. Then, by definition of $\mathbf{m}_\mathcal{G}$ and $\overline{g}$ (see \eqref{m} and \eqref{g}),
		\begin{align}\label{106}
			\overline{g}(x,\zeta,\nu)\leq \limsup_{\varepsilon \searrow 0} \frac{\mathbf{m}_\mathcal{G}(v_\varepsilon,B_\varepsilon(x))}{\varepsilon^{d-1}} \leq \limsup_{\varepsilon \searrow 0} \frac{1}{\varepsilon^{d-1}} \int_{B_\varepsilon(x) \cap J_{v_\varepsilon}} g(y,[v_\varepsilon](y),\nu_{v_\varepsilon}(y)) \, d\mathcal{H}^{d-1}(y).
		\end{align}
		Using the modulus of uniform continuity of $g$ we get that for $\varepsilon$ small enough, for every $y \in B_\varepsilon(x)$ it holds
		\begin{equation}\label{107}
			|g(x,[v_\varepsilon](y),\nu_{v_\varepsilon}(y))-g(y,[v_\varepsilon](y),\nu_{v_\varepsilon}(y))| \leq \eta.
		\end{equation} 
		Combining \eqref{105}, \eqref{106} and \eqref{107} we get
		\begin{align*}
			\overline{g}(x,\zeta,\nu) & \leq \limsup_{\varepsilon \searrow 0} \frac{1}{\varepsilon^{d-1}} \int_{B_\varepsilon(x) \cap J_{v_\varepsilon}} g(y,[v_\varepsilon](y),\nu_{v_\varepsilon}(y)) \, d\mathcal{H}^{d-1}(y) \\
			& \leq \limsup_{\varepsilon \searrow 0} \frac{1}{\varepsilon^{d-1}} \int_{B_\varepsilon(x) \cap J_{v_\varepsilon}} g(x,[v_\varepsilon](y),\nu_{v_\varepsilon}(y)) \, d\mathcal{H}^{d-1}(y)+\eta \\
			& = \int_{B \cap J_w} g(x,[w](y),\nu_{w}(y)) \, d\mathcal{H}^{d-1}(y)+\eta \leq \mathcal{R}g(x,\zeta,\nu) +2\eta.
		\end{align*}
		Letting $\eta \to 0$ gives \eqref{99}.
	\end{proof}

\section*{Acknowledgments}

\noindent The work of S. Almi was funded by the FWF Austrian Science Fund through the Projects ESP-61 and P35359-N.

\noindent The work of F. Solombrino is part of the MIUR - PRIN 2017, project Variational Methods for Stationary and Evolution Problems with Singularities and Interfaces, No. 2017BTM7SN.  He also acknowledges support by project Starplus 2020 Unina Linea 1 "New challenges in the variational modeling of continuum mechanics" from the University of Naples Federico II and Compagnia di San Paolo, and by Gruppo Nazionale per l’Analisi Matematica, la Probabilit\`a e le loro Applicazioni (GNAMPA-INdAM).

	\bibliographystyle{siam}
	\bibliography{bibliography_NEW}

\end{document}